\documentclass{scrartcl}

\usepackage[utf8]{luainputenc}
\usepackage[USenglish]{babel}
\usepackage{csquotes}

\usepackage[a4paper,top=27mm,bottom=20mm,inner=25mm,outer=20mm]{geometry}

\usepackage[%
  backend=bibtex,bibencoding=ascii,
  style=numeric-comp,
  giveninits=true, uniquename=init, 
  natbib=true,
  url=true,
  doi=true,
  isbn=false,
  backref=false,
  maxnames=99,
  ]{biblatex}
\usepackage{amsmath}
\allowdisplaybreaks
\numberwithin{equation}{section}
\usepackage{amssymb}
\usepackage{commath}
\usepackage{mathtools}
\usepackage{bbm}
\usepackage{nicefrac}
\usepackage{subdepth}

\usepackage{siunitx}

\usepackage{amsthm}
\usepackage{thmtools}
\usepackage{etoolbox}
\makeatletter
\patchcmd{\thmt@setheadstyle}
 {\bgroup\thmt@space}
 {\thmt@space}
 {}{}
\patchcmd{\thmt@setheadstyle}
 {\egroup\fi}
 {\fi}
 {}{}
\makeatother
\declaretheoremstyle[
  bodyfont=\normalfont\itshape,
  headformat=\NAME\ \NUMBER\NOTE,
]{myplain}
\declaretheoremstyle[
  headformat=\NAME\ \NUMBER\NOTE,
]{mydefinition}
\newcommand{\envqed}{{\lower-0.3ex\hbox{$\triangleleft$}}}
\declaretheorem[style=myplain,numberwithin=section]{theorem}
\declaretheorem[style=myplain,numberlike=theorem]{lemma}

\declaretheorem[style=myplain,numberlike=theorem]{corollary}
\declaretheorem[style=mydefinition,numberlike=theorem,qed=\envqed]{definition}
\declaretheorem[style=mydefinition,numberlike=theorem,qed=\envqed]{remark}
\declaretheorem[style=mydefinition,numberlike=theorem,qed=\envqed]{example}

\usepackage[plainpages=false,pdfpagelabels,hidelinks,unicode]{hyperref}

\usepackage{color}
\usepackage{graphicx}
\usepackage[small]{caption}
\usepackage{subcaption}

\begingroup\expandafter\expandafter\expandafter\endgroup
\expandafter\ifx\csname pdfsuppresswarningpagegroup\endcsname\relax
\else
\fi

\usepackage{booktabs}
\usepackage{rotating}
\usepackage{multirow}

\usepackage{enumitem}

\usepackage{ifluatex}
\ifluatex
  \usepackage[no-math]{fontspec}
\else
  \usepackage[T1]{fontenc}
\fi
\usepackage{newpxtext,newpxmath}

\usepackage{xparse}

\let\epsilon\varepsilon
\let\phi\varphi
\let\rho\varrho

\usepackage{xspace}
\newcommand{\cf}[0]{{cf.\@}\xspace}
\newcommand{\eg}[0]{{e.g.\@}\xspace}
\newcommand{\ie}[0]{{i.e.\@}\xspace}

\renewcommand{\O}{\mathcal{O}}
\newcommand{\e}{\mathrm{e}}

\newcommand{\N}{\mathbb{N}}

\newcommand{\R}{\mathbb{R}}

\newcommand{\scp}[2]{\left\langle{#1,\, #2}\right\rangle}
\newcommand{\I}{\operatorname{I}}

\DeclarePairedDelimiterX\newset[1]\lbrace\rbrace{\setaux #1||\endsetaux}
\def\setaux#1|#2|#3\endsetaux{\if\relax\detokenize{#2}\relax #1 \else #1 \;\delimsize\vert\; #2 \fi}
\renewcommand{\set}[1]{\newset*{#1}}

\newcommand{\dt}{\Delta t}

\renewcommand{\vec}[1]{\pmb{#1}}
\newcommand{\mat}[1]{\vec{#1}}
\NewDocumentCommand{\D}{m+g}{%
  \IfNoValueTF{#2}
    {D_{#1}}
    {D_{#1,#2}}%
}
\newcommand{\diag}{\operatorname{diag}}
\NewDocumentCommand{\A}{m+g}{%
  \IfNoValueTF{#2}
    {A_{#1}}
    {A_{#1,#2}}%
}
\NewDocumentCommand{\M}{g}{%
  \IfNoValueTF{#1}
    {M}
    {M_{#1}}%
}
\NewDocumentCommand{\proj}{g}{%
  \IfNoValueTF{#1}
    {P}
    {P_{#1}}%
}
\NewDocumentCommand{\eL}{g}{%
  \IfNoValueTF{#1}
    {\vec{e}_{L}}
    {\vec{e}_{L,#1}}%
}
\NewDocumentCommand{\eR}{g}{%
  \IfNoValueTF{#1}
    {\vec{e}_{R}}
    {\vec{e}_{R,#1}}%
}
\NewDocumentCommand{\dL}{g}{%
  \IfNoValueTF{#1}
    {\vec{d}_{L}}
    {\vec{d}_{L,#1}}%
}
\NewDocumentCommand{\dR}{g}{%
  \IfNoValueTF{#1}
    {\vec{d}_{R}}
    {\vec{d}_{R,#1}}%
}
\newcommand{\fnum}{f^\mathrm{num}}

\newcommand{\dx}{\Delta x}
\newcommand{\xmin}{x_\mathrm{min}}
\newcommand{\xmax}{x_\mathrm{max}}

\newcommand{\orcid}[1]{ORCID:~\href{https://orcid.org/#1}{#1}}
\usepackage{authblk}

\newenvironment{keywords}{\par\textbf{Key words.}}{\par}
\newenvironment{AMS}{\par\textbf{AMS subject classification.}}{\par}

\title{A Broad Class of Conservative Numerical Methods for Dispersive Wave Equations}
\author[1]{Hendrik Ranocha\thanks{\orcid{0000-0002-3456-2277}}}
\author[2]{Dimitrios Mitsotakis\thanks{\orcid{0000-0003-2700-6093}}}
\author[1]{David I. Ketcheson\thanks{\orcid{0000-0002-1212-126X}}}
\affil[1]{%
King Abdullah University of Science and Technology (KAUST),
Computer Electrical and Mathematical Science and Engineering Division (CEMSE),
Thuwal, 23955-6900, Saudi Arabia}
\affil[2]{%
School of Mathematics and Statistics, Victoria University of Wellington,
Wellington 6140, New Zealand}
\date{September 14, 2020} 

\makeatletter
\hypersetup{pdfauthor={Hendrik Ranocha, Dimitrios Mitsotakis, David I. Ketcheson}}
\hypersetup{pdftitle={\@title}}
\makeatother

\begin{document}

\maketitle

\begin{abstract}
  We develop a general framework for designing conservative numerical methods
based on summation by parts operators and split forms in space, combined
with relaxation Runge-Kutta methods in time.  We apply this framework to create
new classes of fully-discrete conservative methods for several nonlinear
dispersive wave equations: Benjamin-Bona-Mahony (BBM), Fornberg-Whitham,
Camassa-Holm, Degasperis-Procesi, Holm-Hone, and the BBM-BBM system. These full
discretizations conserve all linear invariants and one nonlinear invariant for
each system. The spatial semidiscretizations include finite difference,
spectral collocation, and both discontinuous and continuous finite element
methods. The time discretization is essentially explicit, using relaxation
Runge-Kutta methods.  We implement some specific schemes from among the
derived classes, and demonstrate their favorable properties through numerical
tests.

\end{abstract}



\section{Introduction}
\label{sec:introduction}

In this work we study and develop numerical discretizations for nonlinear
dispersive wave equations.  One of the most important features of such equations
is the existence of nonlinear invariants.  In addition to the total mass,
many dispersive wave models possess other invariants that may represent
the energy or another important physical quantity.
Perhaps the most interesting feature of these systems,
related to the presence of conserved quantities, is the existence of solitary
wave solutions.  For non-integrable systems, numerical methods are an essential
tool for studying solitary waves; even for integrable systems, numerical methods
are very useful for exploring solution behavior \cite{antonopoulos2019error}.  Both analysis and numerical
experiments have demonstrated that such studies are best undertaken
using numerical methods that exactly preserve the invariants of
the system in question \cite{araujo2001error,duran2003conservative}.
Specifically, conservative methods possess discrete solitary wave solutions that
accurately approximate the true solitary waves, with an amplitude that
is constant in time and a phase error that grows linearly in time
\cite{frutos1997accuracy}.
In contrast, non-conservative methods typically yield discrete solutions
with amplitude errors that grow linearly in time and (therefore) phase errors
that grow quadratically in time.  Conservative methods are thus especially
desirable for conducting studies of solitary wave properties such as
speed-amplitude relationships and solitary wave interactions
\cite{duran2003conservative}, and for long-time simulations.
At the same time, discrete conservation properties can be useful
for proving numerical stability.

Significant work has been devoted to the development of conservative
methods for certain nonlinear dispersive wave equations
\cite{sanzserna1982explicit,sanzserna1983method,duran2000numerical,
duran2003conservative,cai2016geometric,yang2018new,zhang2020dissipative,buli2018local,bona2018finite,winther1980conservative,duran2019multi,duran2019multi2,zhu2011multi}.  Nevertheless, and
despite their known advantages, conservative fully-discrete schemes
are not widely available for many important dispersive nonlinear wave
equations, and most methods being proposed and used are non-conservative;
see e.g. \cite{bona1985numerical,wei_kirby_grilli_subramanya_1995,
pelloni2001numerical,dutykh2013finite,dutykh2013serre,antonopoulos2019error,bona1998boussinesq,eilbeck1975numerical,antonopoulos2010galerkin,walkley1999finite,winther1982finite}.
Indeed, the development of accurate and stable schemes (even without
nonlinear invariant conservation) is a challenging task and often
requires the application of implicit time discretizations \cite{bona1995conservative,yan2002local,dutykh2013finite}.

Usually, proving the conservation of invariants of dispersive partial
differential equations (PDEs) at the continuous
level requires application of the product/chain rule and integration by parts.
To mimic this procedure at the semidiscrete level (discrete in space, continuous
in time), summation by parts (SBP) operators are used, which provide a discrete
analogue of integration by parts. A review of the relevant theory can be found
in \cite{svard2014review, fernandez2014review, chen2020review}.
Nowadays, many different schemes have been formulated in the SBP framework,
\eg finite difference \cite{strand1994summation},
finite volume \cite{nordstrom2001finite, nordstrom2003finite},
discontinuous Galerkin \cite{gassner2013skew},
and flux reconstruction methods \cite{ranocha2016summation}.
At internal interfaces or external boundaries, SBP methods can be combined
with a weak imposition of interface/boundary conditions using so-called
simultaneous approximation terms (SATs) to bound the energy
of the semidiscretization \cite{carpenter1994time, carpenter1999stable}.

Since the chain and product rules cannot hold discretely for many high-order
discretizations \cite{ranocha2019mimetic},
split forms that preserve local conservation laws are used; \cf
\cite{fisher2013discretely}. These are related to entropy-conservative
methods in the sense of Tadmor \cite{tadmor1987numerical, lefloch2002fully,
fisher2013high, ranocha2018comparison}. Although the idea to use split forms
is not exactly new \cite[eq. (6.40)]{richtmyer1967difference}, it is still
state of the art and enables the construction of numerical methods with
desirable properties \cite{gassner2016split}. Conservative discretizations based
on classical finite
element methods require the exact integration of nonlinear
terms, which can become very costly or even impossible for non-polynomial
nonlinearities. Conservative methods based on split forms do not require exact
integration, so they can both be cheaper and result in better
stability properties \cite{winters2018comparative}.

All of the previously existing conservative numerical methods such as
\cite{zhang2020dissipative, liu2016invariantFD, liu2016invariantDG,
hong2019linear, xia2014fourier} are constructed
using ad hoc techniques tailored specifically to both the equation and the
numerical method.
Typically, the ideas used to construct one such scheme cannot
immediately be adapted to another equation or type of discretization.
In contrast, we propose a unifying spatial discretization framework based on
SBP operators.  We first establish general technical
results and then apply these to concrete physical models, obtaining a set
of necessary algebraic conditions for conservative semidiscretizations. These
conditions can be satisfied by numerical schemes from any of the classes included in
the unifying SBP framework.

To transfer the semidiscrete conservation results to fully-discrete schemes,
the recent relaxation approach is used \cite{ketcheson2019relaxation,
ranocha2020relaxation, ranocha2020relaxationHamiltonian, ranocha2020general,
ranocha2020fully}. First ideas for such techniques date back to
\cite{sanzserna1982explicit, sanzserna1983method} and
\cite[pp. 265--266]{dekker1984stability} but have been developed widely just
recently.

The numerical methods developed and studied in this article are implemented in
Julia \cite{bezanson2017julia}, using the time integration schemes of
DifferentialEquations.jl \cite{rackauckas2017differentialequations} and
Matplotlib for the plots \cite{hunter2007matplotlib}. The source code for
all numerical methods and the experiments is available online
\cite{ranocha2020broadRepro}.

This article is structured as follows. Firstly, the concept of SBP operators
is recalled in Section~\ref{sec:SBP} and some technical results are provided
that will be applied later to prove the discrete conservation properties.
Afterwards, the relaxation approach in time is briefly summarized in
Section~\ref{sec:relaxation}.
Having established the framework of numerical methods, we concentrate on the
Benjamin-Bona-Mahony (Section~\ref{sec:bbm}),
Fornberg-Whitham (Section~\ref{sec:fw}),
Camassa-Holm (Section~\ref{sec:ch}),
Degasperis-Procesi (Section~\ref{sec:dp}),
and Holm-Hone (Section~\ref{sec:hh}) equations as well as the
BBM-BBM system (Section~\ref{sec:bbm_bbm}).
For each dispersive wave model, conservative numerical methods are constructed
and tested in some numerical experiments. Finally, we summarize the development
and provide an outlook on future research in Section~\ref{sec:summary}.

\section{Summation by parts operators}
\label{sec:SBP}

In this section, periodic and non-periodic SBP operators are introduced
at first in a general way. Afterwards, several examples of classical
schemes are rephrased as SBP schemes. While there are generalizations of
SBP methods \cite{ranocha2016summation, ranocha2017extended,
ranocha2018generalised, chan2019efficient, chan2019skew}, we concentrate
here on nodal collocation schemes where the boundary points are included.
Thus, an interval $[\xmin, \xmax]$ is discretized using a grid\footnote{Here we include the possibility of repeated nodes in order to accommodate DG meshes in the most natural form.}
$\vec{x} = (\vec{x}_1, \dots, \vec{x}_N)^T$, where
$\xmin = \vec{x}_1 \le \vec{x}_2 \le \dots \le \vec{x}_N = \xmax$.
A function $u\colon [\xmin, \xmax] \to \R$ is represented discretely on
the grid $\vec{x}$ by its nodal values $\vec{u} = (\vec{u}_1, \dots, \vec{u}_N)^T$,
where $\vec{u}_i = u(\vec{x}_i)$. Multiplication of discrete grid functions
$\vec{u}, \vec{v}$ is performed pointwise, \ie $(\vec{uv})_i = \vec{u}_i \vec{v}_i$.

\subsection{First-derivative operators}

The idea of summation by parts operators is to mimic integration by parts.
Hence, compatible derivative and integration/quadrature operators are
necessary \cite{kreiss1974finite, strand1994summation}.

\begin{definition}
  Given a grid $\vec{x}$, a \emph{$p$-th order accurate $i$-th
  derivative matrix} $\D{i}$ is a matrix that satisfies
  \begin{equation}
    \forall k \in \{0, \dots, p\}\colon \quad
      \D{i} \vec{x}^k = k (k-1) \dots (k-i+1) \vec{x}^{k-i},
  \end{equation}
  with the convention $\vec{x}^0 = \vec{1}$ and $0 \vec{x}^k = \vec{0}$.
  We say $\D{i}$ is consistent if $p \ge 0$.
\end{definition}

We will make frequent use of the vectors
\begin{equation}
  \eL = (1, 0, \dots, 0)^T, \quad
  \eR = (0, \dots, 0, 1)^T,
\end{equation}
in order to evaluate grid functions at the left or right endpoint, respectively.

\begin{definition}
\label{def:D1-bounded}
  A \emph{first-derivative SBP operator}
  consists of
  a grid $\vec{x}$,
  a consistent first-derivative matrix $\D1$,
  and a symmetric and positive-definite matrix $\M$,
  such that
  \begin{align}
  \label{eq:D1-bounded}
    \M \D1 + \D1^T \M & = \eR \eR^T - \eL \eL^T.
  \end{align}
  We refer to $M$ as a mass matrix or norm matrix\footnote{The term mass matrix
  is common for finite element methods. In the finite difference SBP community,
  the name norm matrix is more common.}.
\end{definition}
First-derivative SBP operators mimic integration by parts via
\begin{equation}
\label{eq:SBP-D1-IBP}
\begin{array}{*3{>{\displaystyle}c}}
  \underbrace{
    \vec{u}^T \M \D1 \vec{v}
    + \vec{u}^T \D1^T \M \vec{v}
  }
  & = &
  \underbrace{
    \vec{u}^T \eR \eR^T \vec{v} - \vec{u}^T \eL \eL^T \vec{v},
  }
  \\
  \rotatebox{90}{$\!\approx\;$}
  &&
  \rotatebox{90}{$\!\!\approx\;$}
  \\
  \overbrace{
    \int_{\xmin}^{\xmax} u \, (\partial_x v)
    + \int_{\xmin}^{\xmax} (\partial_x u) \, v
  }
  & = &
  \overbrace{
    u(\xmax) v(\xmax) - u(\xmin) v(\xmin)
  }.
\end{array}
\end{equation}
Of course, integration by parts requires some smoothness at the continuous
level, \eg absolute continuity of $u, v$. Such minimal smoothness assumptions
are often used for formal a priori estimate. The purpose of SBP operators
is to enable such a priori estimates also at the discrete level.

In the case of periodic boundary conditions (under which $\xmin$ and $\xmax$
are identical), the evaluations at the endpoints
of the domain cancel. Hence, a periodic SBP operator can be defined
as follows.
\begin{definition}
\label{def:D1-periodic}
  A \emph{periodic first-derivative SBP operator}
  consists of
  a grid $\vec{x}$,
  a consistent first-derivative matrix $\D1$,
  and a symmetric and positive-definite matrix $\M$
  such that
  \begin{equation}
  \label{eq:D1-periodic}
    \M \D1 + \D1^T \M = 0.
  \end{equation}
\end{definition}
We will often refer to an operator $\D{i}$ as a (periodic) SBP operator
if the other operators (such as the mass matrix $\M$) are clear from the context.
We always assume derivative operators are consistent, but we will
usually omit this term.

In periodic domains, first-derivative SBP operators are associated with
skew-symmetric differentiation matrices. Hence, they are usually
energy-conservative for linear hyperbolic problems. To allow for
energy-dissipative SBP methods, upwind operators can be used, \cf
\cite{mattsson2017diagonal, mattsson2018compatible}.
\begin{definition}
\label{def:D1-upwind-bounded}
  A \emph{first-derivative upwind SBP operator}
  consists of
  a grid $\vec{x}$,
  consistent first-derivative matrices $\D1{\pm}$,
  and a symmetric and positive-definite matrix $\M$,
  such that
  \begin{equation}
  \label{eq:D1-upwind-bounded}
    \M \D1{+} + \D1{-}^T \M = \eR \eR^T - \eL \eL^T,
    \quad
    \frac{1}{2} M (\D1{+} - \D1{-}) \text{ is negative semidefinite}.
  \end{equation}
\end{definition}

In matrix form, $\D1{+}$ is biased toward the upper-triangular part (\ie it has
more nonzero entries in the upper part than in the lower)
and $\D1{-}$ is biased toward the lower-triangular part.
The notion of upwind SBP operators can of course be extended to
periodic domains.
\begin{definition}
\label{def:D1-upwind-periodic}
  A \emph{periodic first-derivative upwind SBP operator}
  consists of
  a grid $\vec{x}$,
  consistent first-derivative matrices $\D1{\pm}$,
  and a symmetric and positive-definite matrix $\M$,
  such that
  \begin{equation}
  \label{eq:D1-upwind-periodic}
    \M \D1{+} + \D1{-}^T \M = 0,
    \quad
    \frac{1}{2} M (\D1{+} - \D1{-}) \text{ is negative semidefinite}.
  \end{equation}
\end{definition}

\begin{remark}
  If $\D1{\pm}$ are upwind SBP operators in a bounded or periodic domain,
  then $\D1 = \frac{1}{2} ( \D1{+} + \D1{-} )$ is a (central) SBP operator.
  Furthermore, we can trivially obtain an upwind SBP operator from any
  (central) SBP operator $\D1$ by taking $\D1{+} = \D1 = \D1{-}$.  Though we
  term it {\em upwind}, this latter operator is of course non-dissipative.  In
  general, upwind SBP operators introduce dissipation if $\D1{+} \neq \D1{-}$.
\end{remark}

\subsection{Second-derivative operators}

Similarly to first-derivative SBP operators, second-derivative operators
can be defined by mimicking integration by parts at the discrete level
\cite{mattsson2004summation, mattsson2008stable}.
\begin{definition}
\label{def:D2-bounded}
  A \emph{second-derivative SBP operator}
  consists of
  a grid $\vec{x}$,
  a consistent second-derivative matrix $\D2$,
  a symmetric and positive-definite matrix $\M$,
  and derivative vectors $\dL$, $\dR$ approximating the evaluation of the first
  derivative at the left/right endpoint as
  $\vec{d}_{L/R}^T \vec{u} \approx u'(x_{\mathrm{min}/\mathrm{max}})$,
  such that
  \begin{equation}
  \label{eq:D2-bounded}
    \M \D2 = -\A2 + \eR \dR^T - \eL \dL^T,
    \quad
    \A2 \text{ is symmetric and positive semidefinite}.
  \end{equation}
  First- and second-derivative SBP operators $\D1, \D2$ are
  said to be \emph{compatible} if they are based on the same mass matrix $\M$ and
  $-\A2 \le -\D1^T \M \D1$ (in the sense of the induced quadratic forms).
\end{definition}

\begin{remark}
  If SBP operators for different derivatives are applied in the
  same context, it will be assumed that they have the same mass
  matrix $\M$.
\end{remark}

Second-derivative SBP operators mimic integration by parts via
\begin{equation}
\label{eq:SBP-D2-IBP}
\begin{array}{*5{>{\displaystyle}c}}
  \underbrace{
    \vec{u}^T \M \D2 \vec{v}
  }
  & = &
  \underbrace{
    -\vec{u}^T \A2 \vec{v}
  }
  & + &
  \underbrace{
    \vec{u}^T \eR \dR^T \vec{v} - \vec{u}^T \eL \dL^T \vec{v},
  }
  \\
  \rotatebox{90}{$\!\approx\;$}
  &&
  \rotatebox{90}{$\!\!\approx\;$}
  &&
  \rotatebox{90}{$\!\!\approx\;$}
  \\
  \overbrace{
    \int_{\xmin}^{\xmax} u \, (\partial_x^2 v)
  }
  & = &
  - \int_{\xmin}^{\xmax} (\partial_x u) (\partial_x v)
  & + &
  \overbrace{
    u(\xmax) \partial_x v(\xmax) - u(\xmin) \partial_x v(\xmin)
  }.
\end{array}
\end{equation}

In periodic domains, the boundary terms vanish again, resulting in
the following
\begin{definition}
\label{def:D2-periodic}
  A \emph{periodic second-derivative SBP operator}
  consists of
  a grid $\vec{x}$,
  a consistent second-derivative matrix $\D2$,
  and a symmetric and positive-definite matrix $\M$
  such that
  \begin{equation}
  \label{eq:D2-periodic}
    \M \D2 = -\A2,
    \quad
    \A2 \text{ is symmetric and positive semidefinite}.
  \end{equation}
\end{definition}

One way to obtain a second-derivative SBP operator is to
square a first-derivative operator: $\D2 = \D1^2$.
In the context of finite difference methods, the resulting
operator is known as a wide-stencil operator.
Compatibility of first- and second-derivative operators means that
for any other choice of $\D2$ we have
$- \vec{u}^T \A2 \vec{u} \le \vec{u}^T \D1^T \M \D1 \vec{u}$, \ie
the wide-stencil operator is the least dissipative (for the heat
equation) of all compatible operators.
In periodic
domains with an even number of nodes, the highest frequency of
grid oscillations is mapped to zero by such wide-stencil operators.
To be able to damp such grid oscillations, it is preferable to
use narrow-stencil operators \cite{mattsson2004summation}.
For first-derivative upwind SBP operators $\D1{\pm}$, both
$\D1{+} \D1{-}$ and $\D1{-} \D1{+}$ are second-derivative SBP
operators.

\subsection{Fourth-derivative operators}

Finally, fourth-derivative SBP operators can be defined as follows
\cite{mattsson2014diagonal}.
\begin{definition}
\label{def:D4-bounded}
  A \emph{fourth-derivative SBP operator}
  consists of
  a grid $\vec{x}$,
  a consistent fourth-derivative matrix $\D4$,
  a symmetric and positive-definite matrix $\M$,
  and derivative vectors $\dL{i}^T$, $\dR{i}^T$ approximating the evaluation
  of the $i$th derivative at the left/right endpoint,
  such that
  \begin{equation}
  \label{eq:D4-bounded}
    \M \D2 = \A4 + \eR \dR{3}^T - \eL \dL{3}^T - \dR{1} \dR{2}^T + \dL{1} \dL{2}^T,
    \quad
    \A4 \text{ is symmetric and positive semidefinite}.
  \end{equation}
\end{definition}

Fourth-derivative SBP operators mimic integration by parts via
\begin{equation}
\label{eq:SBP-D4-IBP}
\begin{array}{*5{>{\displaystyle}c}}
  \underbrace{
    \vec{u}^T \M \D4 \vec{v}
  }
  & = &
  \underbrace{
    \vec{u}^T \A4 \vec{v}
  }
  & + &
  \underbrace{
    \vec{u}^T (\eR \dR{3}^T - \eL \dL{3}^T - \dR{1} \dR{2}^T + \dL{1} \dL{2}^T) \vec{v},
  }
  \\
  \rotatebox{90}{$\!\approx\;$}
  &&
  \rotatebox{90}{$\!\!\approx\;$}
  &&
  \rotatebox{90}{$\!\!\approx\;$}
  \\
  \overbrace{
    \int_{\xmin}^{\xmax} u \, (\partial_x^4 v)
  }
  & = &
  \int_{\xmin}^{\xmax} (\partial_x^2 u) (\partial_x^2 v)
  & + &
  \overbrace{
    \bigl(
      u (\partial_x^3 v) - (\partial_x u) ( \partial_x^2 v) \bigr) \big|_{\xmin}^{\xmax}
  }.
\end{array}
\end{equation}

In periodic domains, the boundary terms vanish again, resulting in
the following
\begin{definition}
\label{def:D4-periodic}
  A \emph{periodic fourth-derivative SBP operator}
  consists of
  a grid $\vec{x}$,
  a consistent fourth-derivative matrix $\D4$,
  and a symmetric and positive-definite matrix $\M$
  such that
  \begin{equation}
  \label{eq:D4-periodic}
    \M \D2 = \A4,
    \quad
    \A4 \text{ is symmetric and positive semidefinite}.
  \end{equation}
\end{definition}

\subsection{Finite difference and collocation methods}
\label{sec:FD-collocation}

Classical central finite difference methods result in periodic SBP
operators with mass matrix
\begin{equation}
  M = \dx \I,
\end{equation}
where $\dx = (\xmax - \xmin) / N$ is the grid spacing.
Similarly, Fourier collocation methods
\cite{kreiss1972comparison, fornberg1975fourier}
and wavelet collocation methods \cite{jameson1993wavelet}
yield periodic SBP operators with the same mass matrix as central finite
difference methods.

The first high-order finite difference SBP operators on bounded domains
were proposed in \cite{strand1994summation} and later developed
in several articles, \eg \cite{mattsson2004summation}. Usually, these operators
are associated with a uniform grid although some optimized versions
employ adapted grid nodes near the boundaries \cite{mattsson2018boundary}.

Another class of SBP methods with diagonal mass matrix are spectral methods
using nodal Lobatto-Legendre bases for polynomials of degree $p$
\cite{carpenter1996spectral, gassner2013skew}.
These schemes result in SBP operators with a diagonal mass matrix
$\M = \diag(\omega_0, \dots, \omega_p)$, where $\omega_i$ are
the Lobatto-Legendre quadrature weights. The associated grid $\vec{x}$ is given
by the Lobatto-Legendre quadrature nodes.

In periodic domains, finite difference schemes are usually applied globally.
If multi-block finite difference or spectral collocation methods shall be used,
they have to be constructed using (non-periodic) SBP operators on each element.
Then, the elements have to be coupled either in a discontinuous way using
SATs as described in Section~\ref{sec:DG} or continuously as described in
Section~\ref{sec:CG}.

\subsection{Nodal discontinuous Galerkin methods}
\label{sec:DG}

Multiple first-derivative SBP operators on bounded domains can be coupled
in a discontinuous finite element way via SATs to construct global SBP
operators, \cf
\cite{chan2018discretely, chan2019efficient, fernandez2019extension}.
The construction and the corresponding proof are reproduced here for
completeness and convenience of the reader. Here and in the following,
we consider only the coupling of two SBP operators on adjacent grids
$\vec{x}_{l/r}$, where $\vec{x}_l = (\vec{x}_{i,l})_{i=1}^{N_l}$ is the
grid on the left- and $\vec{x}_r = (\vec{x}_{i,r})_{i=1}^{N_r}$ is the
grid on the right-hand side. Additionally, these grids have one node location
in common: $\vec{x}_{N_l,l} = \vec{x}_{1,r}$.

\begin{theorem}
\label{thm:D1-DG}
  Consider two first-derivative SBP operators $\D1{l/r}$ on the grids
  $\vec{x}_{l/r}$ with $x_{N_l,l} = x_{1,r}$. Then,
  \begin{equation}
  \label{eq:D1-DG}
    \D1 =
    \begin{pmatrix}
      \D1{l} - \frac{1}{2} \M{l}^{-1} \eR{l} \eR{l}^T &
      \frac{1}{2} \M{l}^{-1} \eR{l} \eL{r}^T \\
      -\frac{1}{2} \M{r}^{-1} \eL{r} \eR{l}^T &
      \D1{r} + \frac{1}{2} \M{r}^{-1} \eL{r} \eL{r}^T
    \end{pmatrix},
    \quad
    \M =
    \begin{pmatrix}
      \M{l} & 0 \\
      0 & \M{r}
    \end{pmatrix},
  \end{equation}
  yields a first-derivative SBP operator on the joint grid $\vec{x} =
  (\vec{x}_{1,l}, \dots, \vec{x}_{N_l,l}, \vec{x}_{1,r}, \dots, \vec{x}_{N_r,r})^T$
  with $N = N_l + N_r$ grid nodes.
  This SBP operator has the same order of accuracy as the less accurate
  operator of $\D1{l/r}$.
\end{theorem}
\begin{proof}
  The SBP property \eqref{eq:D1-bounded} is satisfied since
  \begin{equation}
  \begin{aligned}
    \M \D1 + \D1^T \M
    &=
    \begin{pmatrix}
      \M{l} \D1{l} + \D1{l}^T \M{l} - \eR{l} \eR{l}^T &
      0 \\
      0 &
      \M{r} \D1{r} + \D1{r}^T \M{r} + \eL{r} \eL{r}^T
    \end{pmatrix}
    \\
    &=
    \begin{pmatrix}
      - \eL{l} \eL{l}^T &
      0 \\
      0 &
      \eR{r} \eR{r}^T
    \end{pmatrix}.
  \end{aligned}
  \end{equation}
  The order of accuracy can be checked by applying $\D1$ to a polynomial
  and noting that the interface terms vanish because of continuity of
  polynomials.
\end{proof}

\begin{remark}
  The derivative operator constructed in \eqref{eq:D1-DG} yields
  \begin{equation}
    \M \D1 \begin{pmatrix} \vec{u}_l \\ \vec{u}_r \end{pmatrix}
    =
    \begin{pmatrix}
      \M{l} \D1{l} \vec{u}_l
      + \eR{l} \left(
        \fnum\bigl( \eR{l}^T \vec{u}_l, \eL{r}^T \vec{u}_r \bigr)
        - \eR{l}^T \vec{u}_l
      \right) \\
      \M{r} \D1{r} \vec{u}_r
      - \eL{r} \left(
        \fnum\bigl( \eR{l}^T \vec{u}_l, \eL{r}^T \vec{u}_r \bigr)
        - \eL{r}^T \vec{u}_r
      \right),
    \end{pmatrix}
  \end{equation}
  where
  \begin{equation}
    \fnum(u_-, u_+)
    =
    \frac{u_- + u_+}{2}
  \end{equation}
  is the central numerical flux. This is the strong-form DG discretization
  on two elements using the central numerical flux between them and
  ignoring the other boundaries. For nodal DG methods on Lobatto-Legendre
  nodes, this strong form is equivalent to the prevalent weak form
  because of the SBP property \eqref{eq:D1-bounded},
  \cf \cite{kopriva2010quadrature}, which discretizes (ignoring the
  boundary at $x_{1,l}$)
  \begin{equation}
    - \int_{x_{1,l}}^{x_{N_l,l}} (\partial_x \phi_l) u
    + \phi_l(x_{N_l,l}) \fnum\bigl( u_-(x_{N_l,l}), u_+(x_{N_r,r}) \bigr)
  \end{equation}
  on the left element, where $\phi_l$ is a test function.
\end{remark}

\begin{remark}
  The nodal DG methods used in this article are constructed by coupling
  multiple elements/blocks using the SBP operator with diagonal mass matrix
  determined by Lobatto-Legendre quadrature
  \cite[Chapter~1]{kopriva2009implementing} discontinuously as described in
  Theorem~\ref{thm:D1-DG}. The resulting discretization is the discontinuous
  Galerkin spectral element method \cite{gassner2013skew}.
  For these methods, $\dx$ is the size (length) of one element.
\end{remark}

\begin{corollary}
\label{cor:D1-DG}
  Coupling first-derivative SBP operators discontinuously as described
  in Theorem~\ref{thm:D1-DG} on a periodic domain results in a
  periodic first-derivative SBP operator.
\end{corollary}

The discontinuously coupled first-derivative SBP operator $\D1$ in
\eqref{eq:D1-DG} can be squared to get a second-derivative SBP
operator. This corresponds to the first method of Bassi and Rebay
\cite{bassi1997high}, \cf \cite{arnold2002unified}.
In order to increase the order of accuracy for DG methods for
diffusive problems, the application of alternating upwind fluxes
has been proposed in \cite{cockburn1998local}, resulting in the
local DG (LDG) method \cite{xu2010local}, which is of the form
$\D2 = \D1{+} \D1{-}$ or $\D2 = \D1{-} \D1{+}$ with first-derivative
upwind SBP operators described in
\begin{theorem}
\label{thm:D1-upwind-DG}
  Consider two first-derivative upwind SBP operators $\D1{\pm, l/r}$
  on the grids
  $\vec{x}_{l/r}$ with $\vec{x}_{N_l,l} = \vec{x}_{1,r}$. Then,
  \begin{equation}
  \label{eq:D1-upwind-DG}
  \begin{gathered}
    \D1{+} =
    \begin{pmatrix}
      \D1{+,l} - \M{l}^{-1} \eR{l} \eR{l}^T &
      \M{l}^{-1} \eR{l} \eL{r}^T \\
      0 &
      \D1{+,r}
    \end{pmatrix},
    \quad
    \D1{-} =
    \begin{pmatrix}
      \D1{-,l} &
      0 \\
      - \M{r}^{-1} \eL{r} \eR{l}^T &
      \D1{-,r} + \M{r}^{-1} \eL{r} \eL{r}^T
    \end{pmatrix},
    \\
    \M =
    \begin{pmatrix}
      \M{l} & 0 \\
      0 & \M{r}
    \end{pmatrix},
  \end{gathered}
  \end{equation}
  yield first-derivative upwind SBP operators on the joint grid $\vec{x} =
  (\vec{x}_{1,l}, \dots, \vec{x}_{N_l,l}, \vec{x}_{1,r}, \dots, \vec{x}_{N_r,r})^T$
  with $N = N_l + N_r$ nodes.
  These operators have the same order of accuracy as the less accurate
  of the given operators.
\end{theorem}
\begin{proof}
  The upwind SBP property \eqref{eq:D1-upwind-bounded} can be verified
  by applying it for each operator $\D1{\pm, l/r}$. Moreover,
  \begin{equation}
  \begin{split}
    \begin{pmatrix} \vec{u}_l \\ \vec{u}_r \end{pmatrix}^T
    \M (\D1{+} - \D1{-})
    \begin{pmatrix} \vec{u}_l \\ \vec{u}_r \end{pmatrix}
    =
    \vec{u}_l^T \M{l} (\D1{+,l} - \D1{-,l}) \vec{u}_l
    + \vec{u}_r^T \M{r} (\D1{+,r} - \D1{-,r}) \vec{u}_r
    \\
    - (\eR{l}^T \vec{u}_l)^2
    + 2 (\eR{l}^T \vec{u}_l) (\eL{r} \vec{u}_r)
    - (\eL{r} \vec{u}_r)^2
    \leq 0.
  \end{split}
  \end{equation}
  The order of accuracy can be checked as for Theorem~\ref{thm:D1-DG}.
\end{proof}

\subsection{Nodal continuous Galerkin methods}
\label{sec:CG}

An alternative to the discontinuous coupling of multiple elements is
a continuous coupling of first-derivative operators as in continuous
finite element methods, \cf \cite{hicken2016multidimensional,
hicken2020entropy}.
Continuous Galerkin methods have also been studied from the point of view
of SBP operators in other articles, \eg
\cite{nordstrom2006conservative, abgrall2019analysisI,
abgrall2019analysisII}.
In contrast to the discontinuous coupling which uses the interface
node twice and allows a multivalued solution there, the continuous
coupling uses the interface node only once.

To describe the continuous coupling, indices of matrices will be
denoted by subscripts using a syntax similar to MATLAB and Julia
\cite{bezanson2017julia}, \ie $(\M{l})_{1:N_l-1,1:N_l-1}$ denotes
the upper left block of $\M{l}$ excluding the last column and row.

\begin{theorem}
\label{thm:D1-CG}
  Consider two first-derivative SBP operators $\D1{l/r}$ on the grids
  $\vec{x}_{l/r}$ with $\vec{x}_{N_l,l} = \vec{x}_{1,r}$. Then,
  \begin{equation}
  \label{eq:D1-CG}
  \begin{gathered}
    \D1 =
    \M^{-1}
    \begin{pmatrix}
      (\M{l} \D1{l})_{1:N_l-1, 1:N_l-1} &
      (\M{l} \D1{l})_{1:N_l-1, N_l} &
      0 \\
      (\M{l} \D1{l})_{N_l, 1:N_l-1} &
      (\M{l} \D1{l})_{N_l, N_l} + (\M{r} \D1{r})_{1, 1} &
      (\M{r} \D1{r})_{1, 2:N_r} \\
      0 &
      (\M{r} \D1{r})_{2:N_r, 1} &
      (\M{r} \D1{r})_{2:N_r, 2:N_r}
    \end{pmatrix},
    \\
    \M =
    \begin{pmatrix}
      (\M{-})_{1:N_l-1, 1:N_l-1} &
      (\M{-})_{1:N_l-1, N_l} &
      0 \\
      (\M{-})_{N_l, 1:N_l-1} &
      (\M{-})_{N_l, N_l} + \M{+, 1, 1} &
      (\M{+})_{1, 2:N_r} \\
      0 &
      (\M{+})_{2:N_r, 1} &
      (\M{+})_{2:N_r, 2:N_r}
    \end{pmatrix},
  \end{gathered}
  \end{equation}
  yields a first-derivative SBP operator on the joint grid $\vec{x} =
  (\vec{x}_{1,l}, \dots, \vec{x}_{N_l,l} = \vec{x}_{1,r}, \vec{x}_{2,r}, \dots, \vec{x}_{N_r,r})^T$
  with $N = N_l + N_r - 1$ grid nodes.
  This SBP operator has the same order of accuracy as the less accurate
  operator of $\D1{l/r}$.
\end{theorem}
\begin{proof}
  The new mass matrix $\M$ is obviously symmetric and positive definite.
  Moreover,
  \begin{equation}
    \M \D1 + \D1^T \M
    =
    \begin{pmatrix}
      (\eL{l} \eL{l}^T)_{1:N_l-1, 1:N_l-1} & 0 & 0 \\
      0 & 0 & 0 \\
      0 & 0 & (\eR{r} \eR{r}^T)_{2:N_r, 2:N_r}
    \end{pmatrix}.
  \end{equation}
  Again, the order of accuracy can be checked by applying $\D1$ to a
  polynomial.
\end{proof}


\begin{remark}
  The derivative constructed in \eqref{eq:D1-CG} yields
  a strong-form CG discretization on two elements which
  (ignoring the other boundaries) is equivalent to the usual
  weak-form CG discretization
  \begin{equation}
    - \int_{x_{1,l}}^{x_{N_r,r}} (\partial_x \phi) u,
  \end{equation}
  where $\phi$ is a global test function,
  because of the SBP property \eqref{eq:D1-bounded}.
\end{remark}

\begin{remark}
  The nodal CG methods used in this article are constructed by coupling
  multiple elements/blocks using the SBP operator with diagonal mass matrix
  determined by Lobatto-Legendre quadrature
  \cite[Chapter~1]{kopriva2009implementing} continuously as described in
  Theorem~\ref{thm:D1-CG}.
  For these methods, $\dx$ is the size (length) of one element.
\end{remark}

\begin{example}
  Coupling SBP operators using nodal Lobatto-Legendre bases for polynomials
  of degree $p = 1$ continuously on a uniform mesh results in the classical
  finite difference SBP operator
  \begin{equation}
    \M = \dx
    \begin{psmallmatrix}
      \nicefrac{1}{2} \\
      & 1 \\
      && \ddots \\
      &&& 1 \\
      &&&& \nicefrac{1}{2}
    \end{psmallmatrix},
    \qquad
    \D1 = \frac{1}{\dx}
    \begin{psmallmatrix}
      -1 & 1 \\
      \nicefrac{-1}{2} & 0 & \nicefrac{1}{2} \\
      & \ddots & \ddots & \ddots \\
      && \nicefrac{-1}{2} & 0 & \nicefrac{1}{2} \\
      &&& -1 & 1
    \end{psmallmatrix}.
  \end{equation}
\end{example}

\begin{corollary}
\label{cor:D1-CG}
  Coupling first-derivative SBP operators continuously as described in
  Theorem~\ref{thm:D1-CG} on a periodic domain results in a
  periodic first-derivative SBP operator.
\end{corollary}

Similarly to (central) SBP operators, first-derivative upwind SBP
operators can also be coupled continuously.
\begin{theorem}
\label{thm:D1-upwind-CG}
  Consider two first-derivative upwind SBP operators $\D1{l/r,\pm}$
  on the grids
  $\vec{x}_{l/r}$ with $\vec{x}_{N_l,l} = \vec{x}_{1,r}$. Then,
  \begin{equation}
  \label{eq:D1-upwind-CG}
  \begin{gathered}
    \D1{\pm} =
    \M^{-1}
    \begin{pmatrix}
      (\M{l} \D1{l,\pm})_{1:N_l-1, 1:N_l-1} &
      (\M{l} \D1{l,\pm})_{1:N_l-1, N_l} &
      0 \\
      (\M{l} \D1{l,\pm})_{N_l, 1:N_l-1} &
      (\M{l} \D1{l,\pm})_{N_l, N_l} + (\M{r} \D1{r,\pm})_{1, 1} &
      (\M{r} \D1{r,\pm})_{1, 2:N_r} \\
      0 &
      (\M{r} \D1{r,\pm})_{2:N_r, 1} &
      (\M{r} \D1{r,\pm})_{2:N_r, 2:N_r}
    \end{pmatrix},
    \\
    \M =
    \begin{pmatrix}
      \M{-, 1:N_l-1, 1:N_l-1} &
      \M{-, 1:N_l-1, N_l} &
      0 \\
      \M{-, N_l, 1:N_l-1} &
      \M{-, N_l, N_l} + \M{+, 1, 1} &
      \M{+, 1, 2:N_r} \\
      0 &
      \M{+, 2:N_r, 1} &
      \M{+, 2:N_r, 2:N_r}
    \end{pmatrix},
  \end{gathered}
  \end{equation}
  yields first-derivative upwind SBP operators on the joint grid $\vec{x} =
  (\vec{x}_{1,l}, \dots, \vec{x}_{N_l,l} = \vec{x}_{1,r}, \vec{x}_{2,r}, \dots, \vec{x}_{N_r,r})^T$
  with $N = N_l + N_r - 1$ nodes.
  These operators have the same order of accuracy as the less accurate
  given operators.
\end{theorem}
\begin{proof}
  The mass matrix $\M$ is the same as in Theorem~\ref{thm:D1-CG}
  and hence symmetric and positive definite.
  Moreover,
  \begin{equation}
    \M \D1{+} + \D1{-}^T \M
    =
    \begin{pmatrix}
      (\eL{l} \eL{l}^T)_{1:N_l-1, 1:N_l-1} & 0 & 0 \\
      0 & 0 & 0 \\
      0 & 0 & (\eR{r} \eR{r}^T)_{2:N_r, 2:N_r}
    \end{pmatrix}.
  \end{equation}
  Furthermore, $\M (\D1{+} - \D1{-})$
  is negative semidefinite.
\end{proof}

Second-derivative operators can be coupled analogously.
\begin{theorem}
\label{thm:D2-CG}
  Consider two second-derivative SBP operators $\D2{l/r}$ on the grids
  $\vec{x}_{l/r}$ with $\vec{x}_{N_l,l} = \vec{x}_{1,r}$. Then,
  \begin{equation}
  \label{eq:D2-CG}
  \begin{gathered}
    \D2 =
    \M^{-1}
    \begin{pmatrix}
      (-\A2{l} - \eL{l} \dL{l}^T)_{1:N_l-1, 1:N_l-1} &
      (-\A2{l} - \eL{l} \dL{l}^T)_{1:N_l-1, N_l} &
      0 \\
      (-\A2{l})_{N_l, 1:N_l-1} &
      (-\A2{l})_{N_l, N_l} + (-\A2{r})_{1, 1} &
      (-\A2{r})_{1, 2:N_r} \\
      0 &
      (-\A2{r} + \eR{r} \dR{r}^T)_{2:N_r, 1} &
      (-\A2{r} + \eR{r} \dR{r}^T)_{2:N_r, 2:N_r}
    \end{pmatrix},
    \\
    \M =
    \begin{pmatrix}
      \M{-, 1:N_l-1, 1:N_l-1} &
      \M{-, 1:N_l-1, N_l} &
      0 \\
      \M{-, N_l, 1:N_l-1} &
      \M{-, N_l, N_l} + \M{+, 1, 1} &
      \M{+, 1, 2:N_r} \\
      0 &
      \M{+, 2:N_r, 1} &
      \M{+, 2:N_r, 2:N_r}
    \end{pmatrix},
  \end{gathered}
  \end{equation}
  yields a second-derivative SBP operator on the joint grid $\vec{x} =
  (\vec{x}_{1,l}, \dots, \vec{x}_{N_l,l} = \vec{x}_{1,r}, \vec{x}_{2,r}, \dots, \vec{x}_{N_r,r})^T$
  with $N = N_l + N_r - 1$ grid nodes.
  This SBP operator has the same order of accuracy as the less accurate
  operator of $\D2{l/r}$.
\end{theorem}
\begin{proof}
  The new mass matrix $\M$ is the same as in Theorem~\ref{thm:D1-CG}
  and hence symmetric and positive definite.
  Additionally,
  \begin{equation}
  \begin{aligned}
    \M \D2
    &=
    \begin{pmatrix}
      (-\A2{l} )_{1:N_l-1, 1:N_l-1} &
      (-\A2{l})_{1:N_l-1, N_l} &
      0 \\
      (-\A2{l})_{N_l, 1:N_l-1} &
      (-\A2{l})_{N_l, N_l} + (-\A2{r})_{1, 1} &
      (-\A2{r})_{1, 2:N_r} \\
      0 &
      (-\A2{r})_{2:N_r, 1} &
      (-\A2{r})_{2:N_r, 2:N_r}
    \end{pmatrix}
    \\&\quad
    +
    \begin{pmatrix}
      (-\eL{l} \dL{l}^T)_{1:N_l-1, 1:N_l-1} &
      (-\eL{l} \dL{l}^T)_{1:N_l-1, N_l} &
      0 \\
      0 &
      0 &
      0 \\
      0 &
      (\eR{r} \dR{r}^T)_{2:N_r, 1} &
      (\eR{r} \dR{r}^T)_{2:N_r, 2:N_r}
    \end{pmatrix},
  \end{aligned}
  \end{equation}
  where the first matrix is negative semidefinite. This is of the
  required form $\M \D2 = -\A2 + \eR \dR^T - \eL \dL^T$.
  Again, the order of accuracy can be checked by applying $\D2$ to a
  polynomial.
\end{proof}

\begin{remark}
  Ignoring the outer boundaries, applying the SBP property
  \eqref{def:D2-bounded} to the continuously coupled second-derivative
  operator yields a direct discretization of the weak form
  \begin{equation}
    - \int_{x_{1,l}}^{x_{N_r,r}} (\partial_x \phi) (\partial_x u).
  \end{equation}
\end{remark}

\begin{corollary}
\label{cor:D2-CG}
  Coupling second-derivative SBP operators continuously as described in
  Theorem~\ref{thm:D2-CG} on a periodic domain results in a
  periodic second-derivative SBP operator.
\end{corollary}

\begin{example}
  Coupling second-derivative SBP operators using nodal Lobatto-Legendre bases
  for polynomials of degree $p = 1$ continuously on a uniform mesh results in
  \begin{equation}
    \M = \dx
    \begin{psmallmatrix}
      \nicefrac{1}{2} \\
      & 1 \\
      && \ddots \\
      &&& 1 \\
      &&&& \nicefrac{1}{2}
    \end{psmallmatrix},
    \qquad
    \D2 = \frac{1}{\dx^2}
    \begin{psmallmatrix}
      0 & 0 \\
      1 & -2 & 1 \\
      & \ddots & \ddots & \ddots \\
      && 1 & -2 & 1 \\
      &&& 0 & 0
    \end{psmallmatrix},
  \end{equation}
  which is very similar to the narrow-stencil second-derivative SBP operator
  \begin{equation}
    \M = \dx
    \begin{psmallmatrix}
      \nicefrac{1}{2} \\
      & 1 \\
      && \ddots \\
      &&& 1 \\
      &&&& \nicefrac{1}{2}
    \end{psmallmatrix},
    \qquad
    \D2 = \frac{1}{\dx^2}
    \begin{psmallmatrix}
      1 & -2 & 1 \\
      1 & -2 & 1 \\
      & \ddots & \ddots & \ddots \\
      && 1 & -2 & 1 \\
      && 1 & -2 & 1
    \end{psmallmatrix},
  \end{equation}
  of \cite{mattsson2004summation} but uses a different boundary closure.
\end{example}

\subsection{Some useful properties of periodic SBP operators}

Here, we gather some properties of periodic SBP operators that will be
useful to prove conservation properties later in the article.

\begin{lemma}
\label{lem:1-M-Di}
  Periodic first, second, and fourth-derivative operators satisfy
  $\vec{1}^T \M \D{i} = \vec{0}^T$.
  Periodic first-derivative upwind SBP operators satisfy
  $\vec{1}^T \M \D1{\pm} = \vec{0}^T$.
\end{lemma}
\begin{proof}
  Applying the defining conditions
  \eqref{eq:D1-periodic} \& \eqref{eq:D2-periodic} \& \eqref{eq:D4-periodic}
  and consistency of the derivative operators yields
  \begin{equation}
  \begin{aligned}
    \vec{1}^T \M \D1
    &=
    - \vec{1}^T \D1^T \M
    =
    \vec{0}^T,
    \\
    \vec{1}^T \M \D2
    &=
    \vec{1}^T \D2^T \M
    =
    \vec{0}^T,
    \\
    \vec{1}^T \M \D4
    &=
    \vec{1}^T \D4^T \M
    =
    \vec{0}^T.
  \end{aligned}
  \end{equation}
  Similarly,
  \begin{equation}
    \vec{1}^T \M \D1{\pm}
    =
    - \vec{1}^T \D1{\mp}^T \M
    =
    \vec{0}^T.
    \qedhere
  \end{equation}
\end{proof}

\begin{lemma}
\label{lem:1-in-left-kernel-of-ImD2inv}
  If $\D2$ is a periodic second-derivative SBP operator with mass matrix
  $\M$, then $\vec{1}^T \M (\I - \D2)^{-1} = \vec{1}^T \M$.
\end{lemma}
\begin{proof}
  Since $\M (\I - \D2)^{-1}$ is symmetric,
  \begin{equation}
    \vec{1}^T \M (\I - \D2)^{-1}
    =
    \left( \M (\I - \D2)^{-1} \vec{1} \right)^T
    =
    \left( \M \vec{1} \right)^T.
  \end{equation}
  Here, we used $(\I - \D2) \vec{1} = \vec{1}$, since $\D2 \vec{1} = \vec{0}$
  for any consistent second-derivative approximation $\D2$.
\end{proof}

\begin{lemma}
\label{lem:D1-D2-commuting}
  If $\D1, \D2$ are commuting periodic first- and second-derivative
  SBP operators with the same mass matrix $\M$, then
  $\M (\I - \D2)^{-1} \D1$ is skew-symmetric.
\end{lemma}
\begin{proof}
  Since
  \begin{equation}
    \I - \D2
    =
    \M^{-1} \left( \I - \M \D2 \M^{-1} \right) \M
    =
    \M^{-1} \left( \I - \D2^T \right) \M,
  \end{equation}
  we have
  \begin{equation}
    \M (\I - \D2)^{-1} \D1
    =
    \left( \I - \D2^T \right)^{-1} \M \D1
    =
    - \left( \I - \D2^T \right)^{-1} \D1^T \M
    =
    - \D1^T \left( \I - \D2^T \right)^{-1} \M.
    \qedhere
  \end{equation}
\end{proof}

In order to use the lemmas above, we will need pairs of first- and
second-derivative operators that commute.  In the following examples
we see that certain natural approaches lead to commuting operators, while
others do not.

\begin{example}
  The finite difference methods described in Section~\ref{sec:FD-collocation}
  in periodic domains result in commuting first- and second-derivative
  operators, since these can be represented by
  circulant matrices, \ie by Toeplitz matrices where each row is obtained
  from the preceding row by cyclically shifting every entry one step to the right
  \cite[Section~C.7]{leveque2007finite}.
\end{example}

\begin{example}
\label{ex:D2-D1-DG}
  Let $\D1$ be a first-derivative SBP operator and let $\D1{\pm} \ne \D1$
  be first-derivative upwind SBP operators all with the same mass matrix $\M$.
  Clearly, the first- and second-derivative operators $(\D1, \D1^2)$
  commute.  On the other hand, $\D1$ does
  not in general commute with the second-derivative operators
  $\D2=\D1{-} \D1{+}$ or $\D2=\D1{+} \D1{-}$.  Furthermore, in general $\M \D2 \D1$
  is not skew-symmetric.

  For example, in a periodic domain $[-1, 3]$ with two elements using nodal
  Lobatto-Legendre bases for polynomials of degree $p = 1$, we have
  \begin{equation}
  \begin{gathered}
    \D1 =
    \begin{psmallmatrix}
      0 & \frac{1}{2} & 0 & -\frac{1}{2} \\
      -\frac{1}{2} & 0 & \frac{1}{2} & 0 \\
      0 & -\frac{1}{2} & 0 & \frac{1}{2} \\
      \frac{1}{2} & 0 & -\frac{1}{2} & 0
    \end{psmallmatrix},
    \quad
    \D1{-} =
    \begin{psmallmatrix}
      \frac{1}{2} & \frac{1}{2} & 0 & -1 \\
      -\frac{1}{2} & \frac{1}{2} & 0 & 0 \\
      0 & -1 & \frac{1}{2} & \frac{1}{2} \\
      0 & 0 & -\frac{1}{2} & \frac{1}{2}
    \end{psmallmatrix},
    \quad
    \D1{+} =
    \begin{psmallmatrix}
      -\frac{1}{2} & \frac{1}{2} & 0 & 0 \\
      -\frac{1}{2} & -\frac{1}{2} & 1 & 0 \\
      0 & 0 & -\frac{1}{2} & \frac{1}{2} \\
      1 & 0 & -\frac{1}{2} & -\frac{1}{2}
    \end{psmallmatrix},
    \\
    \M \D1{-} \D1{+} \D1 =
    \begin{psmallmatrix}
      \frac{1}{4} & -\frac{5}{4} & -\frac{1}{4} & \frac{5}{4} \\
      \frac{1}{4} & -\frac{1}{4} & -\frac{1}{4} & \frac{1}{4} \\
      -\frac{1}{4} & \frac{5}{4} & \frac{1}{4} & -\frac{5}{4} \\
      -\frac{1}{4} & \frac{1}{4} & \frac{1}{4} & -\frac{1}{4}
    \end{psmallmatrix},
    \quad
    \M \D1{+} \D1{-} \D1 =
    \begin{psmallmatrix}
      \frac{1}{4} & -\frac{1}{4} & -\frac{1}{4} & \frac{1}{4} \\
      \frac{5}{4} & -\frac{1}{4} & -\frac{5}{4} & \frac{1}{4} \\
      -\frac{1}{4} & \frac{1}{4} & \frac{1}{4} & -\frac{1}{4} \\
      -\frac{5}{4} & \frac{1}{4} & \frac{5}{4} & -\frac{1}{4}
    \end{psmallmatrix}.
  \end{gathered}
  \end{equation}
  Hence, the second-derivative operator obtained via the LDG
  procedure does not, in general, commute with the corresponding
  first-derivative operator. Using instead the first method of
  Bassi and Rebay, $\D2 = \D1^2$, the operators commute.
\end{example}

\begin{example}
\label{ex:D2-D1-CG}
  Similarly to the discontinuous coupling described in
  Example~\ref{ex:D2-D1-DG}, a continuous coupling also does not
  result in commuting first- and second-derivative operators and
  $\M \D1 \D2$ is not skew-symmetric in general.
  In a periodic domain $[-1, 3]$ with two elements using nodal
  Lobatto-Legendre bases for polynomials of degree $p = 2$,
  \begin{equation}
  \begin{aligned}
    \D1 &=
    \begin{psmallmatrix}
      0 & 1 & 0 & -1 \\
      -\frac{1}{2} & 0 & \frac{1}{2} & 0 \\
      0 & -1 & 0 & 1 \\
      \frac{1}{2} & 0 & -\frac{1}{2} & 0
    \end{psmallmatrix},
    &
    \D2 &=
    \begin{psmallmatrix}
      -\frac{7}{2} & 2 & -\frac{1}{2} & 2 \\
      1 & -2 & 1 & 0 \\
      -\frac{1}{2} & 2 & -\frac{7}{2} & 2 \\
      1 & 0 & 1 & -2
    \end{psmallmatrix},
    \\
    \M &=
    \begin{psmallmatrix}
      \frac{2}{3} & 0 & 0 & 0 \\
      0 & \frac{4}{3} & 0 & 0 \\
      0 & 0 & \frac{2}{3} & 0 \\
      0 & 0 & 0 & \frac{4}{3}
    \end{psmallmatrix},
    &
    \M \D2 \D1 &=
    \begin{psmallmatrix}
      0 & -2 & 0 & 2 \\
      \frac{4}{3} & 0 & -\frac{4}{3} & 0 \\
      0 & 2 & 0 & -2 \\
      -\frac{4}{3} & 0 & \frac{4}{3} & 0
    \end{psmallmatrix}.
  \end{aligned}
  \end{equation}
  Hence, the second-derivative operator obtained via the continuous
  coupling procedure does not, in general, commute with the corresponding
  first-derivative operator. Moreover, $\M \D2 \D1$ is indefinite.
  In order to obtain commuting operators, one may
  again to use the squared first-derivative operator
  as a second-derivative operator, resulting in a wide-stencil operator.
\end{example}

\begin{lemma}
\label{lem:D2-via-D1pm}
  If $\D1{\pm}$ are periodic upwind SBP operators,
  $\M (\I - \D1{-} \D1{+})^{-1} \D1{-}$ is positive semidefinite and
  $\M (\I - \D1{+} \D1{-})^{-1} \D1{+}$ is negative semidefinite.
\end{lemma}
\begin{proof}
  It suffices to check whether
  \begin{equation}
    \scp{ \vec{u} }{ (\I - \D1{-} \D1{+})^{-1} \D1{-} \vec{u} }_{\M}
    \ge 0
  \end{equation}
  for all $\vec{u}$. Equivalently, one can consider
  $\vec{w} = (\I - \D1{-} \D1{+})^{-1} \vec{u}$
  and compute
  \begin{equation}
  \begin{aligned}
    &\quad
    \scp{ \vec{u} }{ (\I - \D1{-} \D1{+})^{-1} \D1{-} \vec{u} }_{\M}
    \\
    &=
    \scp{ \vec{w} }{ \D1{-} (\I - \D1{-} \D1{+}) \vec{w} }_{\M}
    =
    - \scp{ \vec{w} }{ (\I - \D1{-} \D1{+}) \D1{+} \vec{w} }_{\M}
    \\
    &=
    \frac{1}{2} \scp{ \vec{w} }{ (\D1{-} - \D1{+}) \vec{w} }_{\M}
    - \frac{1}{2} \scp{ \vec{w} }{ \D1{-}^2 \D1{+} \vec{w} }_{\M}
    + \frac{1}{2} \scp{ \vec{w} }{ \D1{-} \D1{+}^2 \vec{w} }_{\M}
    \\
    &=
    \frac{1}{2} \scp{ \vec{w} }{ (\D1{-} - \D1{+}) \vec{w} }_{\M}
    + \frac{1}{2} \scp{ \D1{+} \vec{w} }{ (\D1{-} - \D1{+}) \D1{+} \vec{w} }_{\M}
    \ge 0.
  \end{aligned}
  \end{equation}
  The other assertion is verified by exchanging $+$ and $-$.
\end{proof}

\subsection{Choice of appropriate split forms}
\label{sec:split-forms}

Since SBP operators mimic integration by parts discretely, proofs of invariant
conservation can be transferred to the discrete level directly if this tool
is used. However, systematic integration by parts is often coupled with the
application of the product or chain rule, which do not, in general, hold discretely
\cite{ranocha2019mimetic}. Instead, split forms can be used in a systematic way.
A general recipe for constructing discretely conservative split forms is given
in the following.
\begin{quote}
  Check whether conservation of a nonlinear invariant can be proved using
  only integration by parts and symmetry properties of differential operators.
  If so, apply the same steps discretely using SBP operators.
  Otherwise, write nonlinear terms as linear combinations of different split
  forms obtained by the product/chain rule and repeat the procedure.
\end{quote}
As an example, consider Burgers' equation
\begin{equation}
  \partial_t u(t,x) + \partial_x \frac{u(t,x)^2}{2} = 0
\end{equation}
in a periodic domain. Conservation of the $L^2$ norm can be shown by
applying the chain rule as in
\begin{equation}
\label{eq:burgers-chain-rule}
  \od{}{t} \frac{1}{2} \| u(t) \|_{L^2}^2
  =
  \int u \partial_t u
  =
  - \frac{1}{2} \int u \partial_x u^2
  =
  - \frac{1}{3} \int \partial_x u^3
  =
  0.
\end{equation}
In order to achieve discrete conservation, we look for a way
to show conservation using only integration by parts.
To this end, consider the general splitting of the nonlinear term:
\begin{equation}
  \partial_x \frac{u^2}{2}
  =
  \alpha \partial_x u^2 + (1 - 2 \alpha) u \partial_x u,
\end{equation}
where $\alpha \in \R$ is a real parameter. The energy method using only
integration by parts yields
\begin{equation}
\label{eq:burgers-splitting}
  \od{}{t} \frac{1}{2} \| u(t) \|_{L^2}^2
  =
  \int u \partial_t u
  =
  - \alpha \int u \partial_x u^2
  - (1 - 2 \alpha) \int u^2 \partial_x u
  =
  (1 - 3 \alpha) \int u \partial_x u^2.
\end{equation}
Energy conservation can be obtained by taking $\alpha = \nicefrac{1}{3}$,
and this leads naturally to a conservative numerical method.

The quadratic nonlinearity of Burgers' equation appears in several of the
dispersive wave equations considered in this article, and conservative methods
for them can be designed using the split form just derived.  Such split forms can
also be generalized to higher-order polynomial nonlinearities
\cite[Section~4.5]{ranocha2018thesis} and there are efficient means to
evaluate these for FD, DG, and CG methods \cite{gassner2016split,
ranocha2018comparison}. The same idea can even be applied to get conservative
discretizations for non-polynomial nonlinearities \cite{fisher2013discretely},
but the resulting schemes cannot be interpreted in terms of split forms.

\section{Relaxation methods in time}
\label{sec:relaxation}

In the previous section, we have developed tools for producing conservative
spatial discretizations; in the following sections, these will be applied
to obtain conservative semi-discrete methods for specific wave equations.
These methods reduce an initial boundary value PDE to an initial value ODE system
\begin{align}
    u'(t) & = f(u(t)), \quad u(0) = u_0,
\end{align}
satisfying a conservation property
\begin{align}
  \od{}{t} J(u) & = 0
\end{align}
for some nonlinear invariant $J$.  Herein we employ one-step integration
methods and we enforce the conservation
property discretely in time, so that $J(u^n) = J(u^{n-1}) = J(u_0)$.
We can achieve this by combining our conservative spatial discretizations with
relaxation methods in time
\cite{ketcheson2019relaxation,ranocha2020relaxation,ranocha2020relaxationHamiltonian,ranocha2020general}.

We start with a Runge-Kutta method
\begin{subequations}
\label{eq:RK-step}
\begin{align}
\label{eq:RK-stages}
  y_i
  &=
  u^n + \dt \sum_{j=1}^{s} a_{ij} \, f(t_n + c_j \dt, y_j),
  \qquad i \in \set{1, \dots, s},
  \\
\label{eq:RK-final}
  u(t_n + \dt) \approx u^{n+1}
  &=
  u^n + \dt \sum_{i=1}^{s} b_{i} \, f(t_n + c_i \dt, y_i),
\end{align}
\end{subequations}
and define
\begin{align}
  d^n := \sum_{i=1}^{s} b_{i} f_i,
\end{align}
where we use the shorthand $f_i := f(t_n + c_i \dt, y_i)$.
In general the new solution $u^n$ will not be conservative, so we
replace the update formula \eqref{eq:RK-final} with
an update in the same direction but of a (possibly) different length:
\begin{align} \label{eq:u-gamma}
  u(t_n + \gamma^n\dt) \approx u^{n+1}_\gamma &= u^n + \gamma^n \dt d^n.
\end{align}
The relaxation parameter $\gamma^n$ is chosen as a solution of the conservation
equation
\begin{align} \label{eq:gamma}
    J(u^{n+1}_\gamma) & = J(u^n).
\end{align}
Thus $\gamma^n$ is obtained by solving a scalar nonlinear equation, using some
root-finding method. By the general theory on relaxation methods, there is
exactly one root $\gamma^n = 1 + \O(\dt^{p-1})$ of \eqref{eq:gamma}
\cite[Theorem~2.14]{ranocha2020general}. Other possible roots, such as the
trivial solution $\gamma = 0$, are further away from unity. For quadratic
invariants $J$, these two roots are the only roots and can be computed
explicitly. Similarly, the root $\gamma^n$ closest to unity can also be computed
explicitly for cubic invariants such as for the BBM-BBM system discussed
in Section~\ref{sec:bbm_bbm}. However, these explicit formulas can be sensitive
to floating point errors; we have found that the application of standard
root finding algorithms such as those of \cite{alefeld1995algorithm}
is efficient and often results in more accurate solutions.

The resulting solution $u^{n+1}_\gamma$ conserves the invariant $J$ by construction.
In contrast to projection methods, the relaxation approach also automatically
conserves linear invariants (as long as the semi-discretization conserves
them).  The solution \eqref{eq:u-gamma} has the same local order of accuracy
as that given by the original Runge-Kutta method \eqref{eq:RK-step}.

The use of relaxation Runge-Kutta methods in this context may be viewed as
an application of the ideas developed in \cite{ranocha2020relaxationHamiltonian}.
For some dispersive wave problems, the value of $J$ might change over time
due to boundary conditions or the presence of dissipative terms; in this
case relaxation methods can also be used to improve the accuracy of the
time evolution of $J$ \cite{ketcheson2019relaxation,ranocha2020relaxation}.
For more details regarding the properties of relaxation methods, including
multistep relaxation methods, we refer the reader to \cite{ranocha2020general}.

\section{Conservative discretizations of specific wave equations}
\label{sec:specific-wave-equations}

In this section we develop new conservative discretizations for
several nonlinear dispersive wave equations, using the tools from
the previous two sections.  Most of the discretizations are based on
using appropriate splittings for nonlinear terms and using a basis with a
diagonal mass matrix.  For some
of the equations studied (namely, the Fornberg-Whitham, Camassa-Holm, Hone-Holm,
and BBM-BBM equations) we also require that the various discrete derivative
operators commute.  For other equations (namely, the BBM and Degasperis-Procesi
equations) this is not necessary.

Most of the proposed methods (all except for the BBM-BBM system)
require that the discrete derivative operators
have a diagonal mass matrix.  Alternatively, one can achieve
conservation by discretizing the conservative form of the equation
and using exact integration of all variational forms.  The
latter approach has been used to construct conservative methods,
\eg in \cite{li2020optimal,zhang2020dissipative,xu2011local,xia2014fourier}.
An advantage of the present approach is that
exact integration of the nonlinear terms is not necessary if the
mass matrix is diagonal.

We focus on the development of the discretizations and proofs of their
conservation properties, but we also provide results of simple numerical tests 
that confirm the theoretical properties of the schemes.
Numerical results for each equation are described in the corresponding section
and are also summarized at the end of this work in Table~\ref{tab:summary}.
More extensive numerical experiments, such as studies of solitary 
wave interaction, are left to future work.

We test the accuracy using the
method of manufactured solutions, which consists of choosing a smooth
solution a priori and then adding a source term $f(t,x)$ to the PDE so that
the solution satisfies it \cite{roache2002code}.  In order to isolate the
spatial discretization errors, we discretize in time using the fifth-order
explicit Runge-Kutta pair of \cite{tsitouras2011runge}
with adaptive time stepping and a local error tolerance of $\num{1.0e-12}$,
without relaxation.  For all of the spatial discretizations
proposed, linearized stability analysis suggests that the maximum stable time step
is either proportional to $\dx$ or independent of $\dx$, so explicit time integration
can be efficient.

Conservation tests are performed using solitary wave solutions,
obtained either analytically or via the Petviashvili method
\cite{petviashvili1976equation} using a Fourier collocation method
with $N = 2^{16}$ nodes.  For these tests
we use the classical 4th-order method of Kutta \cite{kutta1901beitrag},
with relaxation. Other space and time discretizations have been tested as well
but are not shown here. We remark in advance that in all cases,
these tests demonstrate conservation of all linear invariants
and the selected nonlinear invariant.  Results for specific discretizations
are shown only when they reveal something of further interest.

In the following, all errors of the form $\| u - u_\mathrm{ana} \|$, where $u$
is the numerical approximation and $u_\mathrm{ana}$ the analytical solution,
are computed using the discrete norm induced by the mass matrix $\M$.
These discrete $L^2$ errors are used to compute the experimental order of
convergence (EOC).

\subsection{Benjamin-Bona-Mahony equation}
\label{sec:bbm}

Consider the Benjamin-Bona-Mahony equation \cite{benjamin1972model}
(also known as regularized long wave equation)
\begin{equation}
\label{eq:bbm-dir}
\begin{aligned}
  (\I - \partial_x^2) \partial_t u(t,x)
  + \partial_x f(u(t,x))
  + \partial_x u(t,x)
  &= 0,
  && t \in (0, T), x \in (\xmin, \xmax),
  \\
  u(0, x) &= u^0(x),
  && x \in [\xmin, \xmax],
  \\
  f(u) &= \frac{u^2}{2},
\end{aligned}
\end{equation}
with periodic boundary conditions, which can also be written as
\begin{equation}
\label{eq:bbm-inv}
  \partial_t u(t,x)
  + (\I - \partial_{x,P}^2)^{-1} \partial_x \left(
    f(u(t,x)) + u(t,x)
  \right) = 0,
\end{equation}
where $(\I - \partial_{x,P}^2)^{-1}$ is the inverse of the elliptic
operator $\I - \partial_x^2$ with periodic boundary conditions.
The functionals
\begin{subequations}
\label{eq:bbm-invariants}
\begin{align}
\label{eq:bbm-invariants-linear}
  J^{\text{BBM}}_1(u)
  &= \int_{\xmin}^{\xmax} u,
  \\
\label{eq:bbm-invariants-quadratic}
  J^{\text{BBM}}_2(u)
  &= \frac{1}{2} \int_{\xmin}^{\xmax} ( u^2 + (\partial_x u)^2 )
  = \frac{1}{2} \int_{\xmin}^{\xmax} u (\I - \partial_x^2) u,
  \\
\label{eq:bbm-invariants-cubic}
  J^{\text{BBM}}_3(u)
  &= \int_{\xmin}^{\xmax} (u + 1)^3,
\end{align}
\end{subequations}
are invariants of solutions \cite{olver1979euler}.
In the following, we will construct numerical methods that conserve
the linear \eqref{eq:bbm-invariants-linear} and quadratic
\eqref{eq:bbm-invariants-quadratic} invariants but not necessarily the
cubic invariant \eqref{eq:bbm-invariants-cubic}.

\subsubsection{Conservative numerical methods}

The rate of change of the quadratic invariant
\eqref{eq:bbm-invariants-quadratic} is
\begin{equation}
  \int_{\xmin}^{\xmax} u (\I - \partial_x^2) \partial_t u
  =
  - \int_{\xmin}^{\xmax} u \partial_x \frac{u^2}{2}
  - \int_{\xmin}^{\xmax} u \partial_x u.
\end{equation}
The first integral on the right-hand side is exactly the one appearing in
the energy rate for Burgers' equation \eqref{eq:burgers-chain-rule}. Hence,
the linear and quadratic invariants are conserved semidiscretely if
periodic SBP operators are employed and the semidiscretization uses
the split form
\begin{equation}
\label{eq:bbm-inv-periodic-SBP}
  \partial_t \vec{u}
  + (\I - \D2)^{-1} \left(
    \frac{1}{3} \D1 \vec{u}^2
    + \frac{1}{3} \vec{u} \D1 \vec{u}
    + \D1 \vec{u}
  \right)
  =
  \vec{0}.
\end{equation}
\begin{theorem}
\label{thm:bbm-inv-periodic-SBP}
  If $\D1$ is a periodic first-derivative SBP operator
  with diagonal mass matrix $\M$
  and $\D2$ is a periodic second-derivative SBP operator,
  then the semidiscretization \eqref{eq:bbm-inv-periodic-SBP}
  conserves the invariants \eqref{eq:bbm-invariants-linear} and
  \eqref{eq:bbm-invariants-quadratic} of \eqref{eq:bbm-dir}.
\end{theorem}
\begin{proof}
  The linear invariant \eqref{eq:bbm-invariants-linear} is conserved
  since
  \begin{equation}
  \begin{aligned}
    \od{}{t} \vec{1}^T \M \vec{u}
    &=
    \vec{1}^T \M \partial_t \vec{u}
    =
    - \vec{1}^T \M (\I - \D2)^{-1} \left(
      \frac{1}{3} \D1 \vec{u}^2
      + \frac{1}{3} \vec{u} \D1 \vec{u}
      + \D1 \vec{u}
    \right)
    \\
    &=
    - \vec{1}^T \M \left(
      \frac{1}{3} \D1 \vec{u}^2
      + \frac{1}{3} \vec{u} \D1 \vec{u}
      + \D1 \vec{u}
    \right)
    =
    0,
  \end{aligned}
  \end{equation}
  where Lemma~\ref{lem:1-in-left-kernel-of-ImD2inv} has been used.

  Since $\I - \D2$ is a symmetric operator, the semidiscrete rate of
  change of the quadratic invariant \eqref{eq:bbm-invariants-quadratic}
  is
  \begin{equation}
    \frac{1}{2} \od{}{t} \vec{u}^T \M (\I - \D2) \vec{u}
    =
    \vec{u}^T \M (\I - \D2) \partial_t \vec{u}
    =
    - \frac{1}{3} \vec{u}^T \M\D1 \vec{u}^2
    - \frac{1}{3} \vec{u}^T \M \vec{u} \D1 \vec{u}
    - \vec{u}^T \M\D1 \vec{u}
    =
    0.
    \qedhere
  \end{equation}
\end{proof}

\begin{remark}
\label{rem:koide2009nonlinear}
  Conservative linearly- and nonlinearly-implicit second-order finite difference
  schemes for the BBM equation based on the discrete variational
  derivative method were proposed in \cite{koide2009nonlinear}.
\end{remark}


\begin{remark}
\label{rem:bbm-dissipative}
  Starting from the conservative semidiscretization \eqref{eq:bbm-inv-periodic-SBP},
  energy-dissipative semidiscretizations can be constructed as well.
  For example, the linear term
  $(\I - \D2)^{-1} \D1 \vec{u}$
  can be replaced by
  $(\I - \D1{-} \D1{+})^{-1} \D1{-}$.
  Then, the linear invariant $J^{\text{BBM}}_1$ is still conserved and the contribution to the energy rate becomes
  $- \vec{u}^T \M (\I - \D1{-} \D1{+})^{-1} \D1{-} \vec{u} \le 0$
  because of Lemma~\ref{lem:D2-via-D1pm}.
  Additional dissipation can be introduced in DG schemes by applying a
  dissipative numerical flux such as Godunov's flux at interfaces for the
  nonlinear term. Another possibility is to add artificial dissipation
  proportional to $\D1{-} - \D1{+}$ to the nonlinear term on left hand side
  of \eqref{eq:bbm-inv-periodic-SBP}.
\end{remark}

\subsubsection{Convergence study in space}
\label{sec:bbm-periodic-manufactured-convergence}

To study convergence, the method of manufactured solutions \cite{roache2002code}
is applied to
\begin{equation}
\label{eq:periodic-manufactured}
  u(t,x) = \e^{t/2} \sin(2 \pi (x - t/2)), \qquad (x,t) \in [0,1]\times[0,1],
\end{equation}
with periodic boundary conditions.
Results for the semidiscretization \eqref{eq:bbm-inv-periodic-SBP}
are shown in Figure~\ref{fig:bbm-periodic-manufactured-convergence}.

\begin{figure}[htbp]
\centering
  \begin{subfigure}[t]{0.49\textwidth}
  \centering
    \includegraphics[width=\textwidth]{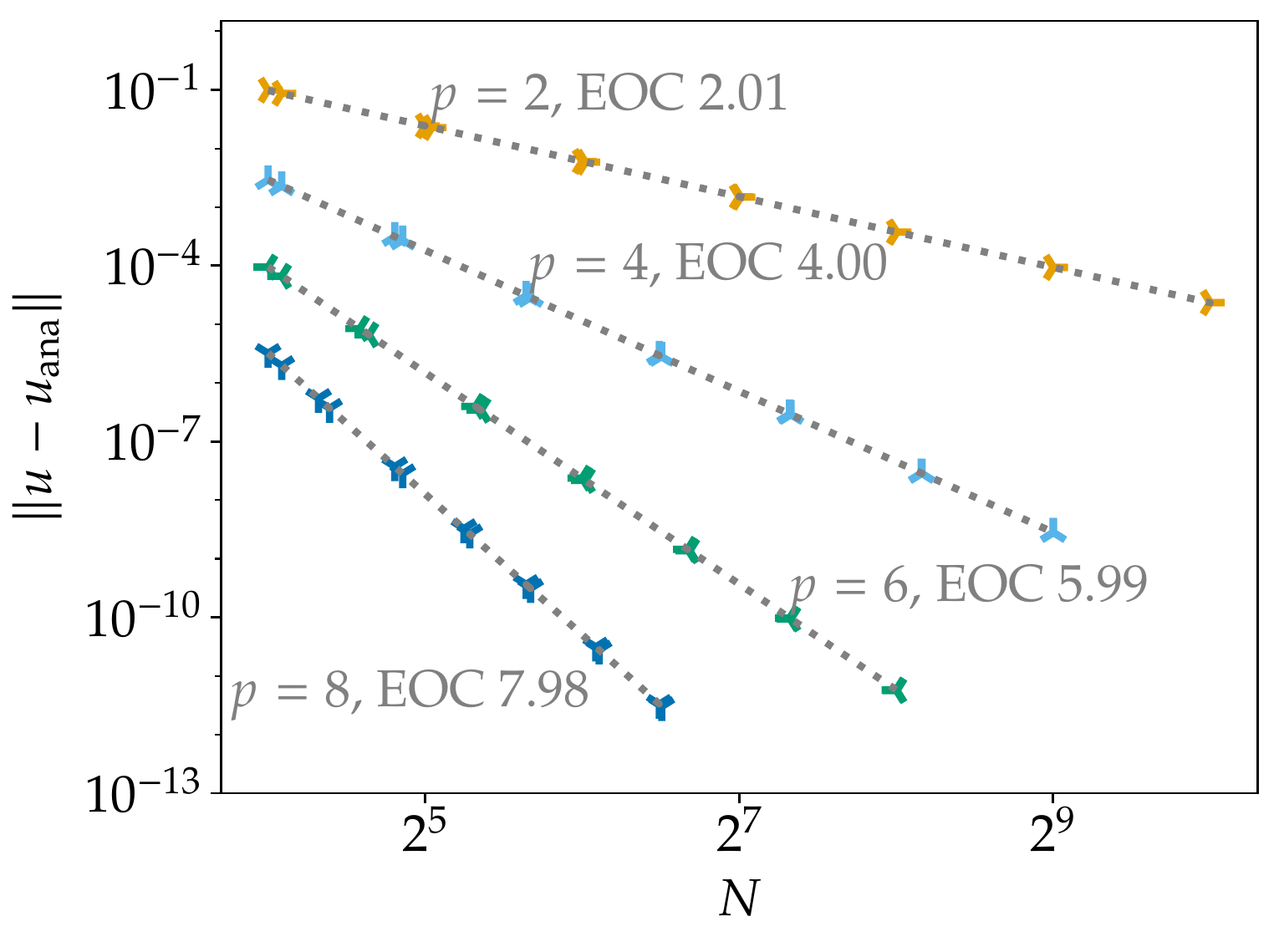}
    \caption{Finite difference methods, $\D2 = \D1^2$.}
  \end{subfigure}%
  \hspace*{\fill}
  \begin{subfigure}[t]{0.49\textwidth}
  \centering
    \includegraphics[width=\textwidth]{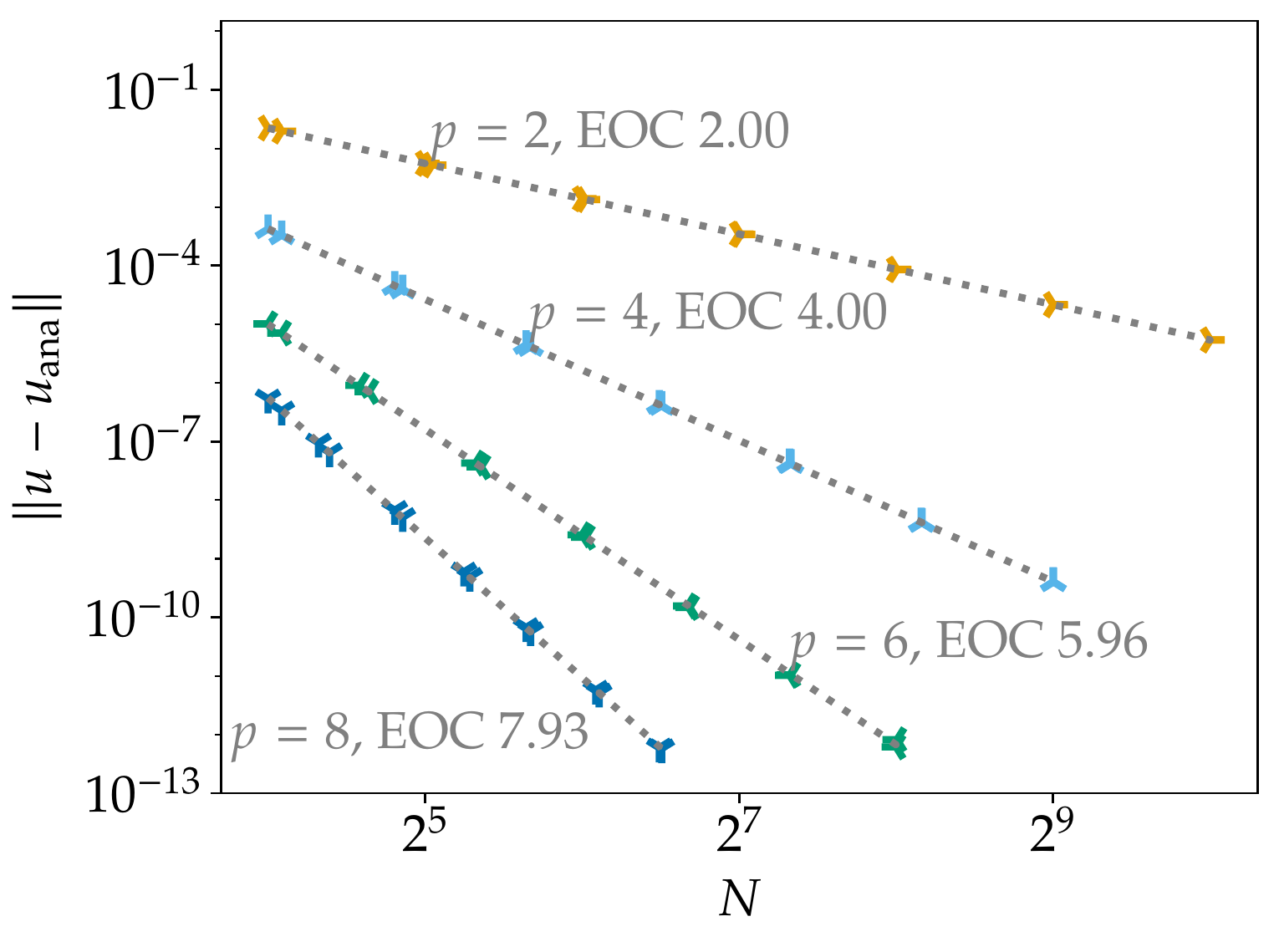}
    \caption{FD methods with narrow-stencil $\D2$.}
  \end{subfigure}%
  \\
  \begin{subfigure}[t]{0.8\textwidth}
  \centering
    \includegraphics[width=\textwidth]{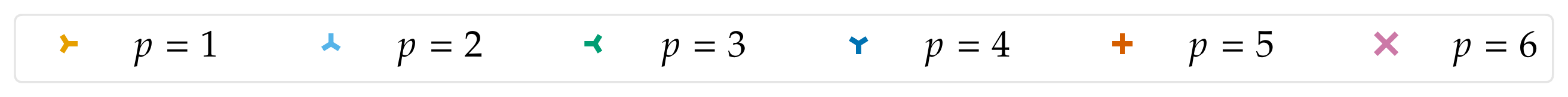}
  \end{subfigure}%
  \\
  \begin{subfigure}[t]{0.49\textwidth}
  \centering
    \includegraphics[width=\textwidth]{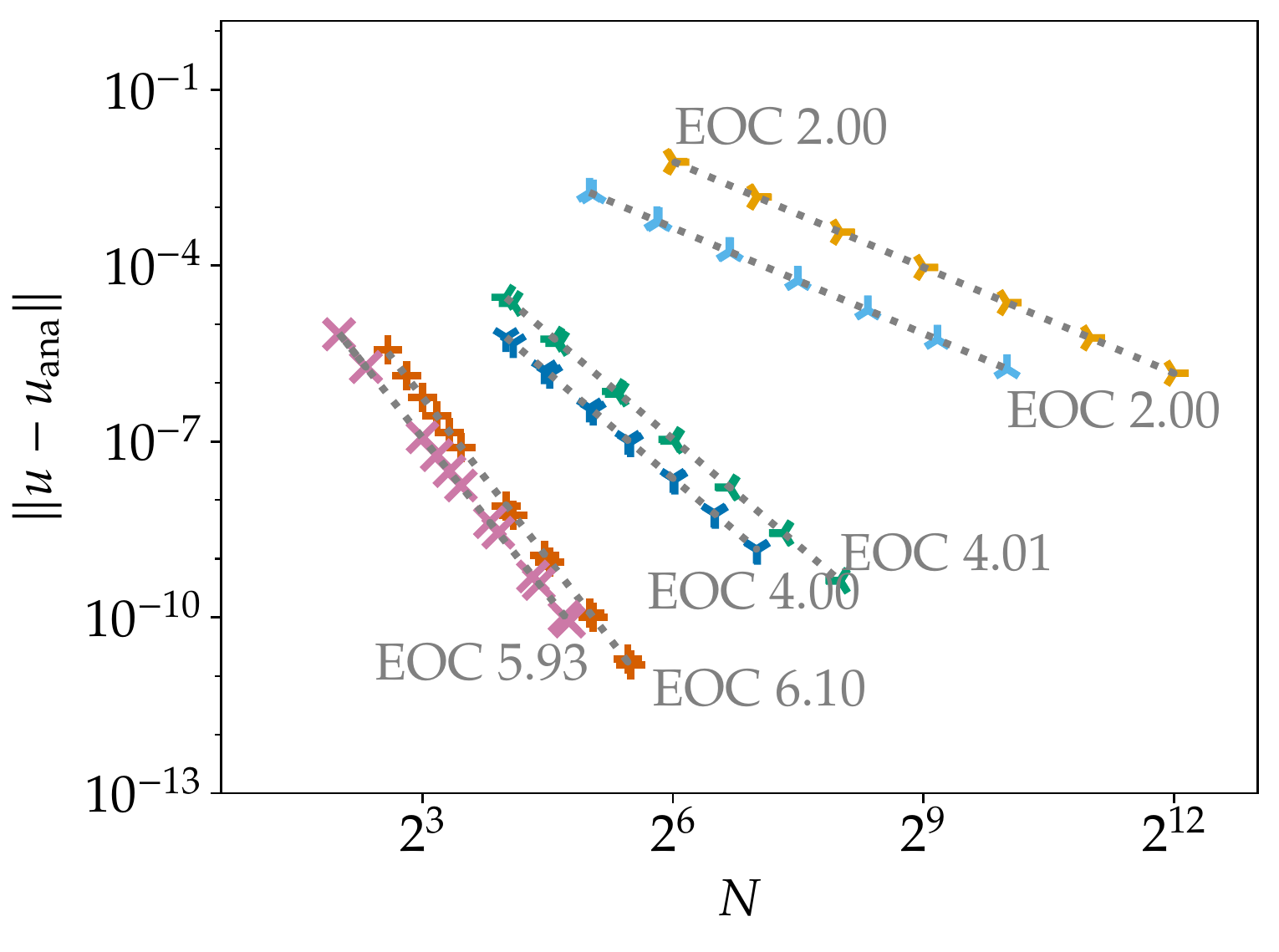}
    \caption{Continuous Galerkin methods, $\D2 = \D1^2$.}
  \end{subfigure}%
  \hspace*{\fill}
  \begin{subfigure}[t]{0.49\textwidth}
  \centering
    \includegraphics[width=\textwidth]{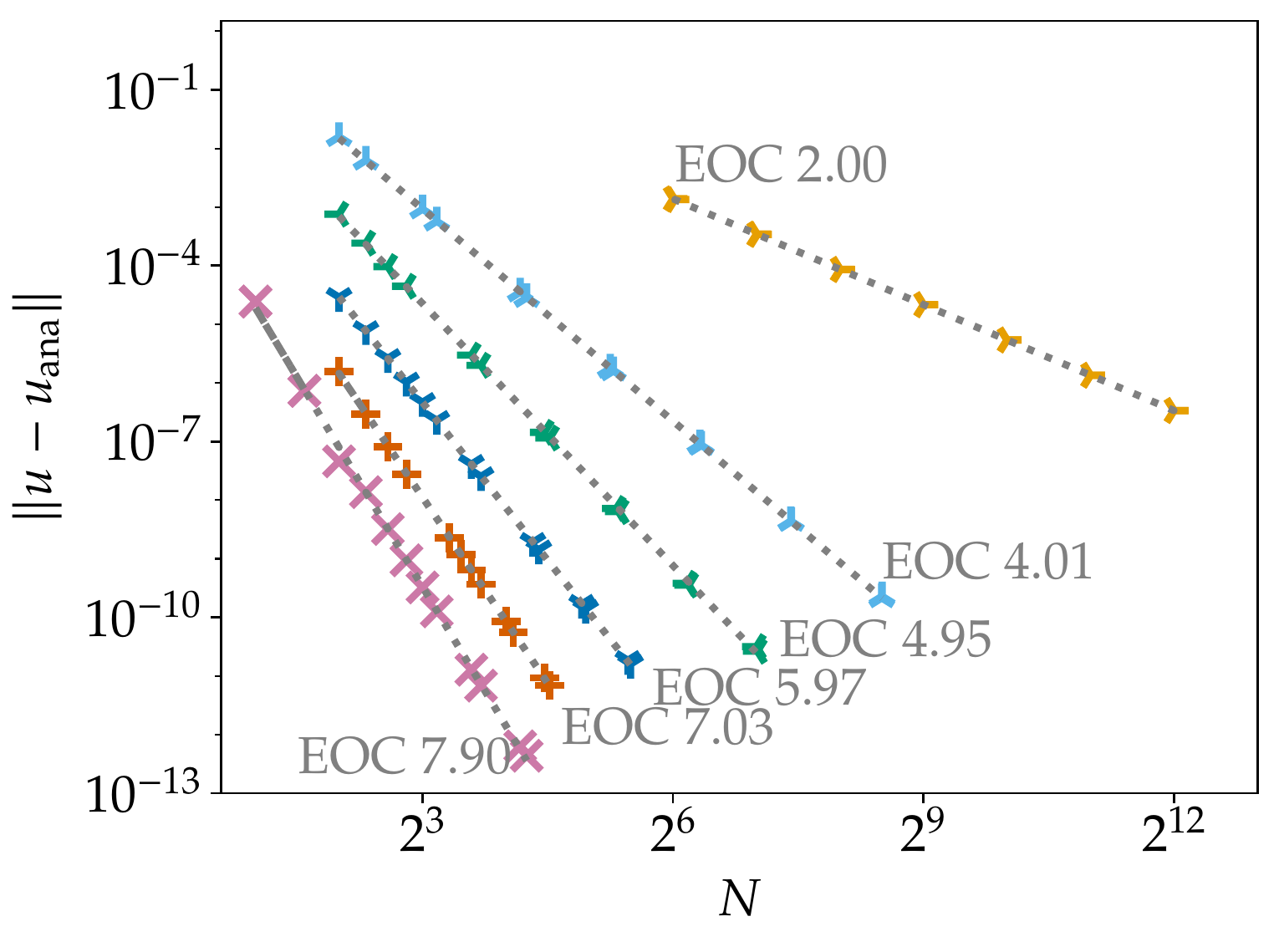}
    \caption{CG methods, narrow-stencil $\D2$.}
  \end{subfigure}%
  \\
  \begin{subfigure}[t]{0.49\textwidth}
  \centering
    \includegraphics[width=\textwidth]{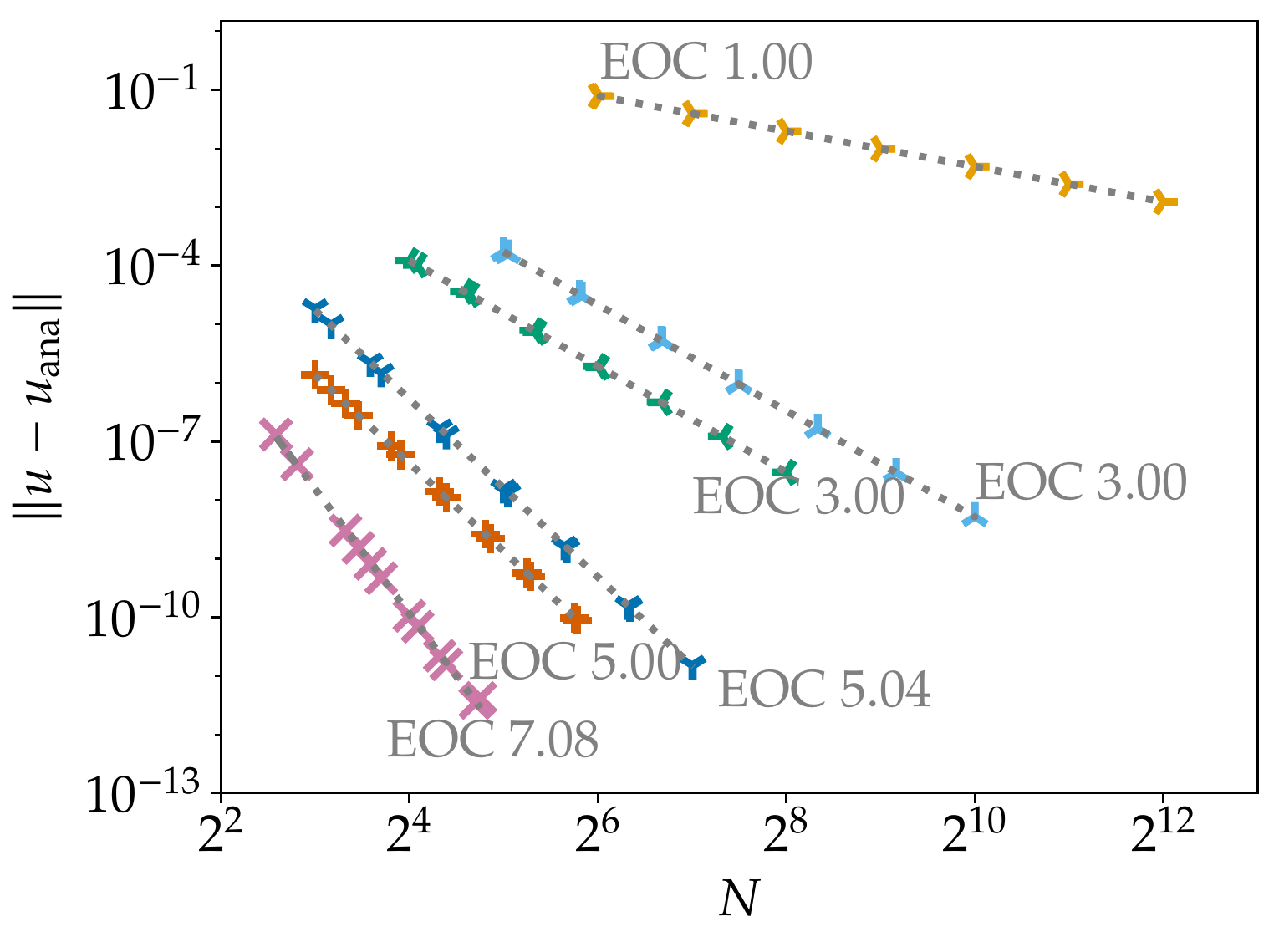}
    \caption{Discontinuous Galerkin methods, $\D2 = \D1^2$.}
  \end{subfigure}%
  \hspace*{\fill}
  \begin{subfigure}[t]{0.49\textwidth}
  \centering
    \includegraphics[width=\textwidth]{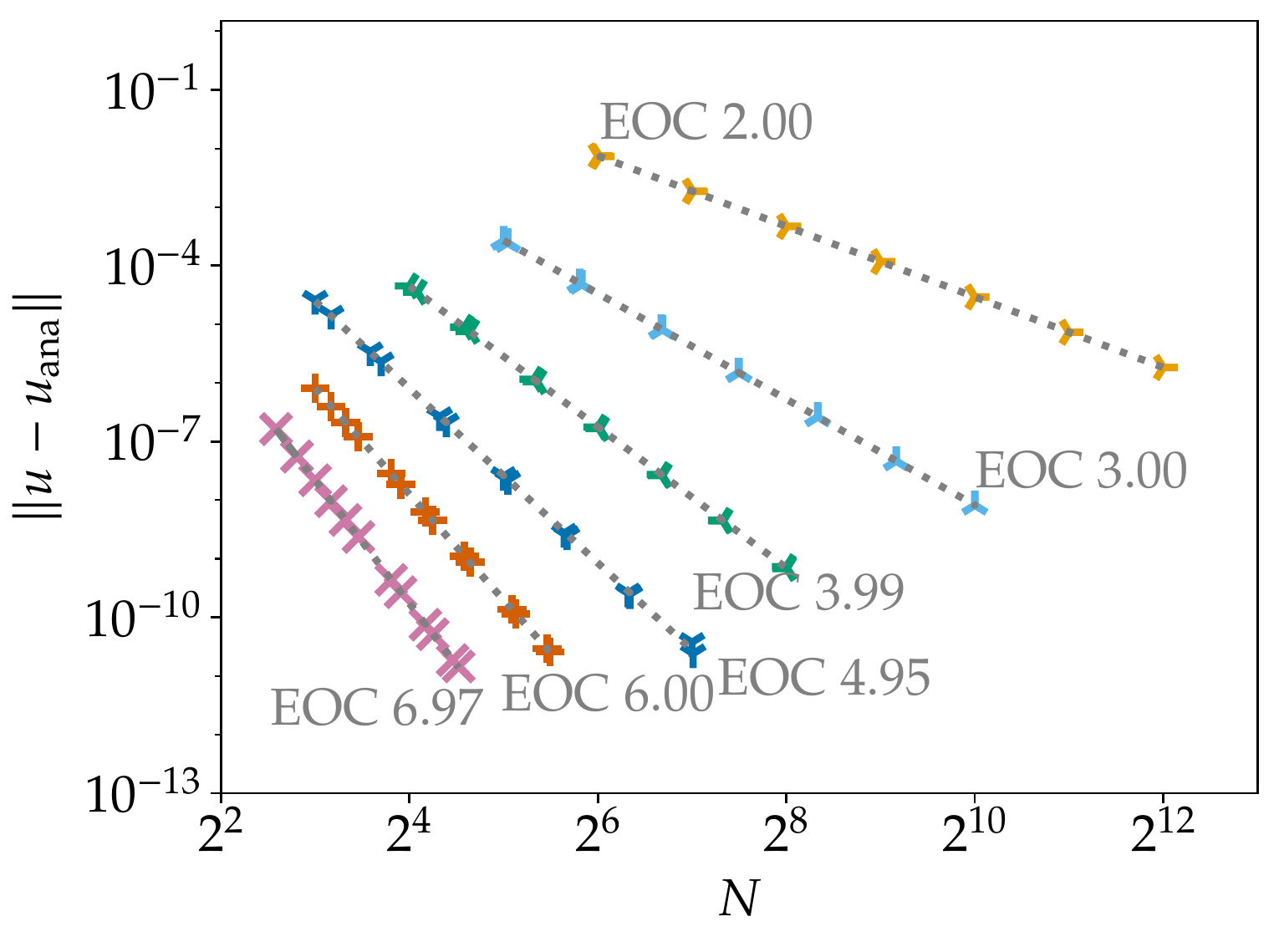}
    \caption{DG methods, $\D2 = \D1{+} \D1{-}$.}
  \end{subfigure}%
  \caption{Convergence results of the spatial semidiscretizations
           \eqref{eq:bbm-inv-periodic-SBP}
           for the manufactured solution \eqref{eq:periodic-manufactured}
           of the BBM equation.
           All of these semidiscretizations conserve the linear and quadratic
           invariants \eqref{eq:bbm-invariants} of the BBM equation
           \eqref{eq:bbm-dir}.}
  \label{fig:bbm-periodic-manufactured-convergence}
\end{figure}

For central finite difference methods with order of accuracy $p$, we observe
an order of convergence approximately equal to $p$.
The results for wide-stencil and narrow-stencil
second-derivative operators are similar but the narrow-stencil operators
result in errors that are smaller by up to an order of magnitude.

For nodal continuous Galerkin methods using Lobatto-Legendre bases, the
results depend on the choice of the second-derivative operator.
Wide-stencil operators $\D2 = \D1^2$ with polynomial degree $p$ yield
$\text{EOC} \approx p+1$ for $p$ odd and
$\text{EOC} \approx p$ for $p$ even.
In contrast, the usual narrow-stencil approximation (Theorem~\ref{thm:D2-CG})
results in an
$\text{EOC} \approx p+1$ for $p = 1$ and an
$\text{EOC} \approx p+2$ for $p > 1$.

A similar observation can be made for nodal discontinuous Galerkin methods.
There, wide-stencil operators $\D2 = \D1^2$ yield
$\text{EOC} \approx p+1$ for even polynomial degrees $p$ and
$\text{EOC} \approx p$ for odd $p$.
The narrow-stencil LDG operator $\D2 = \D1{+} \D1{-}$ results in an
$\text{EOC} \approx p+1$ for all $p$.

\subsubsection{Conservation of invariants}
\label{sec:bbm-error-invariants}

To test the conservation properties of the scheme, we use
the traveling wave solution
\begin{equation}
\label{eq:bbm-traveling-wave}
  u(t,x) = A \cosh( K (x - c t) ),
  \quad A = 3 (c - 1),
  \quad K = \frac{1}{2} \sqrt{1 - 1 / c},
\end{equation}
with speed $c = 1.2$ in the periodic domain $[-90, 90]$.
The classical fourth-order Runge-Kutta method RK4 \cite{kutta1901beitrag}
is used with relaxation (as described in Section \ref{sec:relaxation})
to enforce conservation of $J^{\text{BBM}}_2$.
Results for all spatial discretizations show conservation of the
linear and quadratic invariants, to within roundoff error.
Interestingly, applying relaxation to conserve the quadratic invariant
\eqref{eq:bbm-invariants-quadratic} improves the evolution of the cubic
invariant \eqref{eq:bbm-invariants-cubic} as well.  An example showing
this behavior, with a Fourier pseudospectral method in space, can be seen in Figure~\ref{fig:bbm-error-invariants}.
Results for other spatial discretizations are similar.

Moreover, the deviation of the cubic invariant seems to be bounded and
approximately constant. This indicates that the remaining error of the
energy-conservative relaxation method is mainly a phase error and not an
amplitude/shape error.

\begin{figure}[htbp]
\centering
  \includegraphics[width=0.9\textwidth]{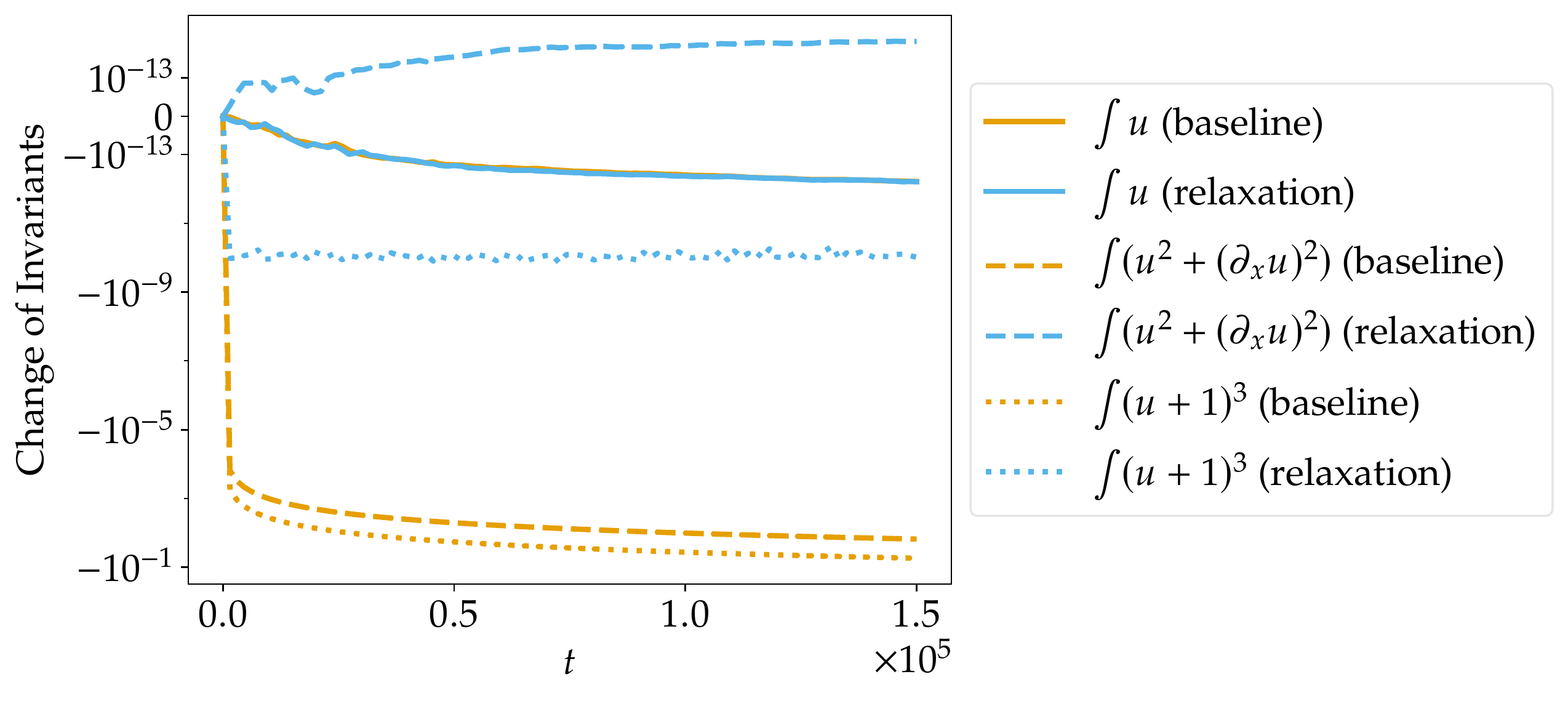}
  \caption{Temporal change in the invariants computed using RK4 with
           and without relaxation to preserve the quadratic invariant
           \eqref{eq:bbm-invariants-quadratic} for energy-conservative Fourier
           collocation semidiscretizations of the traveling wave solution
           \eqref{eq:bbm-traveling-wave} of the BBM equation \eqref{eq:bbm-dir}.}
  \label{fig:bbm-error-invariants}
\end{figure}

\subsection{Fornberg-Whitham equation}
\label{sec:fw}

Consider the Fornberg-Whitham equation \cite{whitham1967variational}
\begin{equation}
\label{eq:fw-dir}
\begin{aligned}
  (\I - \partial_x^2) \partial_t u(t,x)
  + (\I - \partial_x^2) \partial_x f(u(t,x))
  + \partial_x u(t,x)
  &= 0,
  && t \in (0, T), x \in (\xmin, \xmax),
  \\
  u(0, x) &= u^0(x),
  && x \in [\xmin, \xmax],
  \\
  f(u) &= \frac{u^2}{2},
\end{aligned}
\end{equation}
with periodic boundary conditions, which can also be written as
\begin{equation}
\label{eq:fw-inv}
  \partial_t u(t,x) + \partial_x f(u(t,x)) + (\I - \partial_{x,P}^2)^{-1} \partial_x u(t,x) = 0,
\end{equation}
where $(\I - \partial_{x,P}^2)^{-1}$ is the inverse of the elliptic
operator $\I - \partial_x^2$ with periodic boundary conditions.
The functionals
\begin{subequations}
\label{eq:fw-invariants}
\begin{align}
\label{eq:fw-invariants-mass}
  J^{\text{FW}}_1(u) &= \int_{\xmin}^{\xmax} u,
  \\
\label{eq:fw-invariants-linear}
  J^{\text{FW}}_2(u) &= \int_{\xmin}^{\xmax} (u - \partial_x^2 u),
  \\
\label{eq:fw-invariants-energy}
  J^{\text{FW}}_3(u) &= \int_{\xmin}^{\xmax} u^2,
\end{align}
\end{subequations}
are invariants of solutions.
In the following, we will construct numerical methods that conserve
all invariants \eqref{eq:fw-invariants}.

\subsubsection{Conservative numerical methods}

The form of the nonlinearity is very similar to the BBM equation. The only
difference is the additional pre-multiplication by the elliptic operator
$\I - \partial_x^2$, which results in different invariants. Hence, basically
the same split form
\begin{equation}
\label{eq:fw-inv-periodic-SBP}
  \partial_t \vec{u}
  + \frac{1}{3} \D1 \vec{u}^2
  + \frac{1}{3} \mat{u} \D1 \vec{u}
  + (\I - \D2)^{-1} \D1 \vec{u}
  =
  \vec{0}
\end{equation}
for $q = 2$ can be used to conserve the invariants \eqref{eq:fw-invariants}.
For general $q \in \N$, similar splittings can be used
\cite[Section~4.5]{ranocha2018thesis}. By applying the relaxation approach
to enforce conservation of $J^{\text{FW}}_3$ in time, we obtain a
fully-discrete scheme that conserves all three invariants \eqref{eq:fw-invariants}.

\begin{theorem}
  If $\D1$ is a periodic first-derivative SBP operator
  with diagonal mass matrix $\M$ and
  $\D2$ is  a periodic second-derivative SBP operator,
  then the semidiscretization \eqref{eq:fw-inv-periodic-SBP} conserves
  the linear invariants \eqref{eq:fw-invariants-mass} and
  \eqref{eq:fw-invariants-linear} of the Fornberg-Whitham equation
  \eqref{eq:fw-dir} with $q = 2$.
  If $\D1$ \& $\D2$ commute, the quadratic invariant
  \eqref{eq:fw-invariants-energy} is also conserved.
\end{theorem}
\begin{proof}
  The first invariant (total mass) is conserved, since
  \begin{equation}
  \begin{aligned}
    \vec{1}^T \M \partial_t \vec{u}
    &=
    - \frac{1}{3} \vec{1}^T \M \D1 \vec{u}^2
    - \frac{1}{3} \vec{1}^T \M \mat{u} \D1 \vec{u}
    - \vec{1}^T \M (\I - \D2)^{-1} \D1 \vec{u}
    \\
    &=
    - \frac{1}{3} \vec{1}^T \M \D1 \vec{u}^2
    - \frac{1}{3} \vec{u}^T \M \D1 \vec{u}
    - \vec{1}^T \M \D1 \vec{u}
    \\
    &=
    - \frac{1}{6} \vec{u}^T \M \D1 \vec{u}
    + \frac{1}{6} \vec{u}^T \D1^T \M \vec{u}
    =
    0,
  \end{aligned}
  \end{equation}
  where we have used that $\M$ is diagonal and have applied
  Lemma~\ref{lem:1-in-left-kernel-of-ImD2inv} in the second line.
  Lemma~\ref{lem:1-M-Di} has been applied in the third line.
  The conservation of the second invariant can be obtained similarly
  by applying Lemma~\ref{lem:1-M-Di}, resulting in
  \begin{equation}
    \vec{1}^T \M (\I - \D2) \partial_t \vec{u}
    =
    - \frac{1}{3} \vec{1}^T \M (\I - \D2) \D1 \vec{u}^2
    - \frac{1}{3} \vec{1}^T \M (\I - \D2) \mat{u} \D1 \vec{u}
    - \vec{1}^T \M \D1 \vec{u}
    =
    0.
  \end{equation}
  To prove conservation of the third invariant (total energy), compute
  \begin{equation}
  \begin{aligned}
    \vec{u}^T \M \partial_t \vec{u}
    &=
    - \frac{1}{3} \vec{u}^T \M \D1 \vec{u}^2
    - \frac{1}{3} \vec{u}^T \M \mat{u} \D1 \vec{u}
    - \vec{u}^T \M (\I - \D2)^{-1} \D1 \vec{u}
    \\
    &=
    - \frac{1}{3} \vec{u}^T \M \D1 \vec{u}^2
    - \frac{1}{3} (\vec{u}^2)^T \M \D1 \vec{u}
    =
    0,
  \end{aligned}
  \end{equation}
  where Lemma~\ref{lem:D1-D2-commuting} has been applied in the
  second line. Here, we need that $\D1$ \& $\D2$ commute.
\end{proof}


To test conservation, we use a smooth traveling wave solution with speed
$c = 1.2$ computed numerically using the Petviashvili method
\cite{petviashvili1976equation} in the periodic domain $[-80, 80]$.
As expected the scheme conserves all linear and the chosen nonlinear invariant
up to roundoff errors.

\subsubsection{Convergence study in space}
\label{sec:fw-periodic-manufactured-convergence}

To study convergence, the method of manufactured solutions is applied,
again with the prescribed solution \eqref{eq:periodic-manufactured}
with periodic boundary conditions.
The results are shown in Figure~\ref{fig:fw-periodic-manufactured-convergence}.

Here, central finite difference methods with order of accuracy $p$ yield an
$\text{EOC}$ between $p - \nicefrac{1}{2}$ and $p$. For other test problems
such as traveling wave profiles, the $\text{EOC}$ is closer to $p$.
The results for wide-stencil and narrow-stencil second-derivative operators
are similar but the narrow-stencil operators result in smaller errors
(smaller by less than an order of magnitude).

In contrast to results for the BBM equation, the choice of the second-derivative
operator does not influence the $\text{EOC}$ for CG methods significantly.
Both wide and narrow-stencil operators $\D2$ yield
$\text{EOC}$ between $p + \nicefrac{1}{2}$ and $p + 1$ for odd polynomial
degrees $p$ and
$\text{EOC} \approx p$ for even $p$.
However, only the wide-stencil operators conserve the quadratic invariant.

A similar observation can be made for nodal discontinuous Galerkin methods.
There, both types of operators $\D2$ yield
$\text{EOC} \approx p+1$ for even polynomial degrees $p$ and
$\text{EOC} \approx p$ for odd $p$.
Again, only the wide-stencil operators conserve the quadratic invariant.

\begin{figure}[htbp]
\centering
  \begin{subfigure}[t]{0.49\textwidth}
  \centering
    \includegraphics[width=\textwidth]{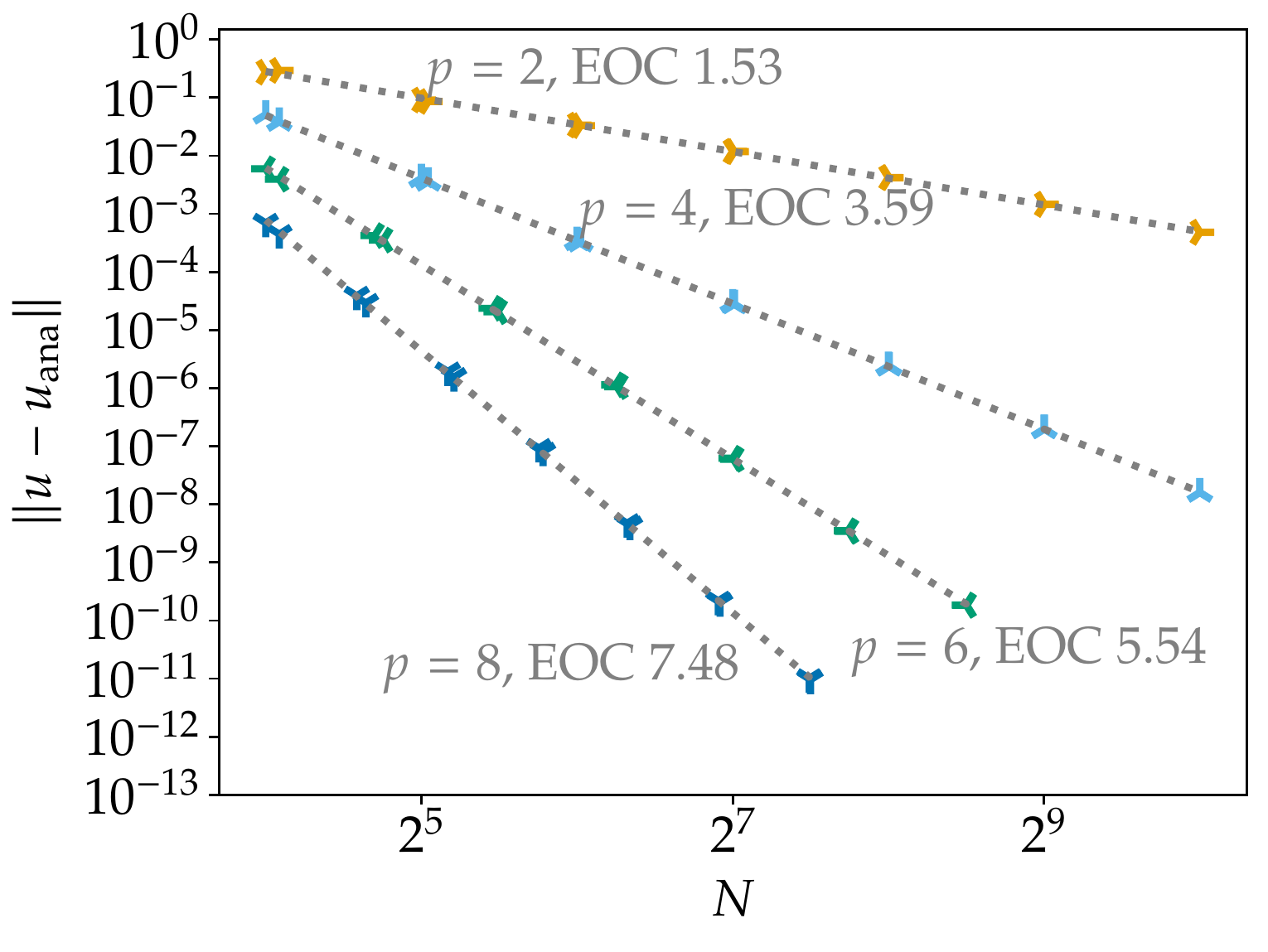}
    \caption{Finite difference methods, wide stencil $\D2 = \D1^2$.}
  \end{subfigure}%
  \hspace*{\fill}
  \begin{subfigure}[t]{0.49\textwidth}
  \centering
    \includegraphics[width=\textwidth]{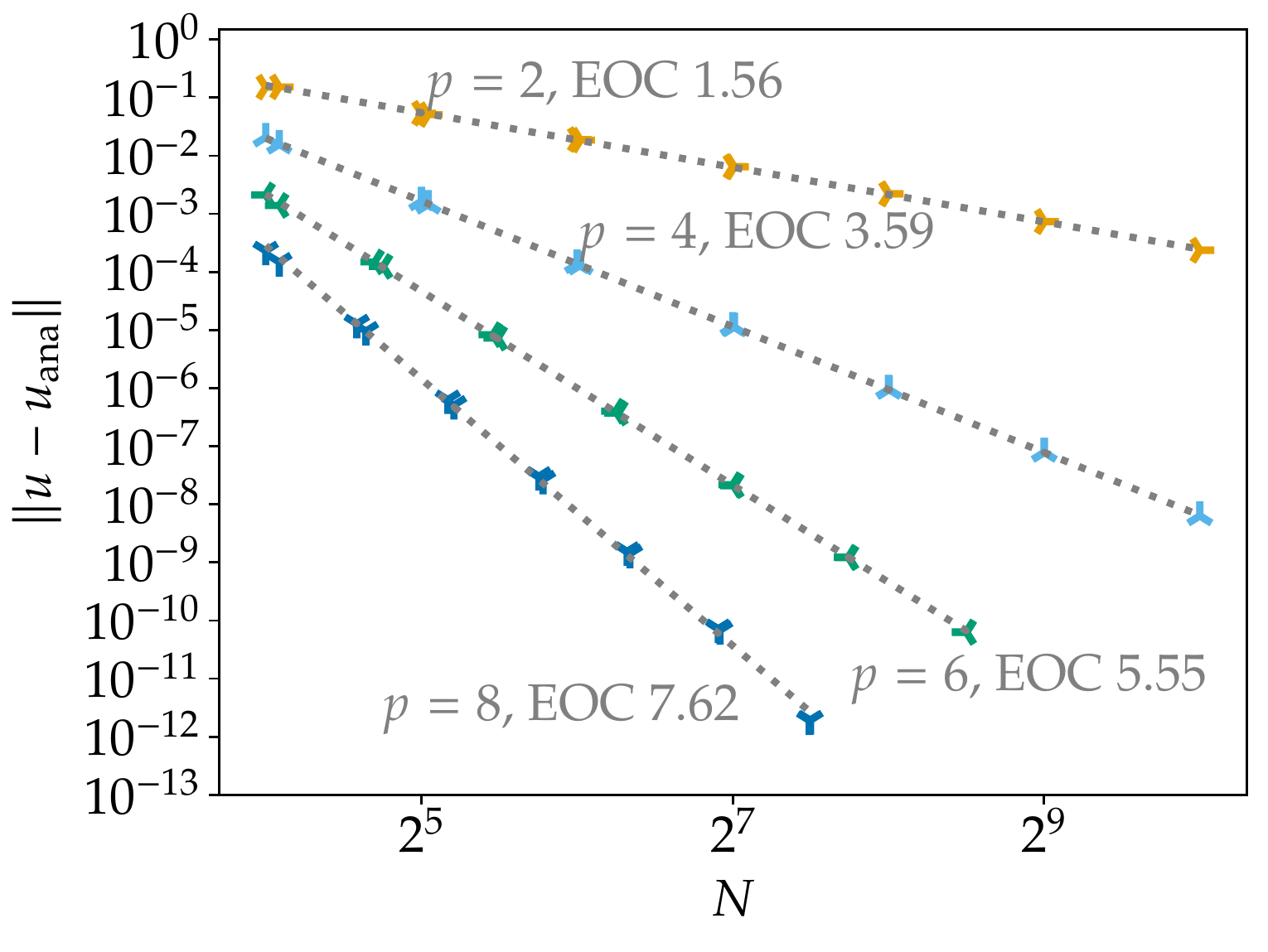}
    \caption{Finite difference methods, narrow stencil $\D2$.}
  \end{subfigure}%
  \\
  \begin{subfigure}[t]{0.8\textwidth}
  \centering
    \includegraphics[width=\textwidth]{figures/Galerkin_legend_p1_p6}
  \end{subfigure}%
  \\
  \begin{subfigure}[t]{0.49\textwidth}
  \centering
    \includegraphics[width=\textwidth]{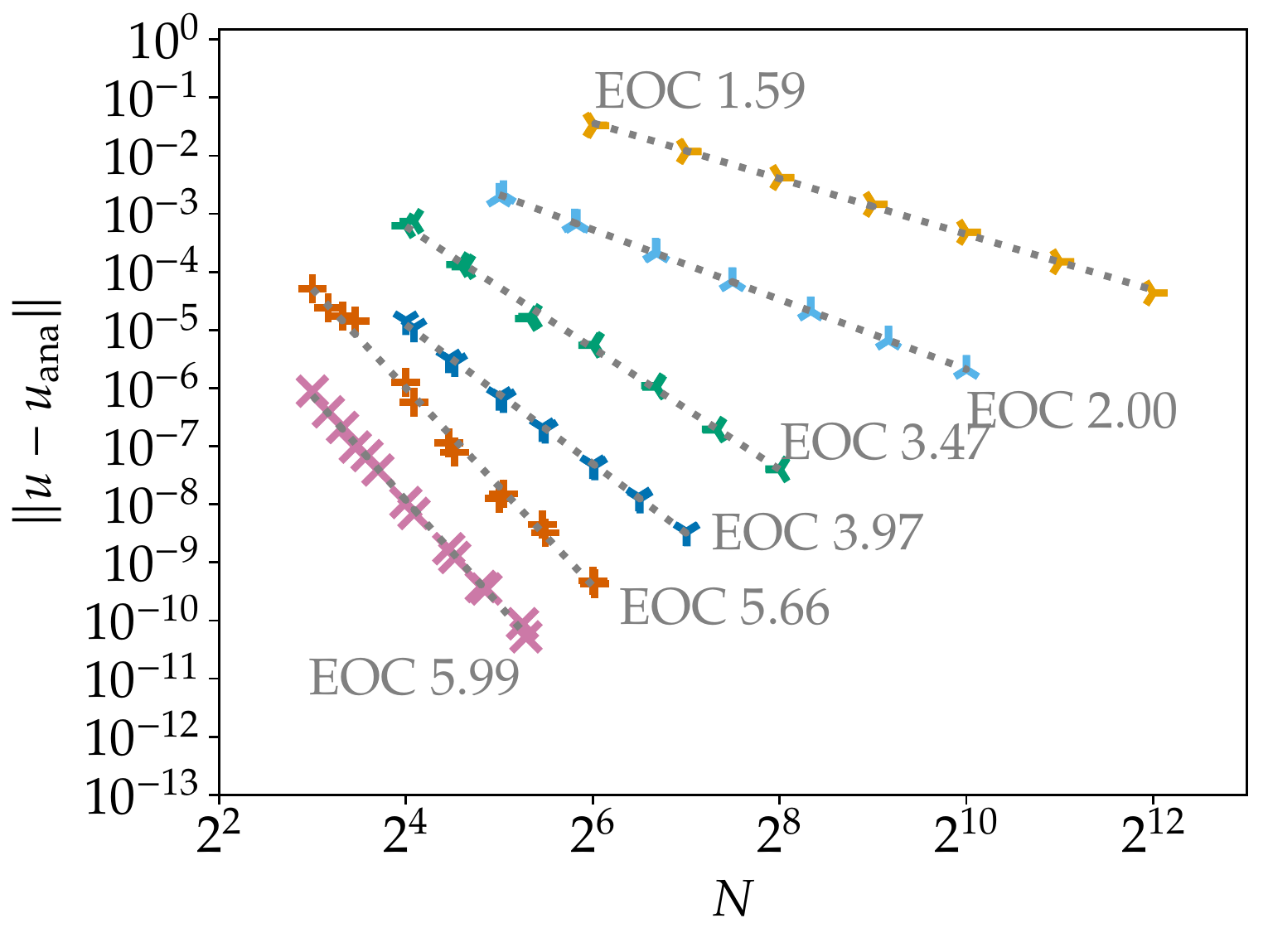}
    \caption{Continuous Galerkin methods, $\D2 = \D1^2$.}
  \end{subfigure}%
  \hspace*{\fill}
  \begin{subfigure}[t]{0.49\textwidth}
  \centering
    \includegraphics[width=\textwidth]{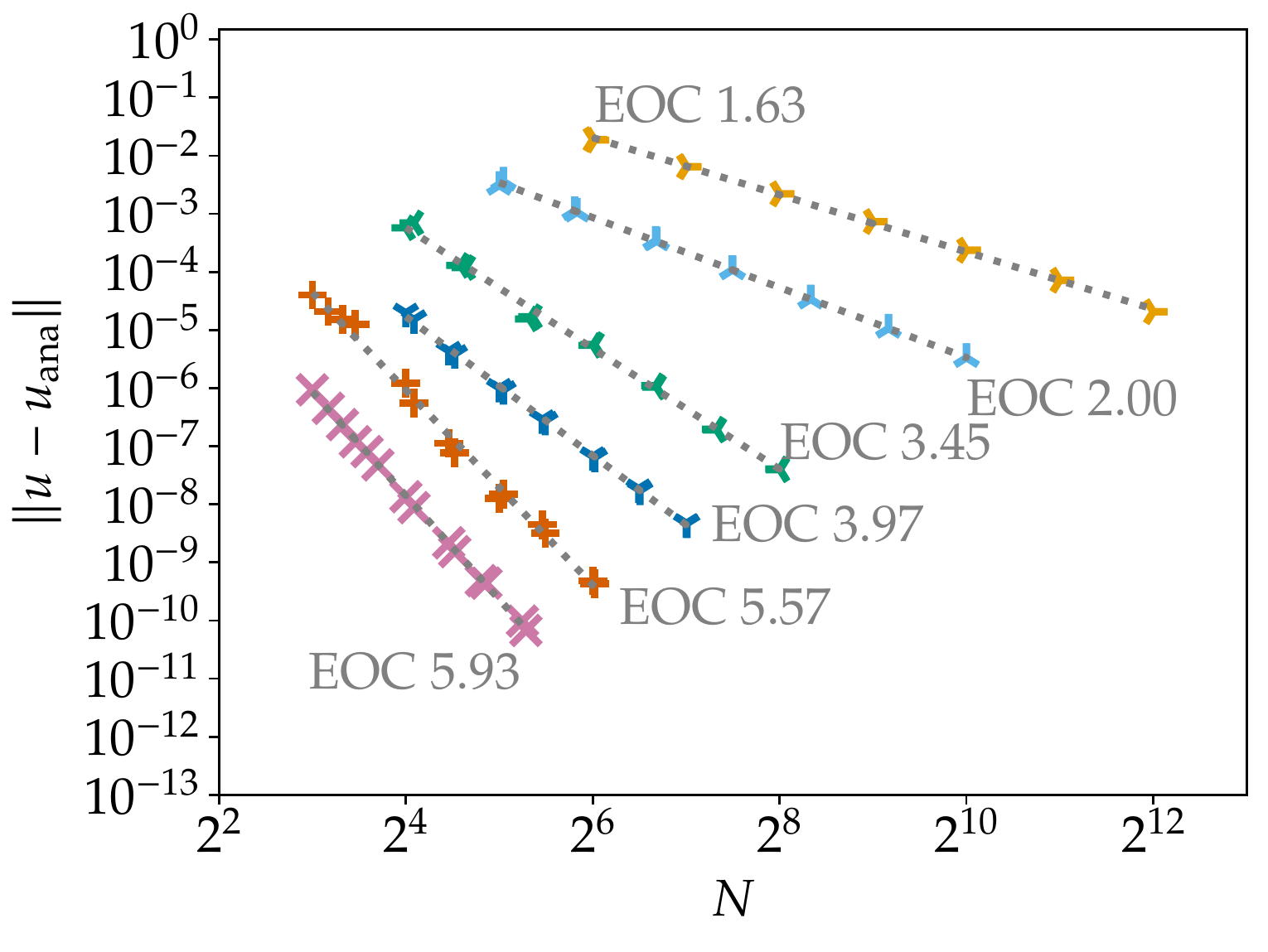}
    \caption{Continuous Galerkin methods, narrow stencil $\D2$.}
  \end{subfigure}%
  \\
  \begin{subfigure}[t]{0.49\textwidth}
  \centering
    \includegraphics[width=\textwidth]{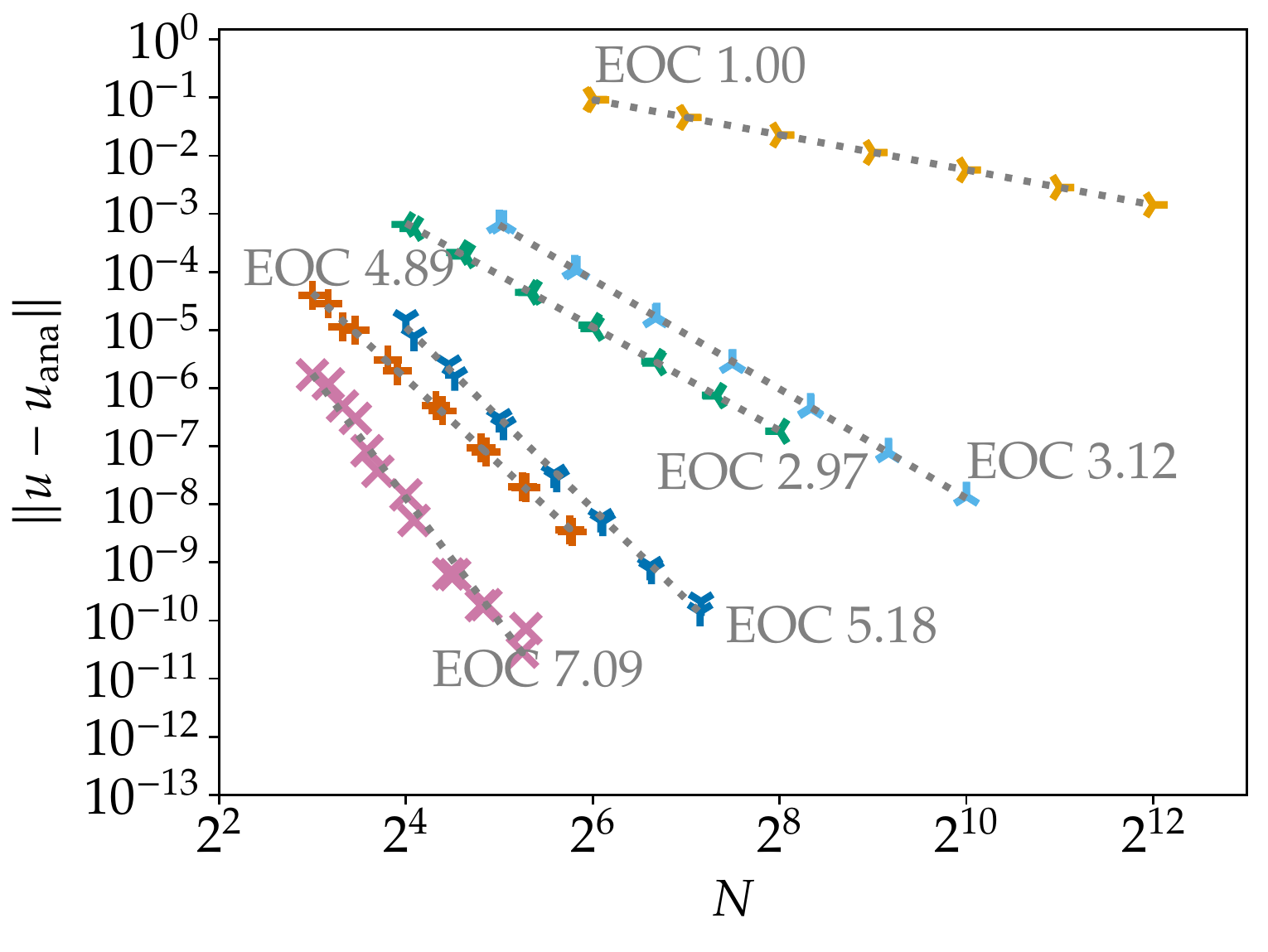}
    \caption{Discontinuous Galerkin methods, $\D2 = \D1^2$.}
  \end{subfigure}%
  \hspace*{\fill}
  \begin{subfigure}[t]{0.49\textwidth}
  \centering
    \includegraphics[width=\textwidth]{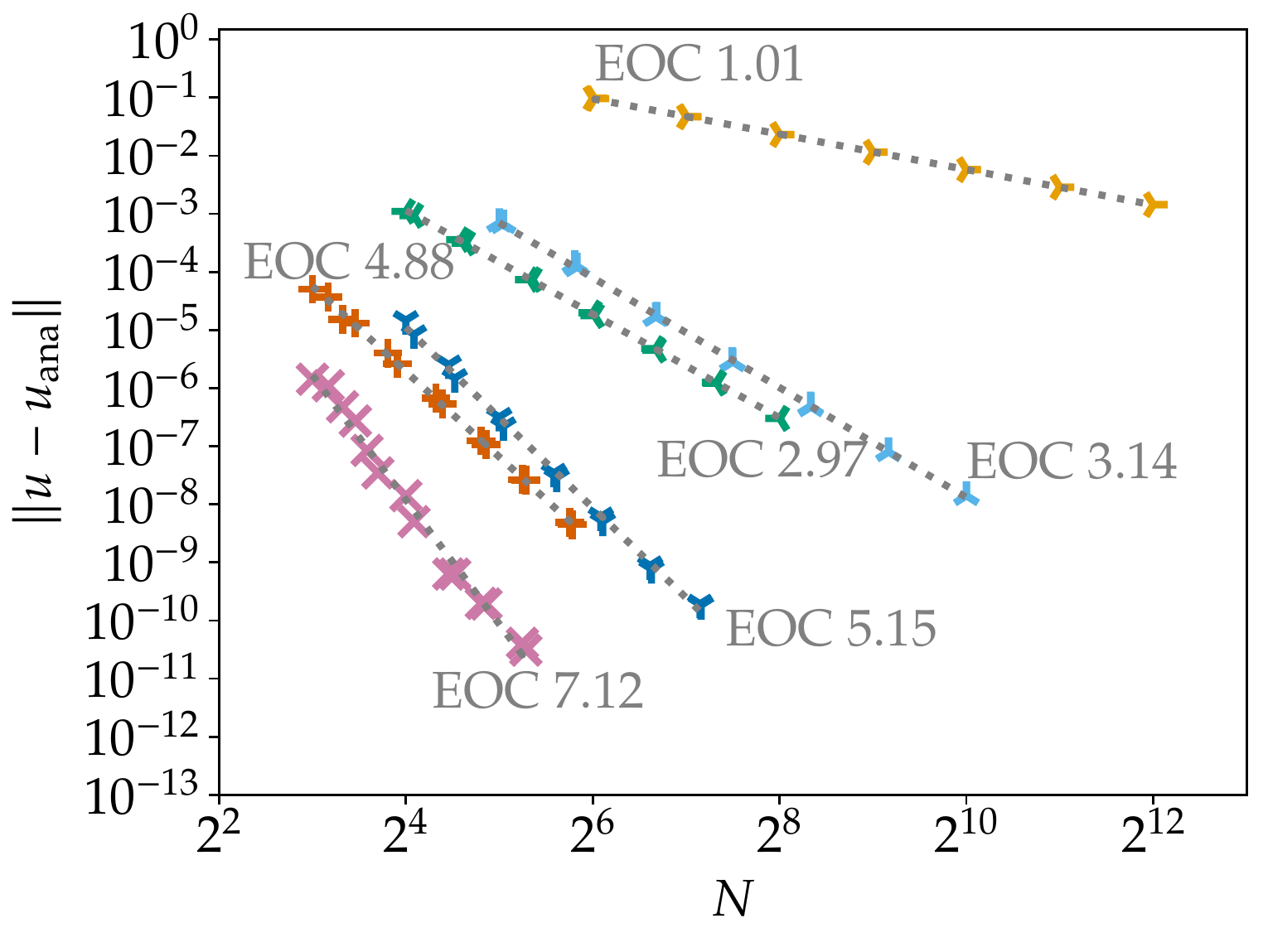}
    \caption{Discontinuous Galerkin methods, $\D2 = \D1{+} \D1{-}$.}
  \end{subfigure}%
  \caption{Convergence results of the spatial semidiscretizations
           \eqref{eq:fw-inv-periodic-SBP}
           for the manufactured solution \eqref{eq:periodic-manufactured}
           of the FW equation.
           All of these semidiscretizations conserve the linear invariants
           \eqref{eq:fw-invariants} of the FW equation \eqref{eq:fw-dir}.
           The FD methods and the Galerkin methods with wide stencil $\D2$
           conserve the quadratic invariant as well.}
  \label{fig:fw-periodic-manufactured-convergence}
\end{figure}

\subsection{Camassa-Holm equation}
\label{sec:ch}

Consider the Camassa-Holm equation \cite{camassa1993integrable}
\begin{equation}
\label{eq:ch-dir}
\begin{aligned}
  (\I - \partial_x^2) \partial_t u(t,x)
  + \partial_x \biggl(
    \frac{3}{2} u(t,x)^2
    - \frac{1}{2} (\partial_x u(t,x))^2
    - u(t,x) \partial_x^2 u(t,x)
  \biggr)
  &= 0,
  \\& t \in (0, T), x \in (\xmin, \xmax),
  \\
  u(0, x) &= u^0(x),
  \\& x \in [\xmin, \xmax],
\end{aligned}
\end{equation}
with periodic boundary conditions, which can also be written as
\begin{equation}
\label{eq:ch-inv}
  \partial_t u
  + (\I - \partial_{x,P}^2)^{-1} \bigl(
    \partial_x u^2
    + u \partial_x u
    - \alpha \partial_x (u \partial_x^2 u)
    - (1 - \alpha) \partial_x^2 (u \partial_x u)
    - (2 \alpha - 1) (\partial_x u) (\partial_x^2 u)
  \bigr)
  = 0,
\end{equation}
where $(\I - \partial_{x,P}^2)^{-1}$ is the inverse of the elliptic
operator $\I - \partial_x^2$ with periodic boundary conditions and
$\alpha \in \R$ is a parameter determining the split form.
The splitting of the quadratic term $\frac{3}{2} \partial_x u^2$ is the
same as for Burgers' equation \eqref{eq:burgers-splitting}. The one-parameter
split form of the third-derivative terms has been constructed using similar
manipulations.

Important invariants of solutions are
\begin{subequations}
\label{eq:ch-invariants}
\begin{align}
\label{eq:ch-invariants-linear}
  J^{\text{CH}}_1(u)
  &= \int_{\xmin}^{\xmax} u,
  \\
\label{eq:ch-invariants-quadratic}
  J^{\text{CH}}_2(u)
  &= \frac{1}{2} \int_{\xmin}^{\xmax} \bigl( u^2 + (\partial_x u)^2 \bigr)
  = \frac{1}{2} \int_{\xmin}^{\xmax} u (\I - \partial_x^2) u,
  \\
\label{eq:ch-invariants-cubic}
  J^{\text{CH}}_3(u)
  &= \int_{\xmin}^{\xmax} \bigl( u^3 + u (\partial_x u)^2 \bigr).
\end{align}
\end{subequations}
In the following, we will construct numerical methods that conserve
the linear \eqref{eq:ch-invariants-linear} and quadratic
\eqref{eq:ch-invariants-quadratic} invariants but not necessarily the
cubic invariant \eqref{eq:ch-invariants-cubic}.

\subsubsection{Conservative numerical methods}

Using the splitting as in \eqref{eq:ch-inv}, semidiscretizations
that conserve the linear and quadratic invariant can be obtained as
\begin{equation}
\label{eq:ch-inv-periodic-SBP}
  \partial_t \vec{u}
  + (\I - \D2{a})^{-1} \!\!\!\! \bigl(\!\!\!\!
    \D1 \vec{u}^2
    + \vec{u} \D1 \vec{u}
    - \alpha \D1 (\vec{u} \D2{b} \vec{u})
    - (1 - \alpha) \D2{b} (\vec{u} \D1 \vec{u})
    - (2 \alpha - 1) (\D1 \vec{u}) \D2{b} \vec{u}
  \!\!\!\bigr)\!
  =
  \vec{0}.
\end{equation}
\begin{theorem}
\label{thm:ch-inv-periodic-SBP}
  If $\D1$ is a periodic first-derivative SBP operator
  with diagonal mass matrix $\M$
  and $\D2{a}, \D2{b}$ are periodic second-derivative SBP operator,
  then the semidiscretization \eqref{eq:ch-inv-periodic-SBP}
  conserves the quadratic invariant \eqref{eq:ch-invariants-quadratic}
  of \eqref{eq:ch-dir}.
  If $\D1$ and $\D2{b}$ commute or $\alpha = \nicefrac{1}{2}$,
  the linear invariant \eqref{eq:ch-invariants-linear} is conserved
  as well.
\end{theorem}
\begin{proof}
  The rate of change of the linear invariant
  \eqref{eq:ch-invariants-linear} is
  \begin{equation}
  \begin{aligned}
    &\quad
    \od{}{t} J^{\text{CH}}_1(\vec{u})
    =
    \od{}{t} \vec{1}^T \M \vec{u}
    =
    \vec{1}^T \M \partial_t \vec{u}
    =
    - \vec{1}^T \M (\I - \D2{a})^{-1} \Bigl(
      (\I - \D2{a}) \partial_t \vec{u}
    \Bigr)
    \\
    &=
    - \vec{1}^T \M \left(
      \D1 \vec{u}^2
      + \vec{u} \D1 \vec{u}
      - \alpha \D1 (\vec{u} \D2{b} \vec{u})
      - (1 - \alpha) \D2{b} (\vec{u} \D1 \vec{u})
      - (2 \alpha - 1) (\D1 \vec{u}) \D2{b} \vec{u}
    \right)
    \\
    &=
    (2 \alpha - 1) \vec{u}^T \D1^T \M \D2{b} \vec{u},
  \end{aligned}
  \end{equation}
  where Lemma~\ref{lem:1-in-left-kernel-of-ImD2inv} has been used in the
  second line and Lemma~\ref{lem:1-M-Di} has been used in the last step.
  If $\D1$ and $\D2{b}$ commute,
  \begin{equation}
    \od{}{t} J^{\text{CH}}_1(\vec{u})
    =
    (2 \alpha - 1) \frac{1}{2} \vec{u}^T \D1^T \D2{b}^T \M \vec{u}
    - (2 \alpha - 1) \frac{1}{2} \vec{u}^T \M \D1 \D2{b} \vec{u}
    =
    0.
  \end{equation}

  Since $\I - \D2{a}$ is a symmetric operator, the semidiscrete rate of
  change of the quadratic invariant \eqref{eq:ch-invariants-quadratic}
  is
  \begin{align*}
  \stepcounter{equation}\tag{\theequation}
    &\quad
    \od{}{t} J^{\text{CH}}_2(\vec{u})
    =
    \frac{1}{2} \od{}{t} \vec{u}^T \M (\I - \D2{a}) \vec{u}
    =
    \vec{u}^T \M (\I - \D2{a}) \partial_t \vec{u}
    \\
    &=
    - \vec{u}^T \M \left(
      \D1 \vec{u}^2
      + \vec{u} \D1 \vec{u}
      - \alpha \D1 (\vec{u} \D2{b} \vec{u})
      - (1 - \alpha) \D2{b} (\vec{u} \D1 \vec{u})
      - (2 \alpha - 1) (\D1 \vec{u}) \D2{b} \vec{u}
    \right)
    \\
    &=
    \alpha \vec{u}^T \M \left(
      \D1 (\vec{u} \D2{b} \vec{u})
      + (\D1 \vec{u}) \D2{b} \vec{u}
    \right)
    + (1 - \alpha) \vec{u}^T \M \left(
      \D2{b} (\vec{u} \D1 \vec{u})
      - (\D1 \vec{u}) \D2{b} \vec{u}
    \right)
    =
    0.
    \qedhere
  \end{align*}
\end{proof}

\begin{remark}
  A dissipative LDG method based on exact integration instead of
  split forms and equivalents of $\D2{a} = \D2{b} = \D1{-} \D1{+}$
  or $\D2{a} = \D2{b} = \D1{+} \D1{-}$ has been proposed in
  \cite{xu2008local}.
  The split form semidiscretization \eqref{eq:ch-inv-periodic-SBP}
  with $\alpha = 1$ and $\D2{a} = \D2{b} = \D1^2$ has been used
  in \cite{liu2016invariantFD} for second order FD methods and
  in \cite{liu2016invariantDG} for DG methods.
  The same split form with $\D2{a} = \D2{b}$ has been used
  in \cite{hong2019linear} for Fourier collocation methods.
  Without applying a splitting, a Fourier collocation method
  conserving the cubic invariant \eqref{eq:ch-invariants-cubic}
  has been used in \cite{cai2016geometric}.
\end{remark}

\begin{remark}
  Similarly to the FW equation and in contrast to the BBM equation, an additional
  restriction on the first- and second-derivative operators arises for the CH equation: They need to commute to conserve the quadratic invariant
  \eqref{eq:ch-invariants-quadratic} unless the splitting parameter is chosen
  as $\alpha = \nicefrac{1}{2}$. The existence of such a splitting and the
  potential possibility to use different second-derivative operators provides
  interesting possibilities. In preliminary numerical studies, using
  $\D2{a} = \D2{b}$ seems to be a good choice.
\end{remark}

To verify the conservation properties of the semidiscretization \eqref{eq:ch-inv-periodic-SBP},
we used a smooth traveling wave solution with speed $c = 1.2$ computed numerically
using the Petviashvili method in the periodic domain $[-40, 40]$.
This traveling wave solution has been computed for the PDE
\begin{equation}
  (\I - \partial_x^2) \partial_t u(t,x)
  + 2 \kappa \partial_x u(t,x)
  + \partial_x \biggl(
    \frac{3}{2} u(t,x)^2
    - \frac{1}{2} (\partial_x u(t,x))^2
    - u(t,x) \partial_x^2 u(t,x)
  \biggr)
  = 0,
\end{equation}
which can be transformed to a solution of the CH equation (with $\kappa = 0$)
by the transformation
\begin{equation}
  x \to x + \kappa t,
  \quad
  u \to u + \kappa.
\end{equation}

\subsubsection{Convergence study in space}
\label{sec:ch-periodic-manufactured-convergence}

For the following convergence study, the method of manufactured solutions
is applied to \eqref{eq:periodic-manufactured} with periodic boundary conditions.
The results are shown in Figure~\ref{fig:ch-periodic-manufactured-convergence}.

Similarly to the BBM equation, central finite difference methods with order
of accuracy $p$ yield an $\text{EOC} \approx p$. The results for wide-stencil
and narrow-stencil second-derivative operators are again similar and the
narrow-stencil operators result in smaller errors (less than an order of
magnitude).

For CG methods, the choice of the second-derivative operator does not influence
the $\text{EOC}$ significantly, similarly to the FW equation and in contrast
to the BBM equation. Both wide and narrow-stencil operators $\D2$ yield
$\text{EOC} \approx p+1$ for odd polynomial degrees $p$ and
$\text{EOC} \approx p$ for even $p$.
In contrast to the other examples discussed before, the error depends on the
parity of the number of elements for odd polynomial degrees and the wide
stencil second-derivative operator $\D2 = \D1^2$. In these cases, the error
is smaller if an odd number of elements is used. By just adding one element
to go from even $N$ to odd $N$, the error can be reduced up to an order of
magnitude. The method using $p = 3$ results in approximately the same error
as the one for $p = 4$ if $N$ is odd while its error is up to an order of
magnitude bigger for even $N$.
Such a behavior cannot be observed for the narrow-stencil operator or even
polynomial degrees.

Similar observations can be made for nodal discontinuous Galerkin methods.
There, both types of operators $\D2$ yield
$\text{EOC} \approx p+1$ for even polynomial degrees $p$ and
$\text{EOC} \approx p$ for odd $p$.
The only exception to this rule is the narrow-stencil second-derivative
operator $\D2 = \D1{+} \D1{-}$ with $p = 1$, which doesn't converge.
Since this phenomenon occurs only for this specific parameter combination
and convergence can be obtained, \eg for $\alpha = 1$, it is not studied
in detail here.
As for CG methods, there is a dependence of the error on the parity of
the number of elements. For DG methods, this dependence manifests only
for wide-stencil operators $\D2 = \D1^2$ and even polynomial degrees
(while it occurs for odd polynomial degrees for CG methods). Again, the error
is up to an order of magnitude smaller for odd $N$.

\begin{figure}[htbp]
\centering
  \begin{subfigure}[t]{0.49\textwidth}
  \centering
    \includegraphics[width=\textwidth]{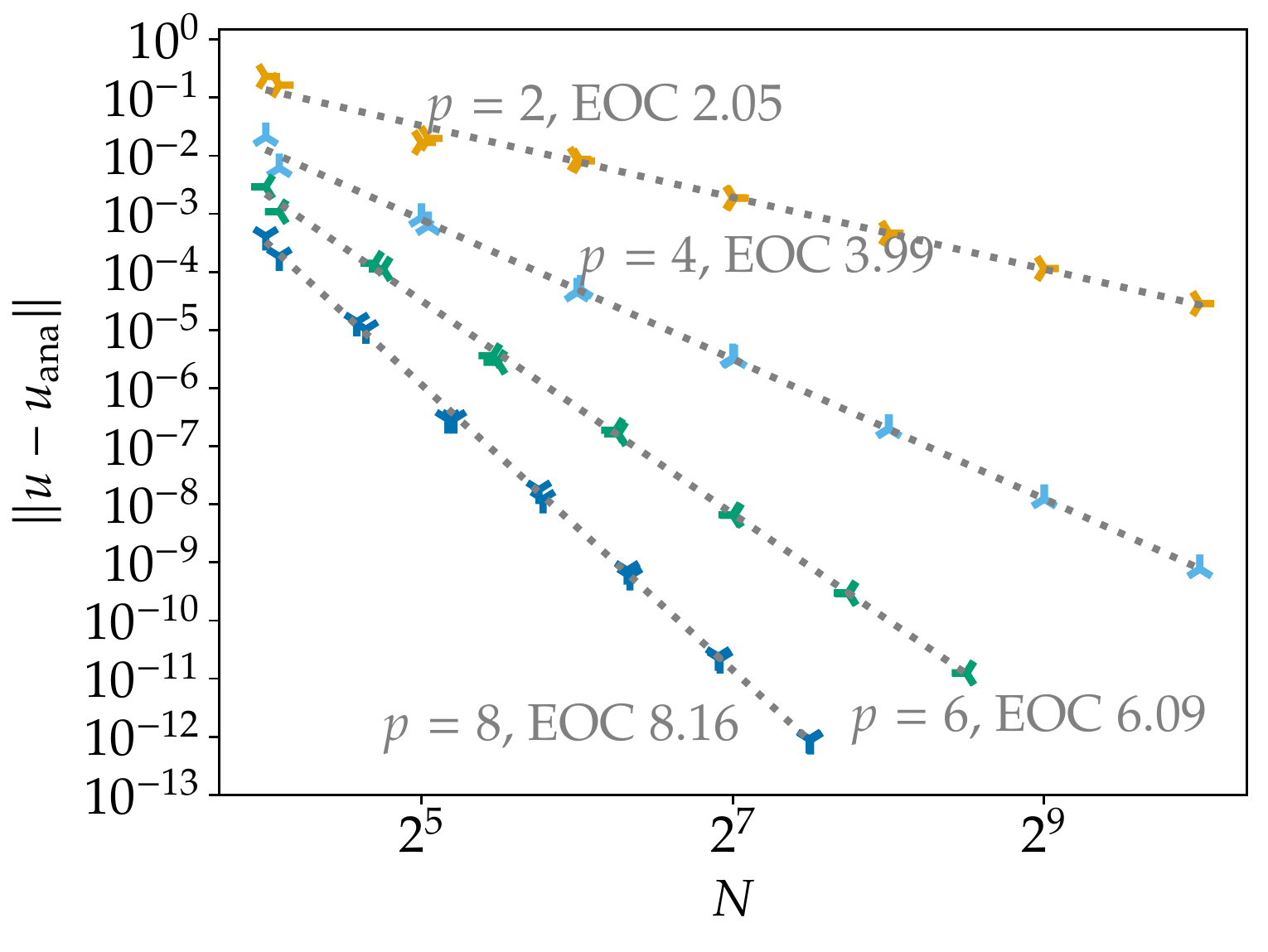}
    \caption{Finite difference methods, wide stencil $\D2 = \D1^2$.}
  \end{subfigure}%
  \hspace*{\fill}
  \begin{subfigure}[t]{0.49\textwidth}
  \centering
    \includegraphics[width=\textwidth]{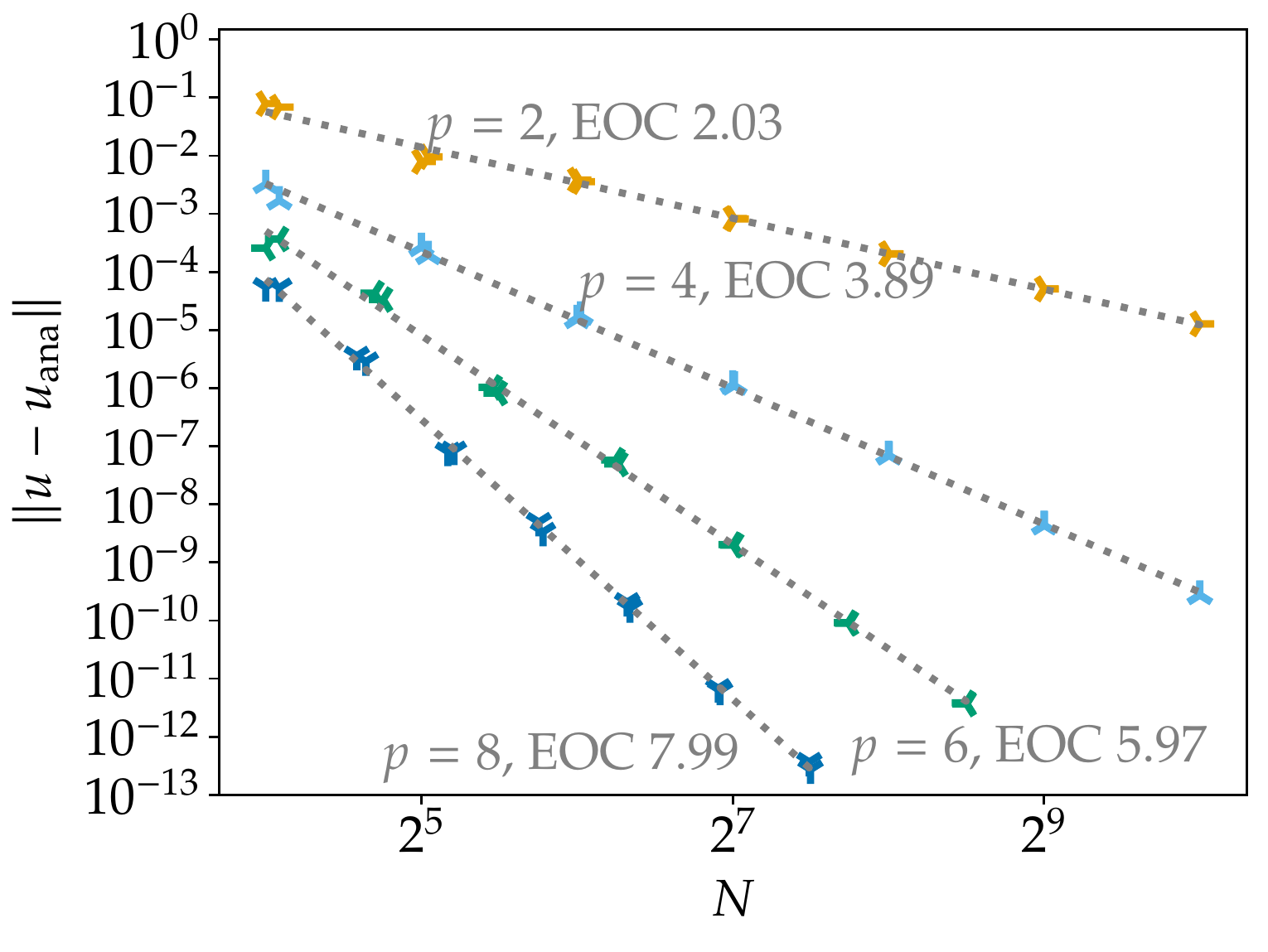}
    \caption{Finite difference methods, narrow stencil $\D2$.}
  \end{subfigure}%
  \\
  \begin{subfigure}[t]{0.8\textwidth}
  \centering
    \includegraphics[width=\textwidth]{figures/Galerkin_legend_p1_p6}
  \end{subfigure}%
  \\
  \begin{subfigure}[t]{0.49\textwidth}
  \centering
    \includegraphics[width=\textwidth]{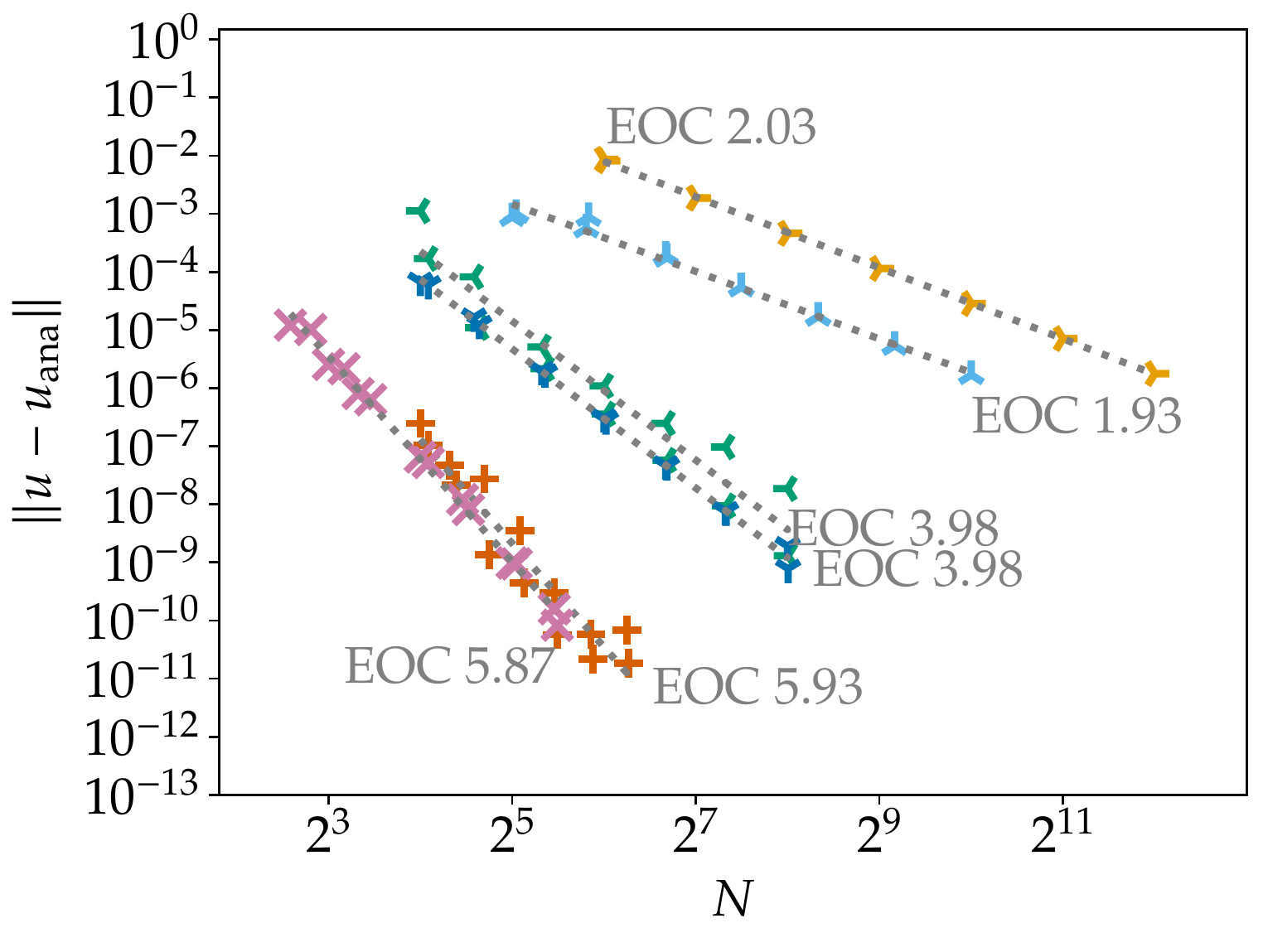}
    \caption{Continuous Galerkin methods, $\D2 = \D1^2$.}
  \end{subfigure}%
  \hspace*{\fill}
  \begin{subfigure}[t]{0.49\textwidth}
  \centering
    \includegraphics[width=\textwidth]{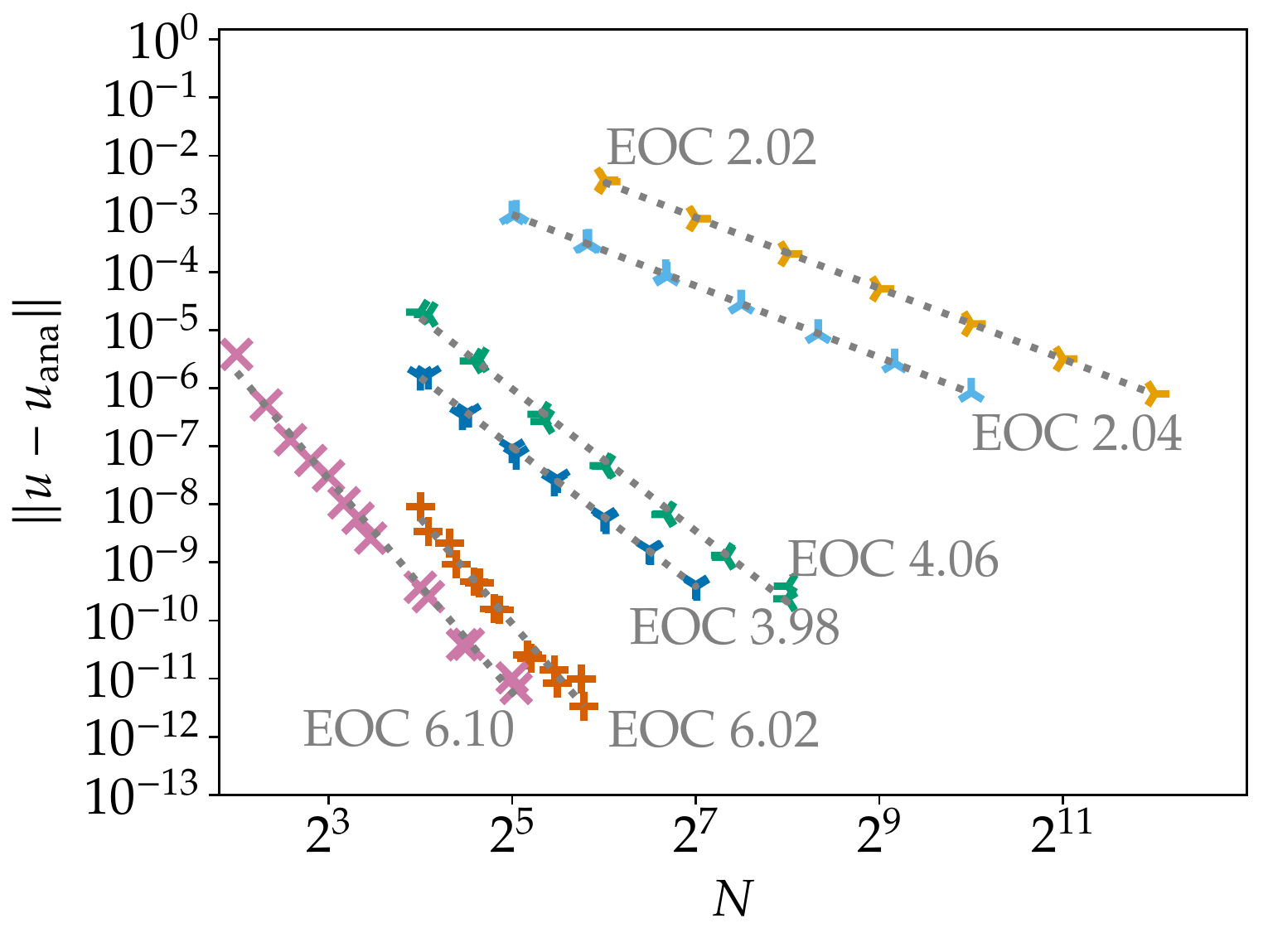}
    \caption{Continuous Galerkin methods, narrow stencil $\D2$.}
  \end{subfigure}%
  \\
  \begin{subfigure}[t]{0.49\textwidth}
  \centering
    \includegraphics[width=\textwidth]{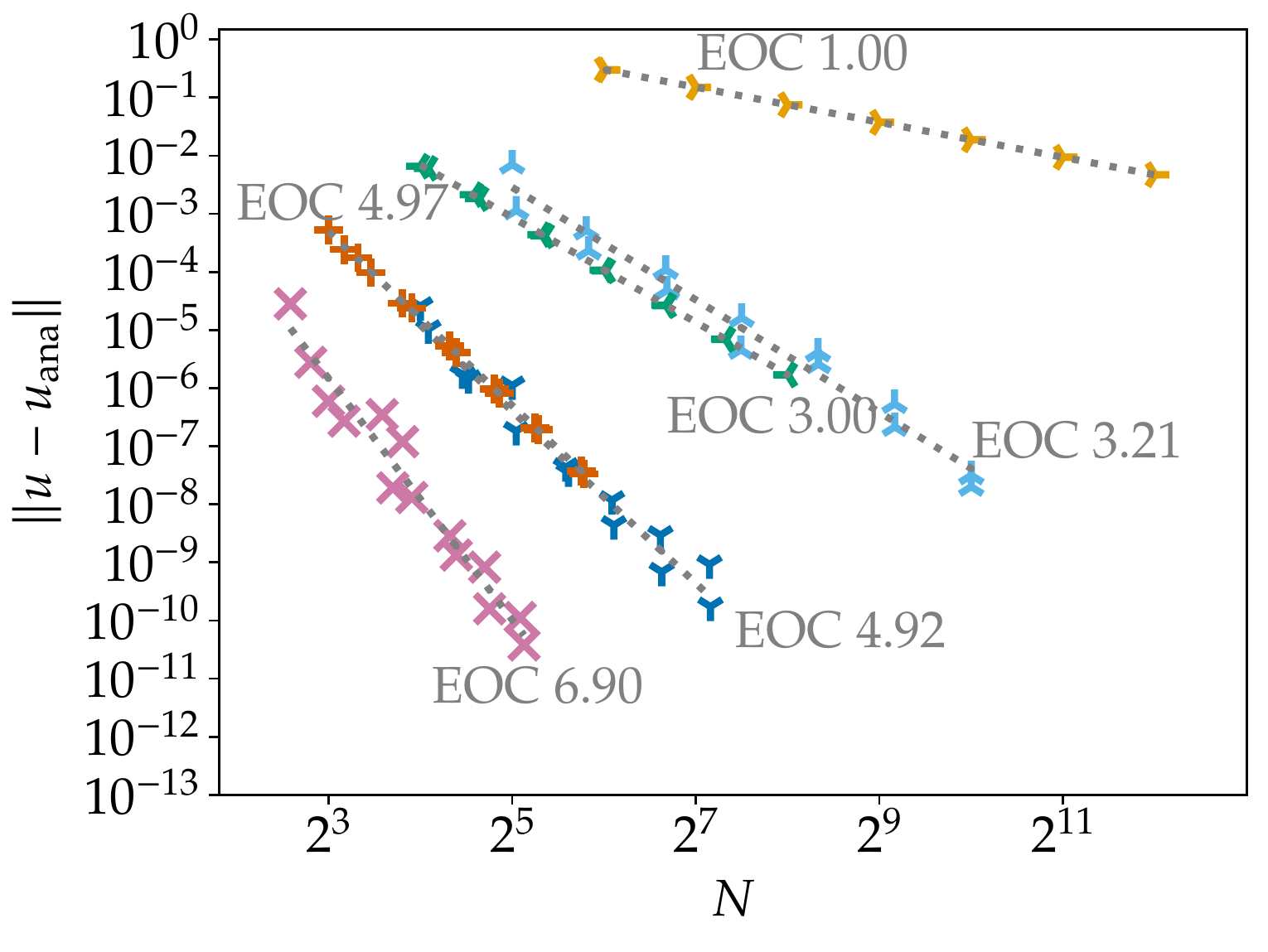}
    \caption{Discontinuous Galerkin methods, $\D2 = \D1^2$.}
  \end{subfigure}%
  \hspace*{\fill}
  \begin{subfigure}[t]{0.49\textwidth}
  \centering
    \includegraphics[width=\textwidth]{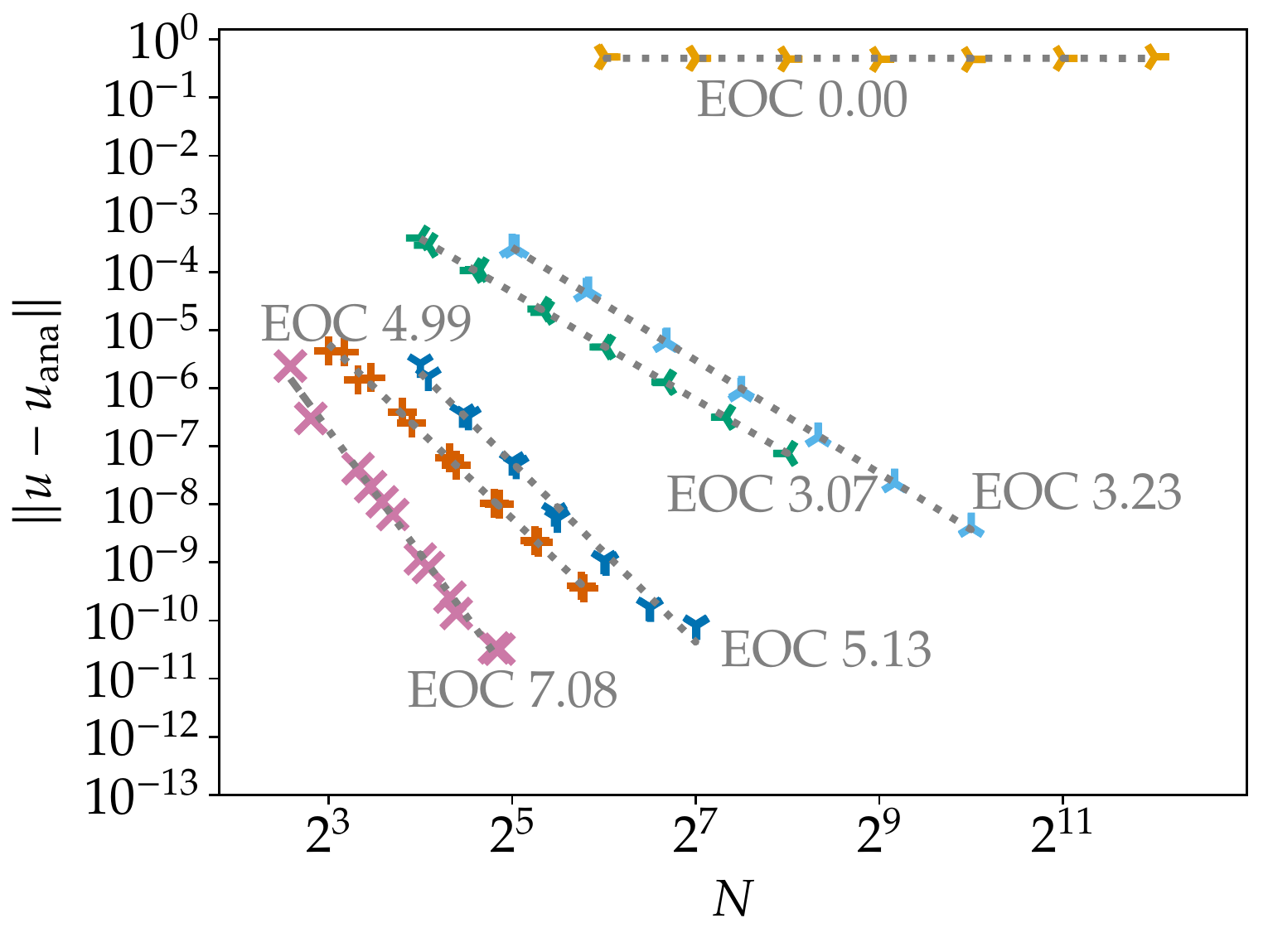}
    \caption{Discontinuous Galerkin methods, $\D2 = \D1{+} \D1{-}$.}
  \end{subfigure}%
  \caption{Convergence results of the spatial semidiscretizations
           \eqref{eq:ch-inv-periodic-SBP}
           with $\alpha = \nicefrac{1}{2}$ and $\D2{a} = \D2{b} = \D2$
           for the manufactured solution \eqref{eq:periodic-manufactured}
           of the CH equation.
           All of these semidiscretizations conserve the linear and quadratic
           invariants \eqref{eq:ch-invariants} of the CH equation
           \eqref{eq:ch-dir}.}
  \label{fig:ch-periodic-manufactured-convergence}
\end{figure}

\subsection{Degasperis-Procesi equation}
\label{sec:dp}

Consider the Degasperis-Procesi equation \cite{degasperis2002new}
\begin{equation}
\label{eq:dp-dir}
\begin{aligned}
  (\I - \partial_x^2) \partial_t u(t,x)
  + 4 \partial_x f(u(t,x))
  - \partial_x^3 f(u(t,x))
  &= 0,
  && t \in (0, T), x \in (\xmin, \xmax),
  \\
  u(0, x) &= u^0(x),
  && x \in [\xmin, \xmax],
  \\
  f(u) &= \frac{u^2}{2},
\end{aligned}
\end{equation}
with periodic boundary conditions, which can also be written as
\begin{equation}
\label{eq:dp-inv}
  \partial_t u(t,x)
  + (\I - \partial_{x,P}^2)^{-1} (4\I - \partial_x^2) \partial_x f(u(t,x))
  = 0,
\end{equation}
where $(\I - \partial_{x,P}^2)^{-1}$ is the inverse of the elliptic
operator $\I - \partial_x^2$ with periodic boundary conditions.
The functionals
\begin{subequations}
\label{eq:dp-invariants}
\begin{align}
  \label{eq:dp-invariants-linear}
  J^{\text{DP}}_1(u)
  &= \int_{\xmin}^{\xmax} ( u - \partial_x^2 u ),
  \\
  \label{eq:dp-invariants-quadratic}
  J^{\text{DP}}_2(u)
  &= \frac{1}{2} \int_{\xmin}^{\xmax} \bigl( (u - \partial_x^2 u) v \bigr) \bigr),
  \qquad v = (4\I - \partial_{x,P}^2)^{-1} u,
  \\
  \label{eq:dp-invariants-cubic}
  J^{\text{DP}}_3(u)
  &= \int_{\xmin}^{\xmax} u^3,
\end{align}
\end{subequations}
are invariants of solutions.
We would like to emphasize that \eqref{eq:ch-dir} and \eqref{eq:dp-dir} can
be written using the same linear and nonlinear terms multiplied by different
constant coefficients. These result in different invariants and other
split forms available/necessary for conservative methods.
In the following, we will construct numerical methods that conserve
the linear \eqref{eq:dp-invariants-linear} and quadratic
\eqref{eq:dp-invariants-quadratic} invariants but not necessarily the
cubic invariant \eqref{eq:dp-invariants-cubic}.

\subsubsection{Conservative numerical methods}

The rate of change of the quadratic invariant \eqref{eq:dp-invariants-quadratic}
results basically in the same integral terms as the energy rate of Burgers'
equation \eqref{eq:burgers-chain-rule}. Hence, the same kind of splitting
can be used to obtain semidiscretizations
\begin{equation}
\label{eq:dp-inv-periodic-SBP}
  \partial_t \vec{u}
  + \frac{1}{3} (\I - \D2)^{-1} (4\I - \D2) \left(
    \D1 \vec{u}^2 + \vec{u} \D1 \vec{u}
  \right)
  =
  \vec{0}
\end{equation}
that conserve the linear and quadratic invariant.
\begin{theorem}
\label{thm:dp-inv-periodic-SBP}
  If $\D1$ is a periodic first-derivative SBP operator
  with diagonal mass matrix $\M$
  and $\D2$ is a periodic second-derivative SBP operator,
  then the semidiscretization
  \eqref{eq:dp-inv-periodic-SBP}
  conserves the invariants
  \eqref{eq:dp-invariants-linear} and
  \eqref{eq:dp-invariants-quadratic} of
  \eqref{eq:dp-dir}.
\end{theorem}
\begin{proof}
  The linear invariant \eqref{eq:dp-invariants-linear} is conserved
  since
  \begin{multline}
    \od{}{t} J^{\text{DP}}_1(\vec{u})
    =
    \od{}{t} \vec{1}^T \M (\I - \D2) \vec{u}
    =
    \vec{1}^T \M (\I - \D2) \partial_t \vec{u}
    =
    - \frac{1}{3} \vec{1}^T \M (4\I - \D2) \left(
      \D1 \vec{u}^2 + \vec{u} \D1 \vec{u}
    \right)
    \\
    =
    - \frac{1}{3} \vec{1}^T \M (4\I - \D2) \left(
      \D1 \vec{u}^2 + \vec{u} \D1 \vec{u}
    \right)
    = 0,
  \end{multline}
  where Lemma~\ref{lem:1-in-left-kernel-of-ImD2inv} and
  Lemma~\ref{lem:1-M-Di} have been used.

  Since $\I - \D2$ and $4\I - \D2$ are commuting symmetric operators,
  the semidiscrete rate of change of the quadratic invariant
  \eqref{eq:dp-invariants-quadratic} is
  \begin{equation}
  \begin{multlined}
    \frac{1}{2} \od{}{t} \vec{u}^T (4\I - \D2)^{-T} \M (\I - \D2) \vec{u}
    =
    \vec{u}^T (4\I - \D2)^{-T} \M (\I - \D2) \partial_t \vec{u}
    \\
    =
    - \frac{1}{3} \vec{u}^T (4\I - \D2)^{-T} \M (4\I - \D2) (\D1 \vec{u}^2 + \vec{u} \D1 \vec{u})
    =
    - \frac{1}{3} \vec{u}^T \M (\D1 \vec{u}^2 + \vec{u} \D1 \vec{u})
    =
    0.
  \end{multlined}
  \qedhere
  \end{equation}
\end{proof}

\begin{remark}
  The split form discretization \eqref{eq:dp-inv-periodic-SBP}
  is used in \cite{xia2014fourier} for Fourier collocation methods.
  Substituting the split form in \eqref{eq:dp-inv-periodic-SBP} by
  the conservative form results in a semidiscretization that conserves
  both the linear invariant \eqref{eq:dp-invariants-linear} and the
  cubic invariant \eqref{eq:dp-invariants-cubic}. This has been
  used for Fourier collocation methods in \cite{cai2016geometric}.
\end{remark}

\begin{remark}
  Conservation of the quadratic invariant
  \eqref{eq:dp-invariants-quadratic} yields the estimate
  $\| \vec{u}(t) \|_M^2 \le 4 \| \vec{u}^0 \|_M^2$, \cf
  \cite{xia2014fourier}.
  Indeed, setting $\vec{v} = (4\I - \D2)^{-1} \vec{u}$,
  \begin{equation}
    \| \vec{u} \|_{\M}^2
    =
    \| (4\I - \D2) \vec{v} \|_{\M}^2
    =
    16 \| \vec{v} \|_{\M}^2
    + 8 \| \vec{v} \|_{\A2}^2
    + \| \D2 \vec{v} \|_{\M}^2.
  \end{equation}
  Using
  \begin{equation}
    J^{\text{DP}}_2(\vec{u})
    =
    \vec{v}^T \M (\I - \D2) \vec{u}
    =
    \vec{v}^T \M (\I - \D2) (4\I - \D2) \vec{v}
    =
    4 \| \vec{v} \|_{\M}^2
    + 5 \| \vec{v} \|_{\A2}^2
    + \| \D2 \vec{v} \|_{\M}^2.
  \end{equation}
  yields the bounds
  $\| \vec{u} \|_{\M}^2 \le 4 J^{\text{DP}}_2(\vec{u})$
  and $\| \vec{u} \|_{\M}^2 \ge J^{\text{DP}}_2(\vec{u})$.
  Hence, $\| \vec{u}(t) \|_M^2
  \le 4 J^{\text{DP}}_2(\vec{u}(t))
  = 4 J^{\text{DP}}_2(\vec{u}^0)
  \le 4 \| \vec{u}^0 \|_M^2$.
\end{remark}

A smooth traveling wave solution with speed $c = 1.2$ computed numerically
using the Petviashvili method in the periodic domain $[-40, 40]$ was used to verify
the conservation properties of the semidiscretization
\eqref{eq:ch-inv-periodic-SBP}.
This traveling wave solution has been computed for the PDE
\begin{equation}
  (\I - \partial_x^2) \partial_t u(t,x)
  + 3 \kappa \partial_x u(t,x)
  + 2 \partial_x u(t,x)^2
  - \frac{1}{2} \partial_x^3 u(t,x)^2
  = 0
\end{equation}
which can be transformed to a solution of the DP equation (with $\kappa = 0$)
by the transformation
\begin{equation}
  x \to x + \kappa t,
  \quad
  u \to u + \kappa.
\end{equation}

\subsubsection{Convergence study in space}
\label{sec:dp-periodic-manufactured-convergence}

For the following convergence study, the method of manufactured solutions
is applied to \eqref{eq:periodic-manufactured} with periodic boundary conditions.
The results are shown in Figure~\ref{fig:dp-periodic-manufactured-convergence}.

Similarly to the FW equation, central finite difference methods with order of
accuracy $p$ yield an $\text{EOC}$ between $p - \nicefrac{1}{2}$ and $p$.
For other test problems such as traveling wave profiles, the $\text{EOC}$ is
closer to $p$.
The results for wide-stencil and narrow-stencil second-derivative operators
are similar but the narrow-stencil operators result again in smaller errors
(up to an order of magnitude).

As for the FW and CH equations and in contrast to the BBM equation, the choice
of the second-derivative operator does not influence the $\text{EOC}$
significantly for nodal continuous Galerkin methods.
Both wide and narrow-stencil operators $\D2$ yield an
$\text{EOC}$ between $p + \nicefrac{1}{2}$ and $p + 1$ for odd polynomial degrees $p$ and
$\text{EOC} \approx p$ for even $p$. For traveling wave solutions and odd
polynomial degrees, the $\text{EOC}$ is closer to $p + 1$.

Similar observations can be made for nodal discontinuous Galerkin methods.
There, both types of operators $\D2$ yield
$\text{EOC} \approx p+1$ for even polynomial degrees $p$ and
$\text{EOC} \approx p$ for odd $p$.

\begin{figure}[htbp]
\centering
  \begin{subfigure}[t]{0.49\textwidth}
  \centering
    \includegraphics[width=\textwidth]{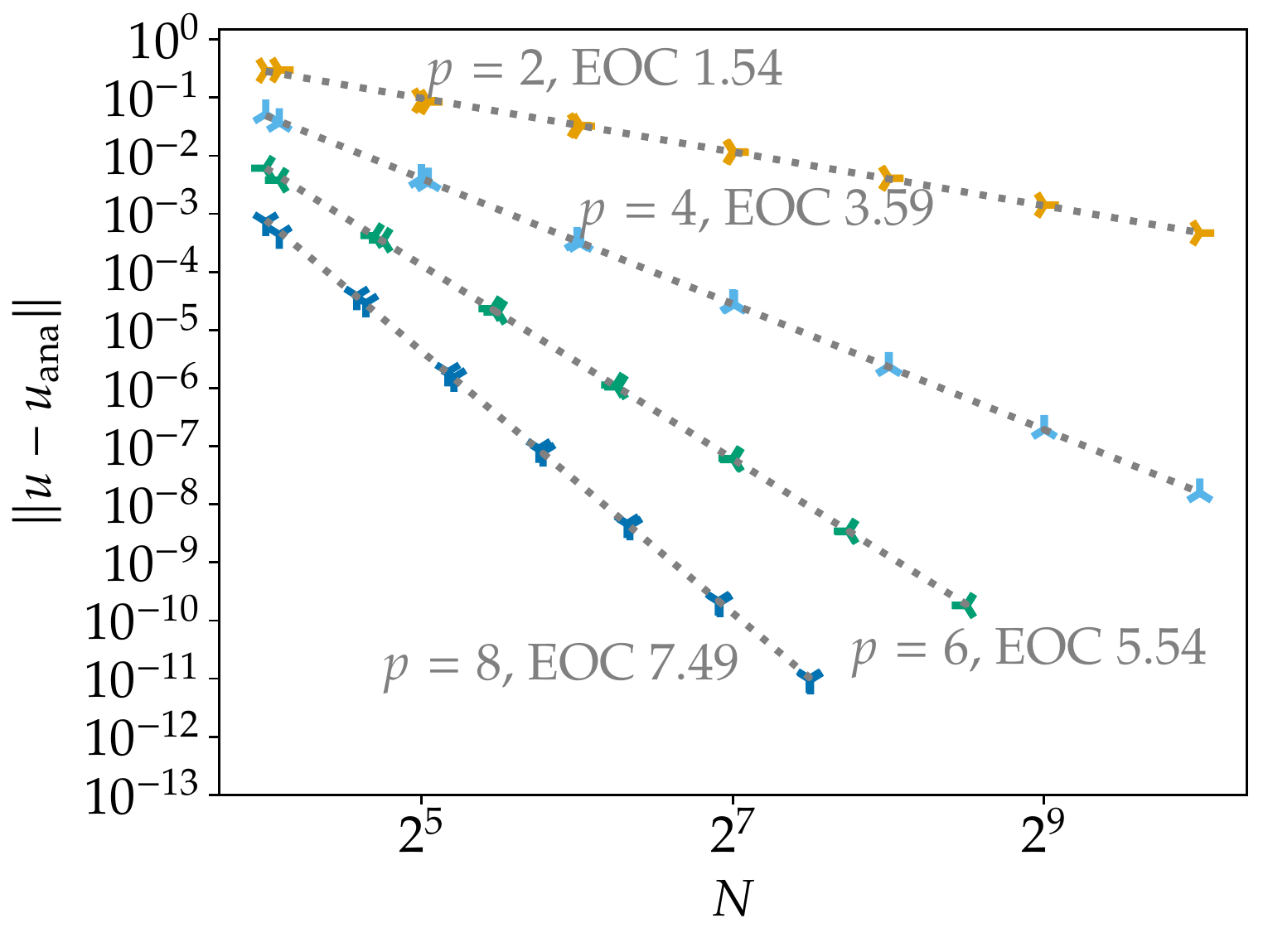}
    \caption{Finite difference methods, wide stencil $\D2 = \D1^2$.}
  \end{subfigure}%
  \hspace*{\fill}
  \begin{subfigure}[t]{0.49\textwidth}
  \centering
    \includegraphics[width=\textwidth]{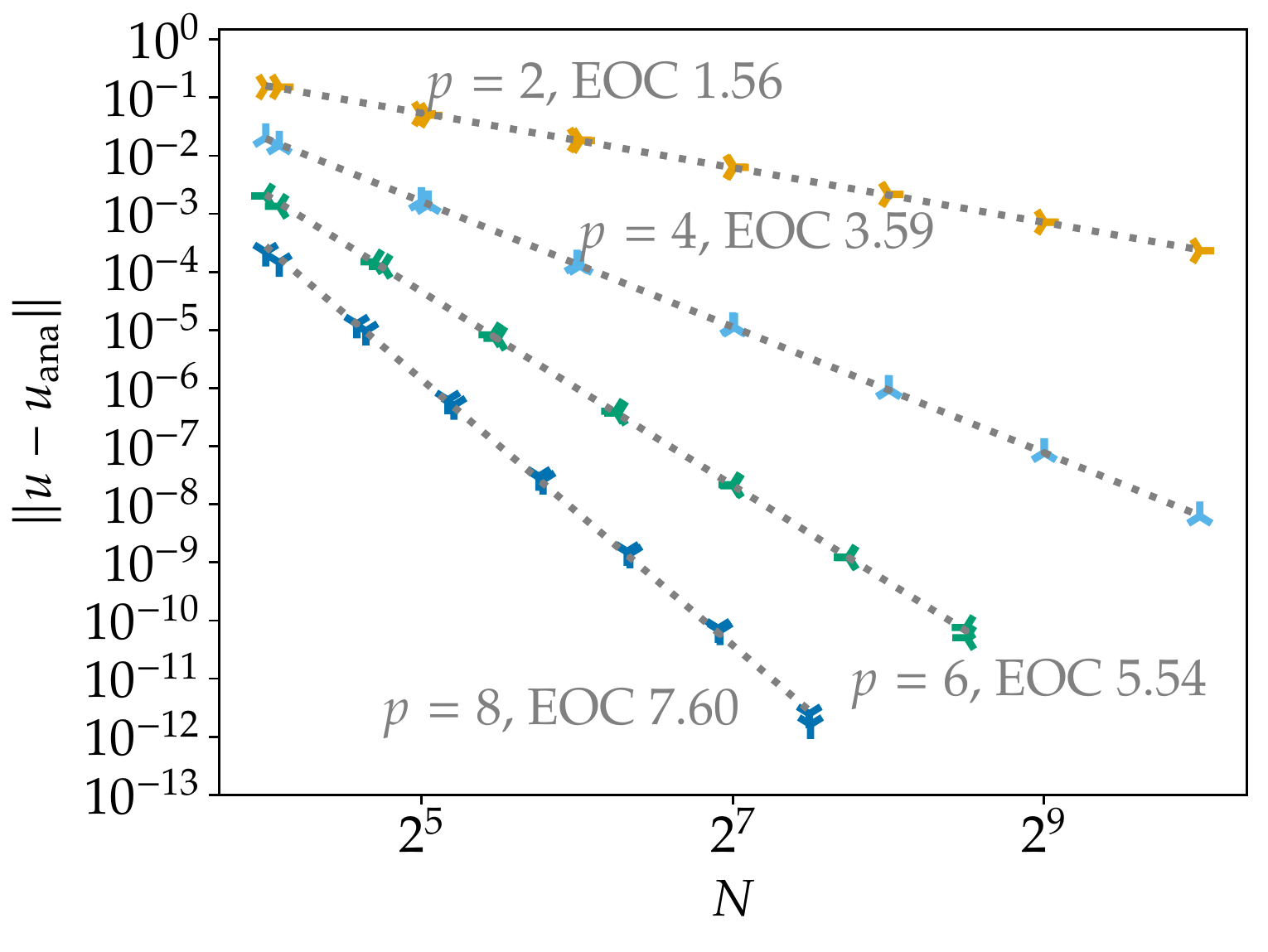}
    \caption{Finite difference methods, narrow stencil $\D2$.}
  \end{subfigure}%
  \\
  \begin{subfigure}[t]{0.8\textwidth}
  \centering
    \includegraphics[width=\textwidth]{figures/Galerkin_legend_p1_p6}
  \end{subfigure}%
  \\
  \begin{subfigure}[t]{0.49\textwidth}
  \centering
    \includegraphics[width=\textwidth]{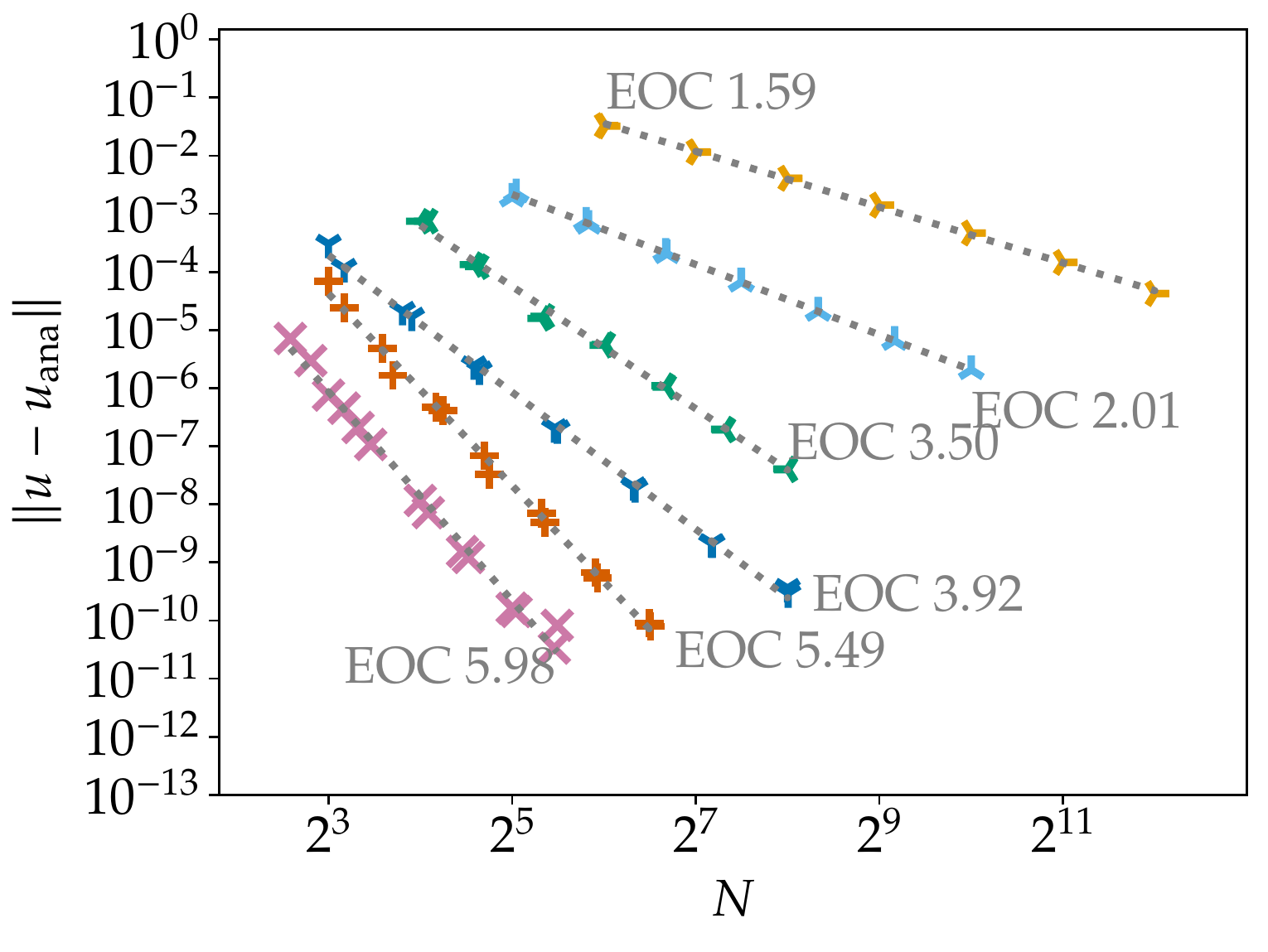}
    \caption{Continuous Galerkin methods, $\D2 = \D1^2$.}
  \end{subfigure}%
  \hspace*{\fill}
  \begin{subfigure}[t]{0.49\textwidth}
  \centering
    \includegraphics[width=\textwidth]{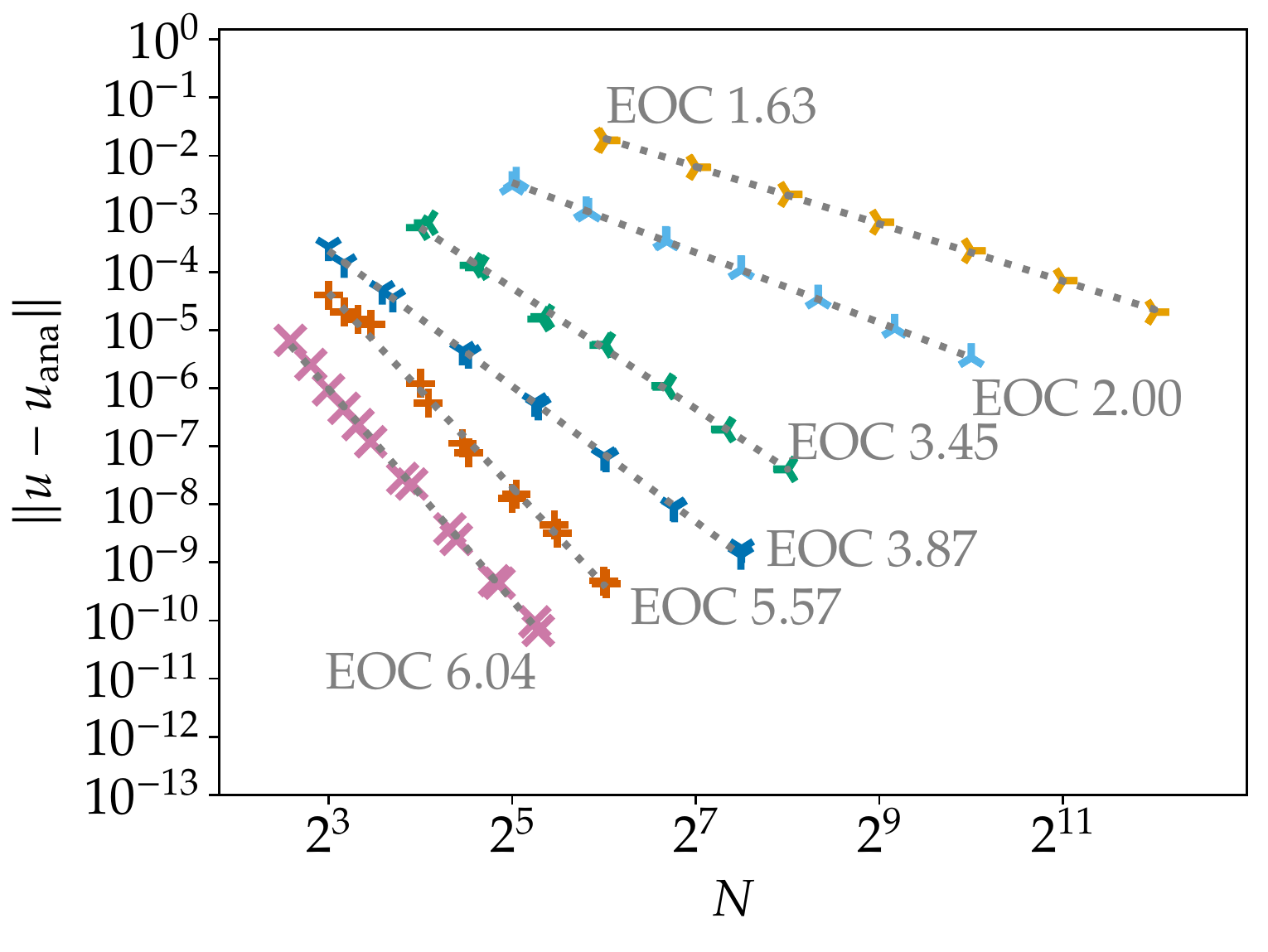}
    \caption{Continuous Galerkin methods, narrow stencil $\D2$.}
  \end{subfigure}%
  \\
  \begin{subfigure}[t]{0.49\textwidth}
  \centering
    \includegraphics[width=\textwidth]{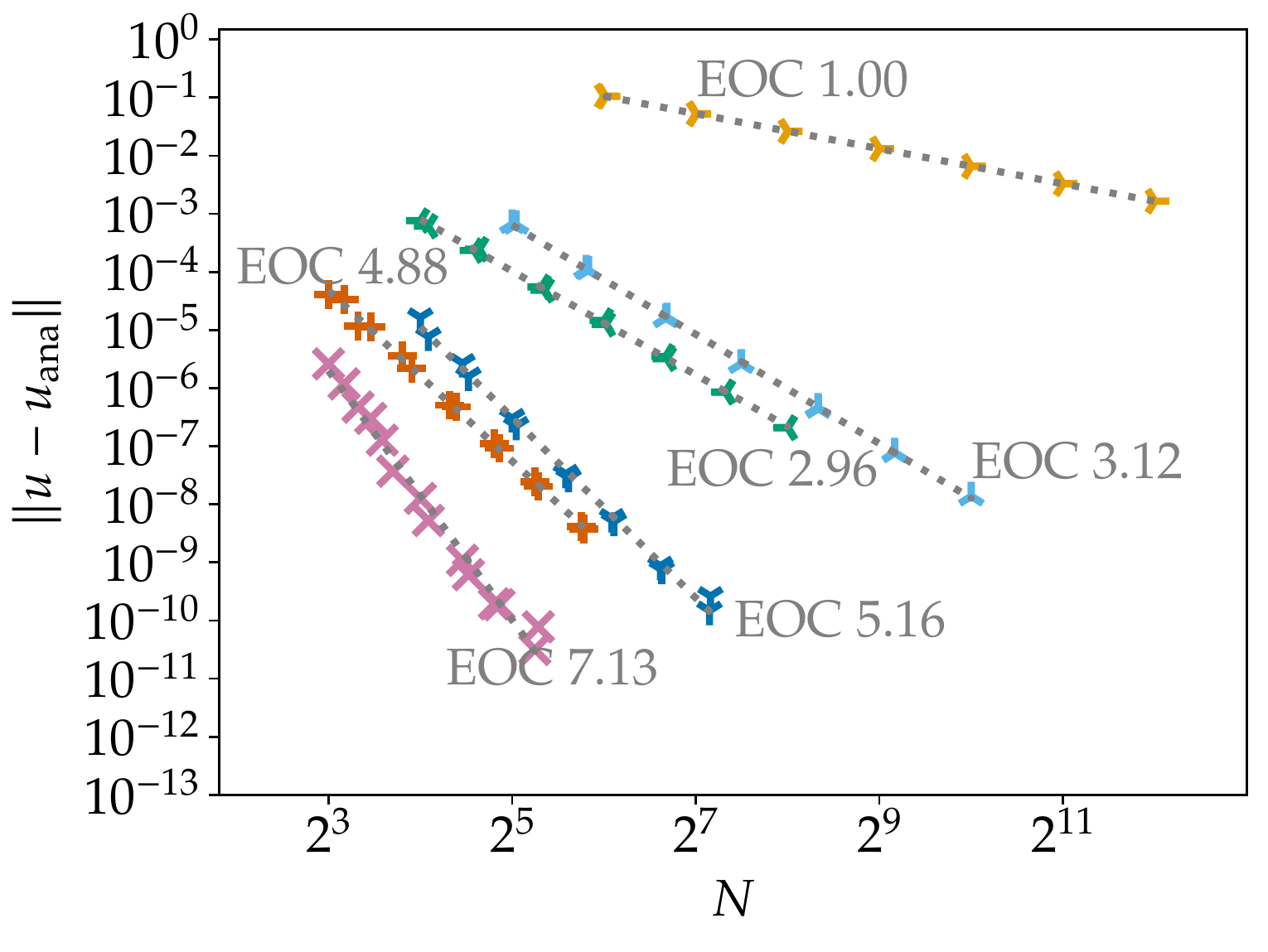}
    \caption{Discontinuous Galerkin methods, $\D2 = \D1^2$.}
  \end{subfigure}%
  \hspace*{\fill}
  \begin{subfigure}[t]{0.49\textwidth}
  \centering
    \includegraphics[width=\textwidth]{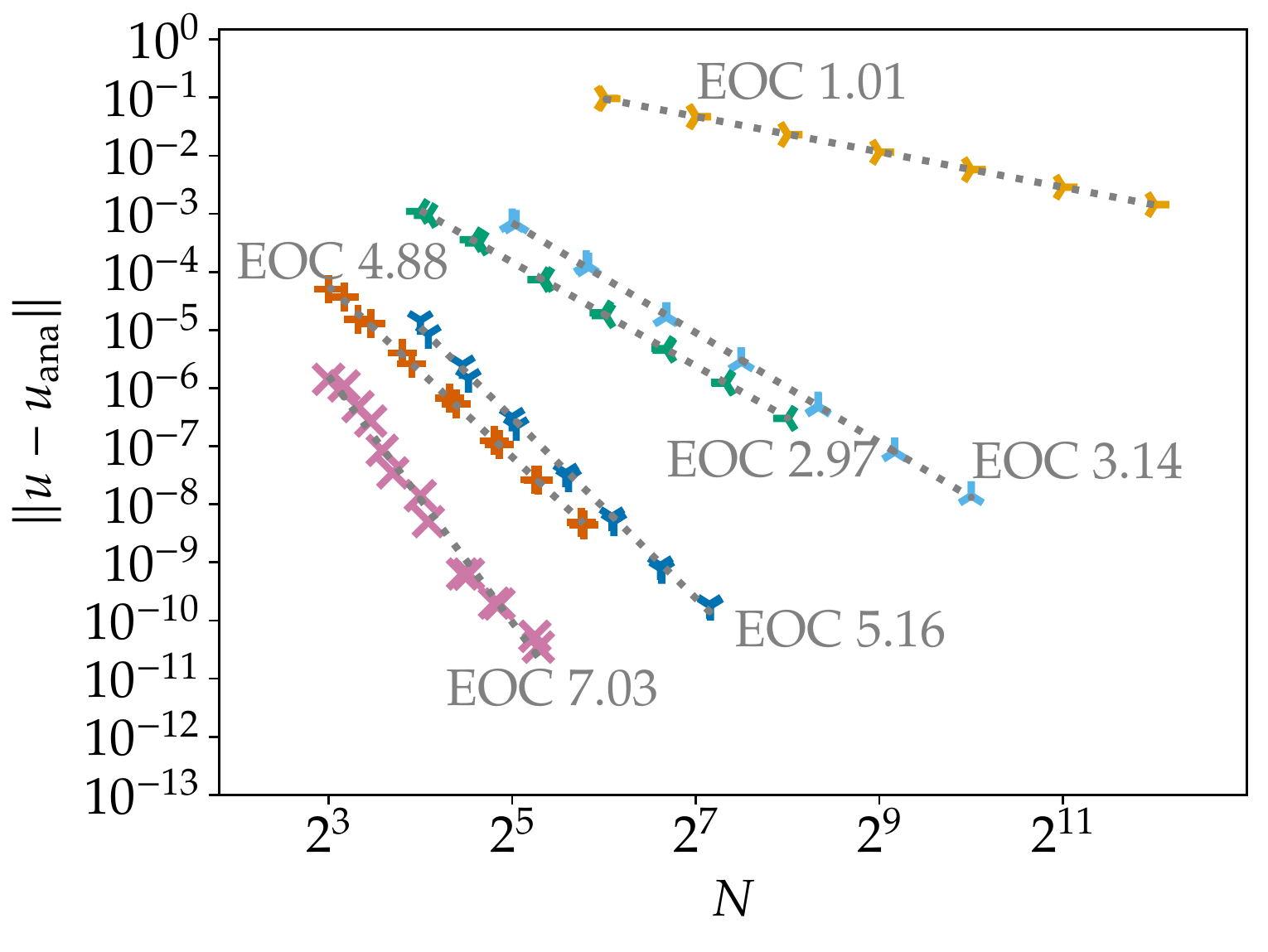}
    \caption{Discontinuous Galerkin methods, $\D2 = \D1{+} \D1{-}$.}
  \end{subfigure}%
  \caption{Convergence results of the spatial semidiscretizations
           \eqref{eq:dp-inv-periodic-SBP}
           for the manufactured solution \eqref{eq:periodic-manufactured}
           of the DP equation.
           All of these semidiscretizations conserve the linear and quadratic
           invariants \eqref{eq:dp-invariants} of the DP equation
           \eqref{eq:dp-dir}.}
  \label{fig:dp-periodic-manufactured-convergence}
\end{figure}

\subsection{Holm-Hone equation}
\label{sec:hh}

Consider the Holm-Hone equation \cite{holm2003nonintegrability}
\begin{equation}
\label{eq:hh-dir}
\begin{aligned}
  (4 - 5 \partial_x^2 + \partial_x^4) \partial_t u
  + u \partial_x^5 u
  + 2 (\partial_x u) \partial_x^4 u
  - 5 u \partial_x^3 u
  - 10 (\partial_x u) \partial_x^2 u
  + 12 u \partial_x u
  &= 0,
  \\& t \in (0, T), x \in (\xmin, \xmax),
  \\
  u(0, x) &= u^0(x),
  \\& x \in [\xmin, \xmax],
\end{aligned}
\end{equation}
with periodic boundary conditions, which can also be written as
\begin{equation}
\label{eq:hh-inv}
  \partial_t u
  + (4\I - 5 \partial_{x,P}^2 + \partial_{x,P}^4)^{-1} \left(
      \partial_x \bigl( u (4\I - 5 \partial_x^2 + \partial_x^4) u \bigr)
    + (\partial_x u) (4\I - 5 \partial_x^2 + \partial_x^4) u
  \right)
  = 0,
\end{equation}
where $(4\I - 5 \partial_{x,P}^2 + \partial_{x,P}^4)^{-1}$ is the
inverse of the elliptic operator $4\I - 5 \partial_x^2 + \partial_x^4$
with periodic boundary conditions.
The functionals
\begin{subequations}
\label{eq:hh-invariants}
\begin{align}
  \label{eq:hh-invariants-mass}
  J^{\text{HH}}_1(u)
  &= \int_{\xmin}^{\xmax} u,
  \\
  \label{eq:hh-invariants-linear}
  J^{\text{HH}}_2(u)
  &= \int_{\xmin}^{\xmax} (4\I - \partial_x^2) (\I - \partial_x^2) u
  = \int_{\xmin}^{\xmax} (4\I - 5 \partial_x^2 + \partial_x^4) u,
  \\
  \label{eq:hh-invariants-quadratic}
  J^{\text{HH}}_3(u)
  &= \frac{1}{2} \int_{\xmin}^{\xmax} u (4\I - \partial_x^2) (\I - \partial_x^2) u
  = \frac{1}{2} \int_{\xmin}^{\xmax} \bigl( 4 u^2 + 5 (\partial_x u)^2 + (\partial_x^2 u)^2 \bigr),
\end{align}
\end{subequations}
are invariants of solutions.
In the following, we will construct numerical methods that conserve
all invariants \eqref{eq:hh-invariants}.

\subsubsection{Conservative numerical methods}

Using the splitting in \eqref{eq:hh-inv}, semidiscretizations
\begin{equation}
\label{eq:hh-inv-periodic-SBP}
\begin{split}
  \partial_t \vec{u}
  =
  - (4\I - 5 \D2{a} + \D4{a})^{-1} \Bigl(
      \D1 \bigl( \vec{u} (4\I - 5 \D2{b} + \D4{b}) \vec{u} \bigr)
    + (\D1 \vec{u}) (4\I - 5 \D2{b} + \D4{b}) \vec{u}
  \Bigr)
\end{split}
\end{equation}
that conserve the linear and quadratic invariant can be obtained.
\begin{theorem}
\label{thm:hh-inv-periodic-SBP}
  If $\D1$ is a periodic first-derivative SBP operator
  with diagonal mass matrix $\M$,
  $\D2{a}$ \& $\D2{b}$ are periodic second-derivative SBP operators,
  and $\D4{a}$ \& $\D4{b}$ are periodic fourth-derivative SBP operators,
  then the semidiscretization \eqref{eq:hh-inv-periodic-SBP}
  conserves the quadratic invariant \eqref{eq:hh-invariants-quadratic}.
  If $\D1$ commutes with $\D2{b}$ \& $\D4{b}$, the linear invariants
  \eqref{eq:hh-invariants-mass} and \eqref{eq:hh-invariants-linear}
  are also conserved.
\end{theorem}
\begin{proof}
  Because of Lemma~\ref{lem:1-M-Di}, \eqref{eq:hh-invariants-mass}
  is conserved if and only if \eqref{eq:hh-invariants-linear} is
  conserved. For \eqref{eq:hh-invariants-linear}, consider
  \begin{align*}
  \stepcounter{equation}\tag{\theequation}
    &\quad
    \od{}{t} J^{\text{HH}}_1(\vec{u})
    =
    \od{}{t} \vec{1}^T \M (4\I - 5 \D2{a} + \D4{a}) \vec{u}
    =
    \vec{1}^T \M (4\I - 5 \D2{a} + \D4{a}) \partial_t \vec{u}
    \\
    &=
    - \vec{1}^T \M \left(
        \D1 (\vec{u} \D4{b} \vec{u})
      + (\D1 \vec{u}) (\D4{b} \vec{u})
      - 5 \D1 (\vec{u} \D2{b} \vec{u})
      - 5 (\D1 \vec{u}) (\D2{b} \vec{u})
      + 4 \D1 \vec{u}^2
      + 4 \vec{u} \D1 \vec{u}
    \right)
    \\
    &=
    - \vec{1}^T \M \left(
        (\D1 \vec{u}) (\D4{b} \vec{u})
      - 5 (\D1 \vec{u}) (\D2{b} \vec{u})
    \right),
  \end{align*}
  where Lemma~\ref{lem:1-in-left-kernel-of-ImD2inv} and
  Lemma~\ref{lem:1-M-Di} have been used. If $\D1$ commutes with
  $\D2{b}$ and $\D4{b}$,
  \begin{equation}
    \od{}{t} J^{\text{HH}}_1(\vec{u})
    =
    - \vec{u}^T \D1^T \M \D4 \vec{u}
    + 5 \vec{u}^T \D1^T \M \D2 \vec{u}
    =
    0.
  \end{equation}

  Since $4\I - 5 \D2{a} + \D4{a}$ is a symmetric operator,
  the semidiscrete rate of change of the quadratic invariant
  \eqref{eq:hh-invariants-quadratic} is
  \begin{equation}
  \begin{multlined}
    \od{}{t} J^{\text{HH}}_2(\vec{u})
    =
    \frac{1}{2} \od{}{t} \vec{u}^T \M (4\I - 5 \D2{a} + \D4{a}) \vec{u}
    =
    \vec{u}^T \M (4\I - 5 \D2{a} + \D4{a}) \partial_t \vec{u}
    \\
    =
    - \vec{u}^T \M \left(
        \D1 \bigl( \vec{u} (4\I - 5 \D2{b} + \D4{b}) \vec{u} \bigr)
      + (\D1 \vec{u}) (4\I - 5 \D2{b} + \D4{b}) \vec{u}
    \right)
    =
    0.
  \end{multlined}
  \qedhere
  \end{equation}
\end{proof}

A smooth traveling wave solution with speed $c = 1.2$ computed numerically
using the Petviashvili method in the periodic domain $[-40, 40]$ was used to verify
the conservation properties of the semidiscretization
\eqref{eq:hh-inv-periodic-SBP}.
This traveling wave solution has been computed for the PDE
\begin{equation}
  (4 - 5 \partial_x^2 + \partial_x^4) \partial_t u
  + 8 \kappa \partial_x u
  + u \partial_x^5 u
  + 2 (\partial_x u) \partial_x^4 u
  - 5 u \partial_x^3 u
  - 10 (\partial_x u) \partial_x^2 u
  + 12 u \partial_x u
  = 0
\end{equation}
which can be transformed to a solution of the HH equation (with $\kappa = 0$)
by the transformation
\begin{equation}
  x \to x + \kappa t,
  \quad
  u \to u + \kappa.
\end{equation}

\subsubsection{Convergence study in space}
\label{sec:hh-periodic-manufactured-convergence}

For the following convergence study, the method of manufactured solutions
is applied to \eqref{eq:periodic-manufactured}
with periodic boundary conditions
and the semidiscretization \eqref{eq:hh-inv-periodic-SBP} with
$\D2{a} = \D2{b} = \D2$ and $\D4{a} = \D4{b} = \D2^2$.
The results are shown in Figure~\ref{fig:hh-periodic-manufactured-convergence}.

Similarly to the BBM and CH equations, central finite difference methods with
order of accuracy $p$ result in an $\text{EOC} \approx p$.
Wide-stencil operators $\D2 = \D1^2$ for $p \in \set{6, 8}$ show a strong
dependence of the final error on the parity of the number of nodes. Again,
the error is up to an order of magnitude smaller for odd $N$.
For $p \in \set{2, 4}$ or narrow-stencil operators $\D2$, such a dependence
is not visible for this test case.

As for the BBM equation and in contrast to the FW, CH, and DP equations,
nodal continuous Galerkin methods using Lobatto-Legendre bases with wide
stencil operator $\D2 = \D1^2$ yield
$\text{EOC} \approx p+1$ for odd polynomial degrees $p$ and
$\text{EOC} \approx p$ for even $p$.
The error is approximately an order of magnitude smaller for $p = 2$
than for $p = 1$ for the same number of elements $N$. In contrast, $p = 3$
can result in a similar error as $p = 4$ for odd numbers of elements $N$
while the error is up to an order of magnitude bigger for even $N$.

\begin{figure}[htb]
\centering
  \begin{subfigure}[t]{0.49\textwidth}
  \centering
    \includegraphics[width=\textwidth]{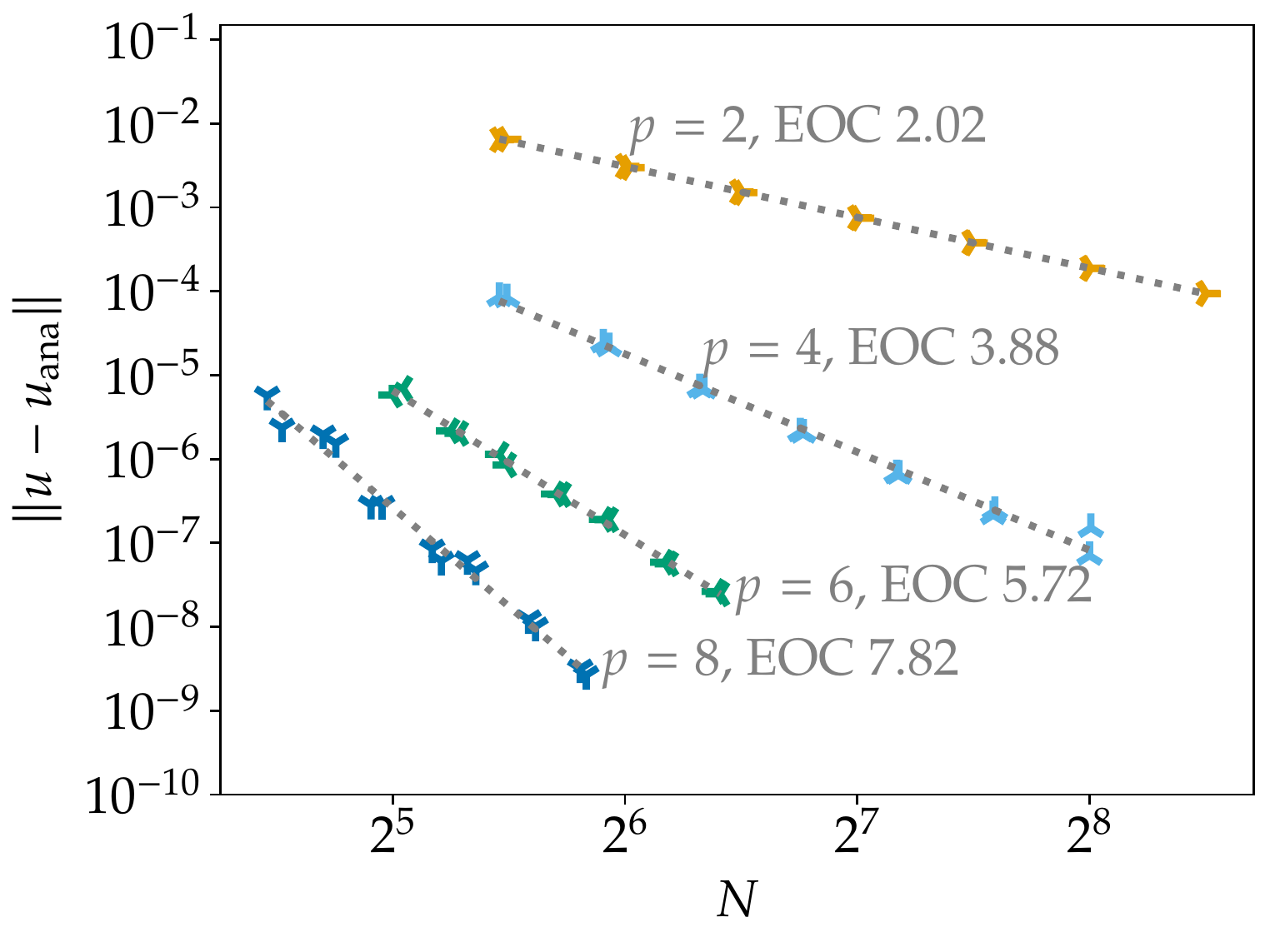}
    \caption{Finite difference methods, narrow stencil $\D2$.}
  \end{subfigure}%
  \hspace*{\fill}
  \begin{subfigure}[t]{0.49\textwidth}
  \centering
    \includegraphics[width=\textwidth]{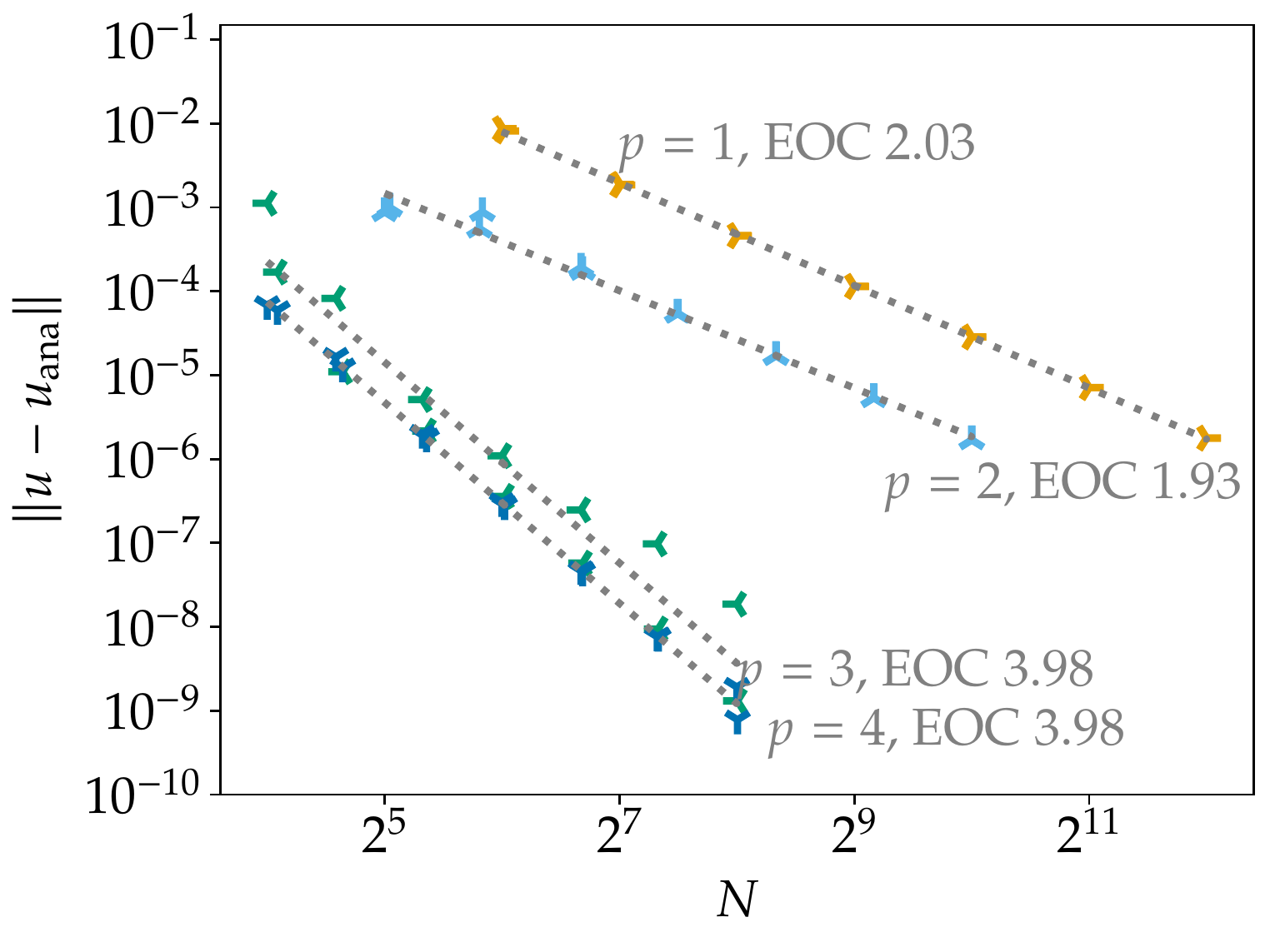}
    \caption{Continuous Galerkin methods, $\D2 = \D1^2$.}
  \end{subfigure}%
  \caption{Convergence results of the spatial semidiscretizations
           \eqref{eq:hh-inv-periodic-SBP} with $\D4 = \D2^2$
           for the manufactured solution \eqref{eq:periodic-manufactured}
           of the HH equation.
           All of these semidiscretizations conserve the linear and quadratic
           invariants \eqref{eq:hh-invariants} of the HH equation
           \eqref{eq:hh-dir}.}
  \label{fig:hh-periodic-manufactured-convergence}
\end{figure}

\subsection{BBM-BBM system}
\label{sec:bbm_bbm}

Consider the system \cite{bona2002boussinesq, bona2004boussinesq,
antonopoulos2009initial, antonopoulos2010numerical}
\begin{equation}
\label{eq:bbm_bbm-dir-periodic}
\begin{aligned}
  \partial_t \eta(t,x)
  + \partial_x u(t,x)
  + \partial_x \bigl( \eta(t,x) u(t,x) \bigr)
  - \partial_t \partial_x^2 \eta(t,x)
  &= 0,
  && t \in (0, T), x \in (\xmin, \xmax),
  \\
  \partial_t u(t,x)
  + \partial_x \eta(t,x)
  + \partial_x \frac{u(t,x)^2}{2}
  - \partial_t \partial_x^2 u(t,x)
  &= 0,
  && t \in (0, T), x \in (\xmin, \xmax),
  \\
  \eta(0, x) &= \eta^0(x),
  && x \in [\xmin, \xmax],
  \\
  u(0, x) &= u^0(x),
  && x \in [\xmin, \xmax],
\end{aligned}
\end{equation}
with periodic boundary conditions, which can also be written as
\begin{equation}
\label{eq:bbm_bbm-inv-periodic}
\begin{aligned}
  \partial_t \eta(t,x)
  + (\I - \partial_{x,P}^2)^{-1} \partial_x \bigl( u(t,x) + \eta(t,x) u(t,x) \bigr)
  &= 0,
  \\
  \partial_t u(t,x)
  + (\I - \partial_{x,P}^2)^{-1} \partial_x \biggl( \eta(t,x) + \frac{u(t,x)^2}{2} \biggr)
  &= 0,
\end{aligned}
\end{equation}
where $(\I - \partial_{x,P}^2)^{-1}$ is the inverse of the elliptic operator
$\I - \partial_x^2$ with periodic boundary conditions.
The functionals
\begin{subequations}
\label{eq:bbm_bbm-invariants}
\begin{align}
\label{eq:bbm_bbm-invariants-mass-eta}
  J^{\text{BBM-BBM}}_1(\eta, u)
  &=
  \int_{\xmin}^{\xmax} \eta,
  \\
\label{eq:bbm_bbm-invariants-mass-u}
  J^{\text{BBM-BBM}}_2(\eta, u)
  &=
  \int_{\xmin}^{\xmax} u,
  \\
\label{eq:bbm_bbm-invariants-energy}
  J^{\text{BBM-BBM}}_3(\eta, u)
  &=
  \int_{\xmin}^{\xmax} (\eta^2 + (1 + \eta) u^2),
\end{align}
\end{subequations}
are invariants of solutions of \eqref{eq:bbm_bbm-dir-periodic}.

Interestingly, only integration by parts but no chain or product rule is
necessary to prove conservation of the energy
\eqref{eq:bbm_bbm-invariants-energy}.
Hence, all three invariants are conserved semidiscretely if periodic SBP
operators are employed and the semidiscretization uses the conservative form
\begin{equation}
\label{eq:bbm_bbm-inv-periodic-SBP}
\begin{aligned}
  \partial_t \vec{\eta}
  + (\I - \D2)^{-1} \D1 ( \vec{u} + \vec{\eta} \vec{u} )
  &= \vec{0},
  \\
  \partial_t \vec{u}
  + (\I - \D2)^{-1} \D1 \biggl( \vec{\eta} + \frac{1}{2} \vec{u}^2 \biggr)
  &= \vec{0}.
\end{aligned}
\end{equation}

\begin{theorem}
\label{thm:bbm_bbm-inv-periodic-SBP}
  If $\D1$ is a periodic first-derivative SBP operator
  and $\D2$ is  a periodic second-derivative SBP operator,
  then the semidiscretization \eqref{eq:bbm_bbm-inv-periodic-SBP}
  conserves the linear invariants \eqref{eq:bbm_bbm-invariants-mass-eta} and
  \eqref{eq:bbm_bbm-invariants-mass-u} of \eqref{eq:bbm_bbm-dir-periodic}.
  If $\D1$ \& $\D2$ commute, then the quadratic invariant
  \eqref{eq:bbm_bbm-invariants-energy} is also conserved.
\end{theorem}
\begin{proof}
  Conservation of the linear invariants $J^{\text{BBM-BBM}}_1,
  J^{\text{BBM-BBM}}_2$ follows immediately from
  Lemmas~\ref{lem:1-M-Di} and \ref{lem:1-in-left-kernel-of-ImD2inv}.
  Given a matrix $A$, set $\langle\vec{x}, \vec{y}\rangle_{A} =
  \vec{x}^T A \vec{y}$.
  Conservation of $J^{\text{BBM-BBM}}_3$ follows from Lemmas~\ref{lem:1-M-Di}
  and \ref{lem:D1-D2-commuting}, since (letting $L = \M (\I - \D2)^{-1} \D1$)
  \begin{align*}
    \od{}{t} J^{\text{BBM-BBM}}_3(\vec{\eta}, \vec{u})
    &=
    \od{}{t} \left( \langle \vec{\eta}, \vec{\eta} \rangle_{\M} + \langle \vec{u}^2, \vec{1} + \vec{\eta}\rangle_{\M} \right)
    =
    \langle \partial_t\vec{\eta} ,2\vec{\eta} + \vec{u}^2\rangle_M + \langle\partial_t \vec{u}, 2(1+\vec{\eta})\vec{u} \rangle_M
    \\
    &=
    -\langle \vec{u}+\vec{\eta}\vec{u}, 2\vec{\eta} + \vec{u}^2 \rangle_L
    - 2\langle \vec{\eta} + \frac{1}{2}\vec{u}^2, (1+\vec{\eta})\vec{u} \rangle_L
    \\
    &=
    -2\langle \vec{u},\vec{\eta} \rangle_L - 2 \langle \vec{\eta u}, \vec{\eta} \rangle_L
    - \langle \vec{u}, \vec{u}^2 \rangle_L - \langle \vec{\eta u}, \vec{u^2} \rangle_L
    -2 \langle \vec{\eta}, \vec{u} \rangle_L - \langle \vec{u}^2, \vec{u} \rangle_L
    \\&\quad
    -2 \langle \vec{\eta}, \vec{\eta u} \rangle_L -  \langle \vec{u}^2, \vec{\eta u} \rangle_L
    = 0,
  \end{align*}
  because $L$ is skew-symmetric by Lemma~\ref{lem:D1-D2-commuting}.
\end{proof}

\subsubsection*{Reflecting boundary conditions}

Another interesting set of boundary conditions is given by
reflecting boundary conditions, \ie
\begin{equation}
\label{eq:bbm_bbm-dir-reflecting}
\begin{aligned}
  \partial_t \eta(t,x)
  + \partial_x u(t,x)
  + \partial_x \bigl( \eta(t,x) u(t,x) \bigr)
  - \partial_t \partial_x^2 \eta(t,x)
  &= 0,
  && t \in (0, T), x \in (\xmin, \xmax),
  \\
  \partial_t u(t,x)
  + \partial_x \eta(t,x)
  + \partial_x \frac{u(t,x)^2}{2}
  - \partial_t \partial_x^2 u(t,x)
  &= 0,
  && t \in (0, T), x \in (\xmin, \xmax),
  \\
  \partial_x \eta(t,x) &= 0,
  && t \in (0, T), x \in \{\xmin, \xmax\},
  \\
  u(t,x) &= 0,
  && t \in (0, T), x \in \{\xmin, \xmax\},
  \\
  \eta(0, x) &= \eta^0(x),
  && x \in [\xmin, \xmax],
  \\
  u(0, x) &= u^0(x),
  && x \in [\xmin, \xmax].
\end{aligned}
\end{equation}
Such boundary conditions occur in the study of two-way dispersive waves
in fluids subject to solid wall boundary conditions \cite{antonopoulos2009initial}.
For this system, the total mass of $\eta$
\eqref{eq:bbm_bbm-invariants-mass-eta}
and the total energy
\eqref{eq:bbm_bbm-invariants-energy}
are still conserved, but the total mass of $u$
\eqref{eq:bbm_bbm-invariants-mass-u}
is not necessarily conserved.

To create a conservative semidiscretization, the following
operators will be used.
The projection operator $\proj{D} = \diag(0, 1, \dots, 1, 0)$
maps onto the space of grid functions with homogeneous Dirichlet
boundary conditions.
For an SBP first-derivative operator $\D1$, denote the second
derivative operator induced by $\I - \D1^2$
\begin{itemize}
  \item
  with strong imposition of homogeneous Dirichlet boundary
  conditions as $\I - \D2{D}$. The corresponding solution
  operator satisfies
  \begin{equation}
    \vec{v} = (\I - \D2{D})^{-1} \vec{w}
    \iff
    \proj{D} (\I - \D2{D}) \vec{v} = \proj{D} \vec{w}
    \land
    (\I - \proj{D}) \vec{v} = \vec{0}.
  \end{equation}

  \item
  with weak-strong imposition of homogeneous Neumann boundary
  conditions as $\I - \D2{N}$. The corresponding solution
  operator satisfies
  \begin{equation}
  \label{eq:D2N}
    \vec{v} = (\I - \D2{N})^{-1} \vec{w}
    \iff
    (\I + \M^{-1} \D1^T \M \proj{D} \D1) \vec{v} = \vec{w}.
  \end{equation}
  This discretization uses neither a strong imposition of boundary
  conditions (setting $\vec{e}_{L/R}^T \D1 (\I - \D2{N})^{-1} = \vec{0}^T$
  and solving the PDE in the interior) nor the usual weak
  imposition (where \eqref{eq:D2N} is used without the additional
  projection $\proj{D}$).
\end{itemize}
These operators correspond to $\I - \D1^2$ in the interior and are
modified near the boundaries to impose the boundary conditions.
Using these operators results in the semidiscretization
\begin{equation}
\label{eq:bbm_bbm-inv-reflecting-SBP}
\begin{aligned}
  \partial_t \vec{\eta}
  + (\I - \D2{N})^{-1} \D1 ( \vec{u} + \vec{\eta} \vec{u} )
  &= \vec{0},
  \\
  \partial_t \vec{u}
  + (\I - \D2{D})^{-1} \D1 \biggl( \vec{\eta} + \frac{1}{2} \vec{u}^2 \biggr)
  &= \vec{0}.
\end{aligned}
\end{equation}

\begin{theorem}
\label{thm:bbm_bbm-inv-reflecting-SBP}
  If $\D1$ is a first-derivative SBP operator
  and the initial condition for $\vec{u}$ satisfies the homogeneous
  Dirichlet boundary condition,
  then the semidiscretization \eqref{eq:bbm_bbm-inv-reflecting-SBP}
  conserves the two invariants of \eqref{eq:bbm_bbm-dir-reflecting},
  \ie the total mass of $\eta$
  \eqref{eq:bbm_bbm-invariants-mass-eta}
  and the total energy
  \eqref{eq:bbm_bbm-invariants-energy}.
\end{theorem}
\begin{proof}
  The semidiscretization \eqref{eq:bbm_bbm-inv-reflecting-SBP}
  can be written as
  \begin{equation}
  \label{eq:bbm_bbm-dir-reflecting-SBP}
  \begin{aligned}
    (\I - \D2{N}) \partial_t \vec{\eta}
    + \D1 ( \vec{u} + \vec{\eta} \vec{u} )
    &= \vec{0},
    \\
    (\I - \D2{D}) \partial_t \vec{u}
    + \D1 \biggl( \vec{\eta} + \frac{1}{2} \vec{u}^2 \biggr)
    &= \vec{0}.
  \end{aligned}
  \end{equation}
  Since $\vec{e}_{L/R}^T \partial_t \vec{u} = 0$, the homogeneous
  Dirichlet boundary condition for $\vec{u}$ is satisfied for all
  times if it is satisfied initially.

  The total mass of $\eta$ is conserved, since the strong imposition
  of homogeneous Dirichlet boundary conditions for $\vec{u}$ yields
  \begin{equation}
  \begin{aligned}
    \vec{1}^T \M \partial_t \vec{\eta}
    &=
    - \vec{1}^T \M \D1 ( \vec{u} + \vec{\eta} \vec{u} )
    + \vec{1}^T \M \D2{N} \partial_t \vec{\eta}
    \\
    &=
    - \vec{1}^T \M \D1 ( \vec{u} + \vec{\eta} \vec{u} )
    - \vec{1}^T \D1^T \M \proj{D} \D1 \partial_t \vec{\eta}
    \\
    &=
    - \vec{1}^T (\eR \eR^T - \eL \eL^T) ( \vec{u} + \vec{\eta} \vec{u} )
    = 0.
  \end{aligned}
  \end{equation}
  To compute the semidiscrete rate of change of the energy,
  observe that
  \begin{equation}
    (-\D1 \partial_t \vec{u})^T \M \partial_t \vec{\eta}
    + (-\D1 \partial_t \vec{\eta})^T \M \partial_t \vec{u}
    =
    0.
  \end{equation}
  Hence,
  \begin{equation}
  \begin{aligned}
    \od{}{t} J^{\text{BBM-BBM}}_3(\vec{\eta}, \vec{u})
    &=
    (2\vec{\eta} + \vec{u}^2)^T \M \partial_t \vec{\eta}
    + 2 (\vec{u} + \vec{\eta} \vec{u})^T \M \partial_t \vec{u}
    \\
    &=
    (2\vec{\eta} + \vec{u}^2 - 2 \D1 \partial_t \vec{u})^T \M \partial_t \vec{\eta}
    + 2 (\vec{u} + \vec{\eta} \vec{u} - 2 \D1 \partial_t \vec{\eta})^T \M \partial_t \vec{u}
    \\
    &=
    (2\vec{\eta} + \vec{u}^2 - 2 \D1 \partial_t \vec{u})^T \M \partial_t \vec{\eta}
    + 2 (\vec{u} + \vec{\eta} \vec{u} - 2 \D1 \partial_t \vec{\eta})^T \proj{D}^T \M \partial_t \vec{u},
  \end{aligned}
  \end{equation}
  because of the strong imposition of the homogeneous Dirichlet
  boundary conditions for $\partial_t \vec{u}$.
  Inserting the semidiscretization \eqref{eq:bbm_bbm-dir-reflecting-SBP}
  and using again the homogeneous Dirichlet boundary conditions for
  $\vec{u}$ results in
  \begin{equation}
  \begin{aligned}
    \od{}{t} J^{\text{BBM-BBM}}_3(\vec{\eta}, \vec{u})
    &=
    - (2\vec{\eta} + \vec{u}^2 - 2 \D1 \partial_t \vec{u})^T \M ( \D1 ( \vec{u} + \vec{\eta} \vec{u}) - \D2{N} \partial_t \vec{\eta} )
    \\&\quad
    - (\vec{u} + \vec{\eta} \vec{u} - 2 \D1 \partial_t \vec{\eta})^T \proj{D}^T \M ( \D1 ( 2 \vec{\eta} + \vec{u}^2 ) - \D2{D} \partial_t \vec{u} )
    \\
    &=
    + (2\vec{\eta} + \vec{u}^2 - 2 \D1 \partial_t \vec{u})^T \D1^T \M \proj{D} ( ( \vec{u} + \vec{\eta} \vec{u}) + \D1 \partial_t \vec{\eta} )
    \\&\quad
    - (\vec{u} + \vec{\eta} \vec{u} - 2 \D1 \partial_t \vec{\eta})^T \proj{D}^T \M \D1 ( ( 2 \vec{\eta} + \vec{u}^2 ) - \D1 \partial_t \vec{u} )
    = 0.
  \qedhere
  \end{aligned}
  \end{equation}
\end{proof}

\begin{remark}
\label{rem:bbm_bbm-traveling-wave}
  A smooth traveling wave solution with speed $c = 1.2$ computed numerically
  using the Petviashvili method in the periodic domain $[-90, 90]$ was used to
  verify the conservation properties of the semidiscretizations
  \eqref{eq:bbm_bbm-inv-periodic-SBP} and \eqref{eq:bbm_bbm-inv-reflecting-SBP}.
  As expected, relaxation methods in time conserve all linear and the chosen
  nonlinear invariant up to roundoff errors; \cf Section~\ref{sec:bbm_bbm-reflecting}.
\end{remark}

\subsubsection{Convergence study in space}
\label{sec:bbm_bbm-convergence}

For the following convergence study, the method of manufactured solutions
is applied to
\begin{equation}
\label{eq:bbm_bbm-periodic-manufactured}
  \eta(t,x) = \e^{t} \cos(2 \pi (x - 2t)),
  \qquad
  u(t,x) = \e^{t/2} \sin(2 \pi (x - t/2)),
\end{equation}
with periodic boundary conditions in the domain $[0, 1]$ for $t \in [0, 1]$.
The results are shown in Figure~\ref{fig:bbm_bbm-periodic-manufactured-convergence}.

Similarly to the single BBM equation, central finite difference methods with
order of accuracy $p$ yield an $\text{EOC} \approx p$. Again, the results for
wide-stencil and narrow-stencil second-derivative operators are similar but
the narrow-stencil operators result again in smaller errors
(up to an order of magnitude).

The results for nodal continuous Galerkin methods using Lobatto-Legendre bases
are very similar to the ones for the single BBM equation.
Wide-stencil operators $\D2 = \D1^2$ yield
$\text{EOC} \approx p+1$ for odd polynomial degrees $p$ and
$\text{EOC} \approx p$ for even $p$.
In contrast, the usual narrow-stencil approximation results in an
$\text{EOC} \approx p+1$ for $p = 1$ and an
$\text{EOC} \approx p+2$ for $p > 1$.
However, in contrast to the single BBM equation, only the wide-stencil
operators yield energy-conservative methods.

A similar observation can be made for nodal discontinuous Galerkin methods.
There, wide-stencil operators $\D2 = \D1^2$ yield
$\text{EOC} \approx p+1$ for even polynomial degrees $p$ and
$\text{EOC} \approx p$ for odd $p$.
The narrow-stencil LDG operator $\D2 = \D1{+} \D1{-}$ results in an
$\text{EOC} \approx p+1$ for all $p$.
Again, only the narrow-stencil operators result in energy-conservative
methods.

\begin{figure}[htbp]
\centering
  \begin{subfigure}[t]{0.49\textwidth}
  \centering
    \includegraphics[width=\textwidth]{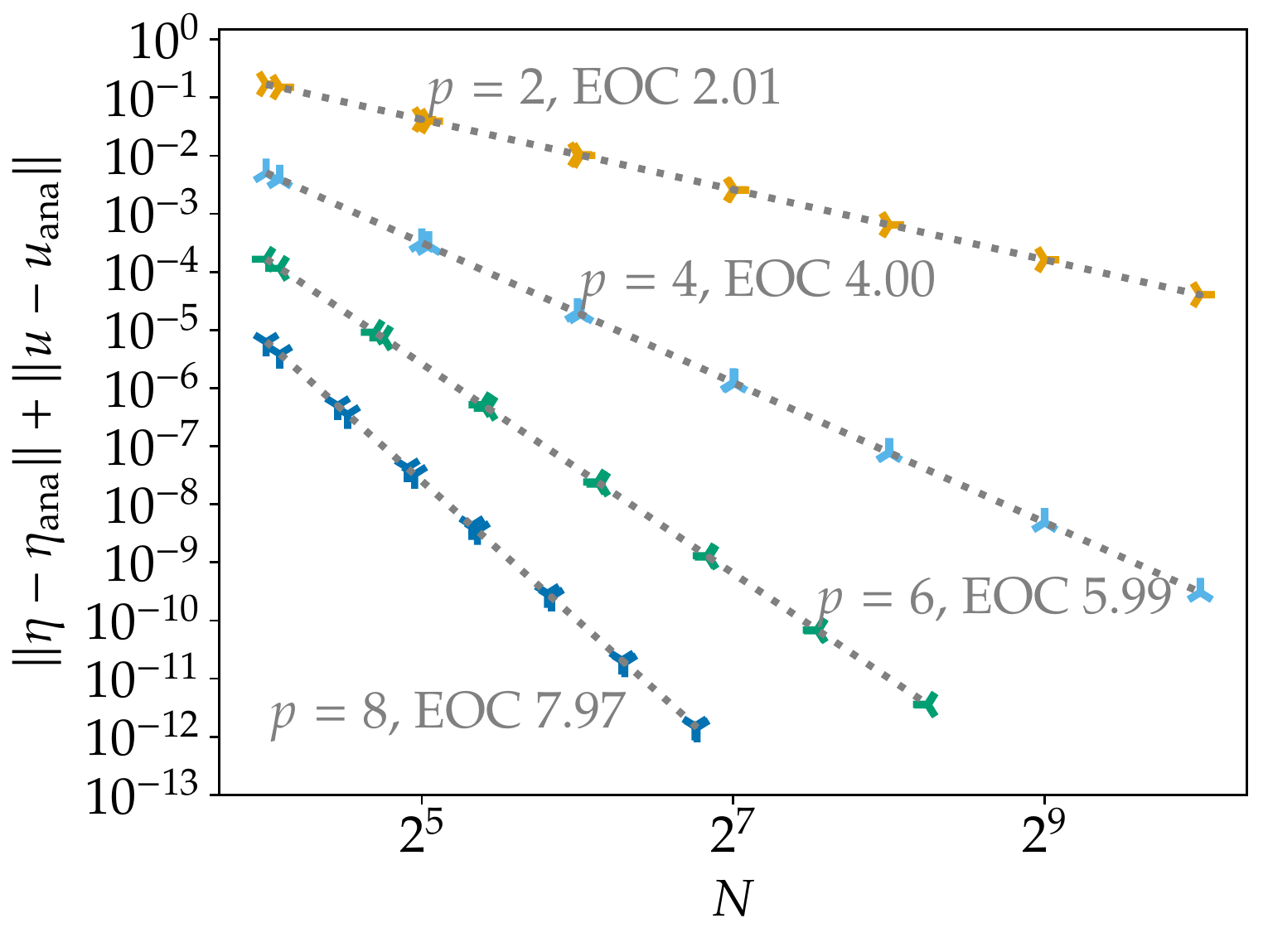}
    \caption{Finite difference methods, wide stencil $\D2 = \D1^2$.}
  \end{subfigure}%
  \hspace*{\fill}
  \begin{subfigure}[t]{0.49\textwidth}
  \centering
    \includegraphics[width=\textwidth]{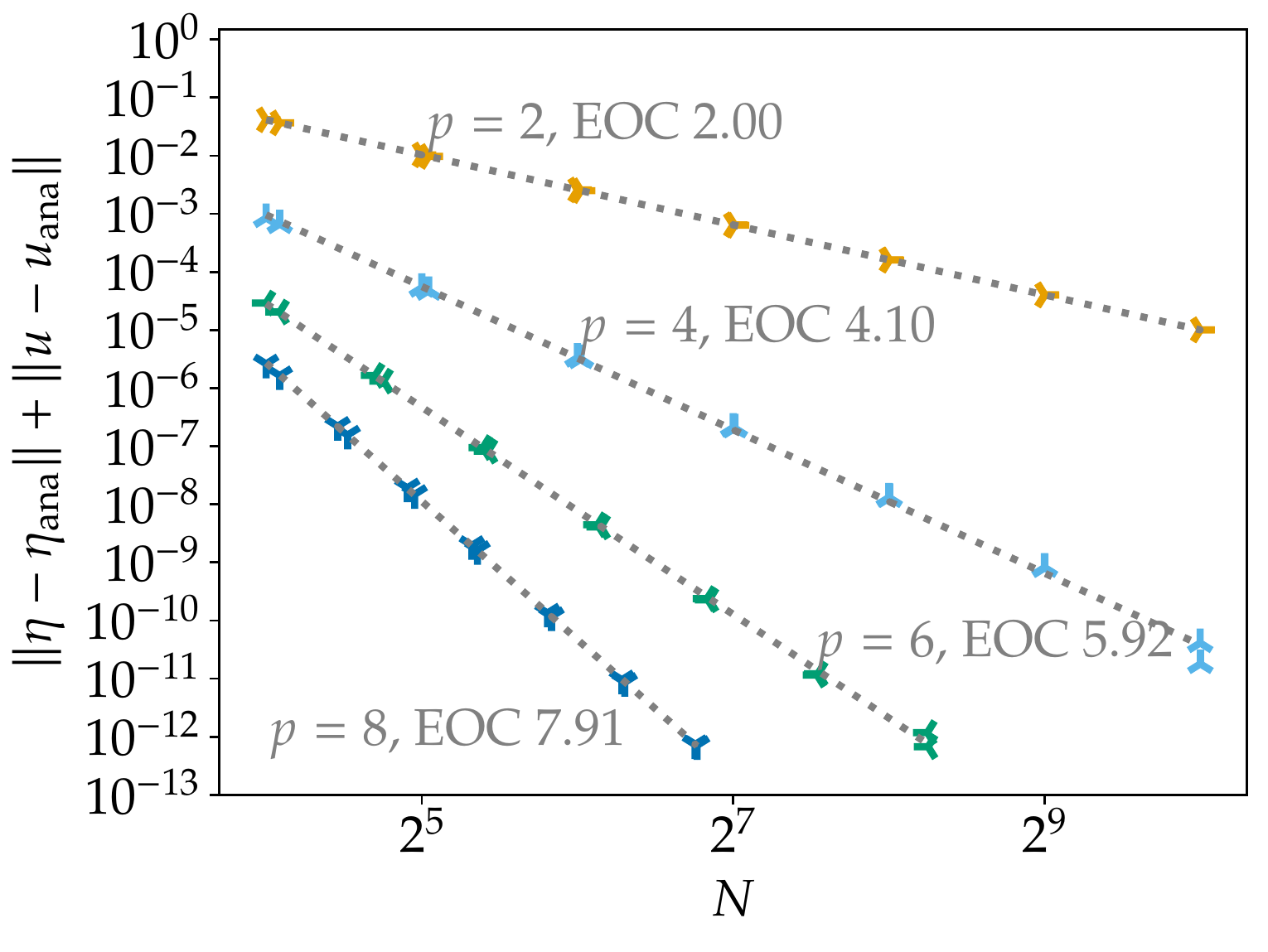}
    \caption{Finite difference methods, narrow stencil $\D2$.}
  \end{subfigure}%
  \\
  \begin{subfigure}[t]{0.8\textwidth}
  \centering
    \includegraphics[width=\textwidth]{figures/Galerkin_legend_p1_p6}
  \end{subfigure}%
  \\
  \begin{subfigure}[t]{0.49\textwidth}
  \centering
    \includegraphics[width=\textwidth]{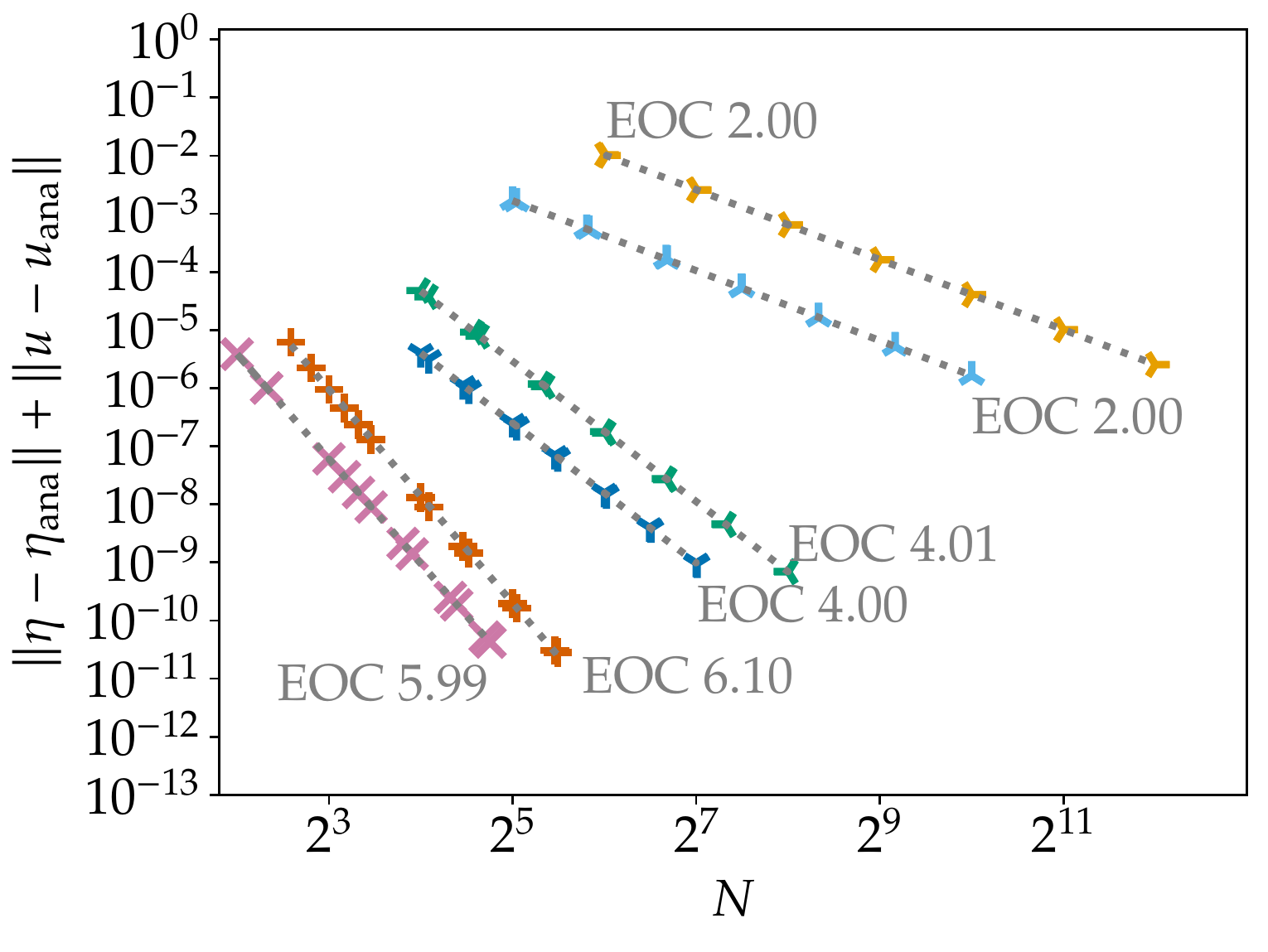}
    \caption{Continuous Galerkin methods, $\D2 = \D1^2$.}
  \end{subfigure}%
  \hspace*{\fill}
  \begin{subfigure}[t]{0.49\textwidth}
  \centering
    \includegraphics[width=\textwidth]{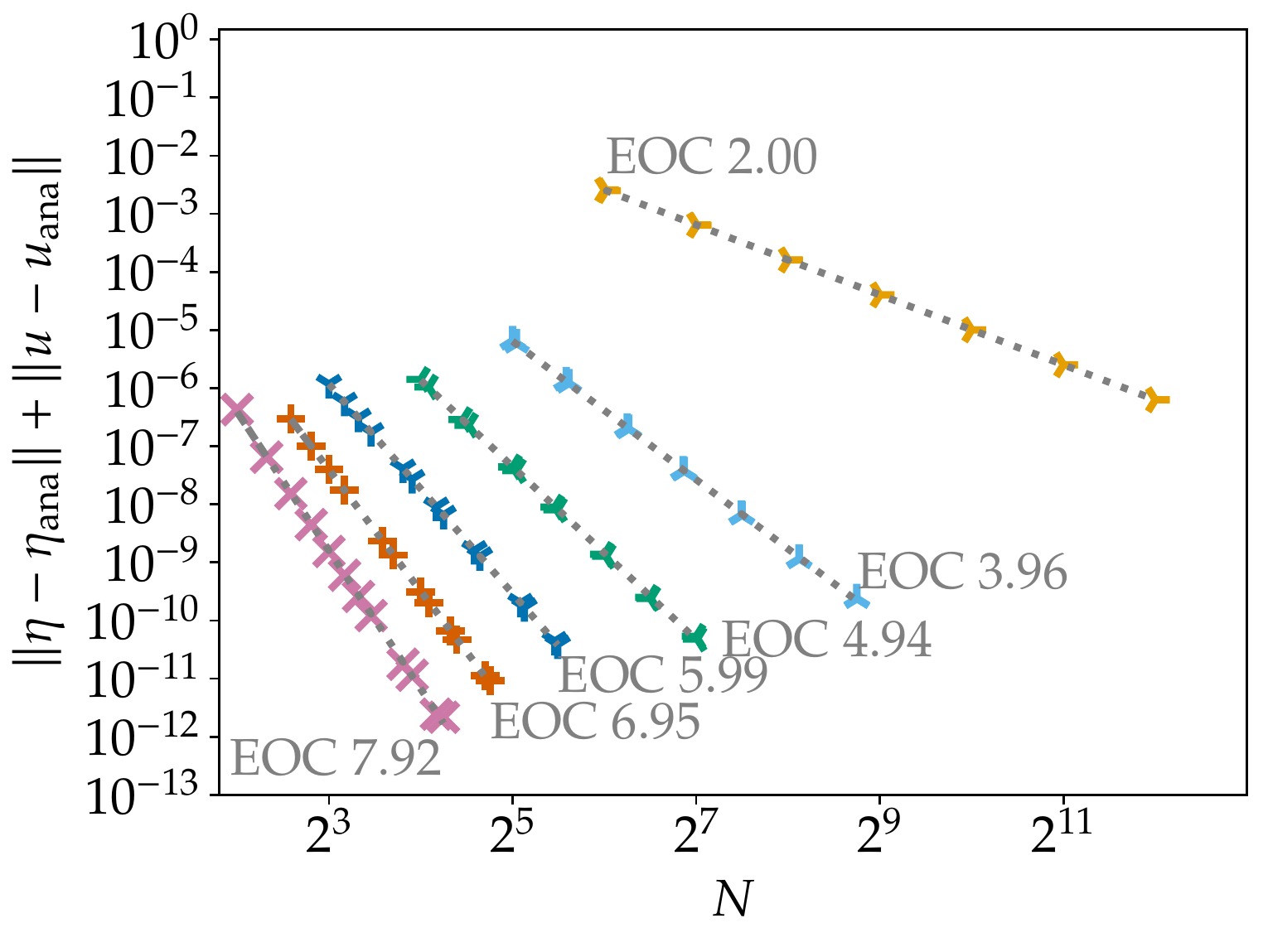}
    \caption{Continuous Galerkin methods, narrow stencil $\D2$.}
  \end{subfigure}%
  \\
  \begin{subfigure}[t]{0.49\textwidth}
  \centering
    \includegraphics[width=\textwidth]{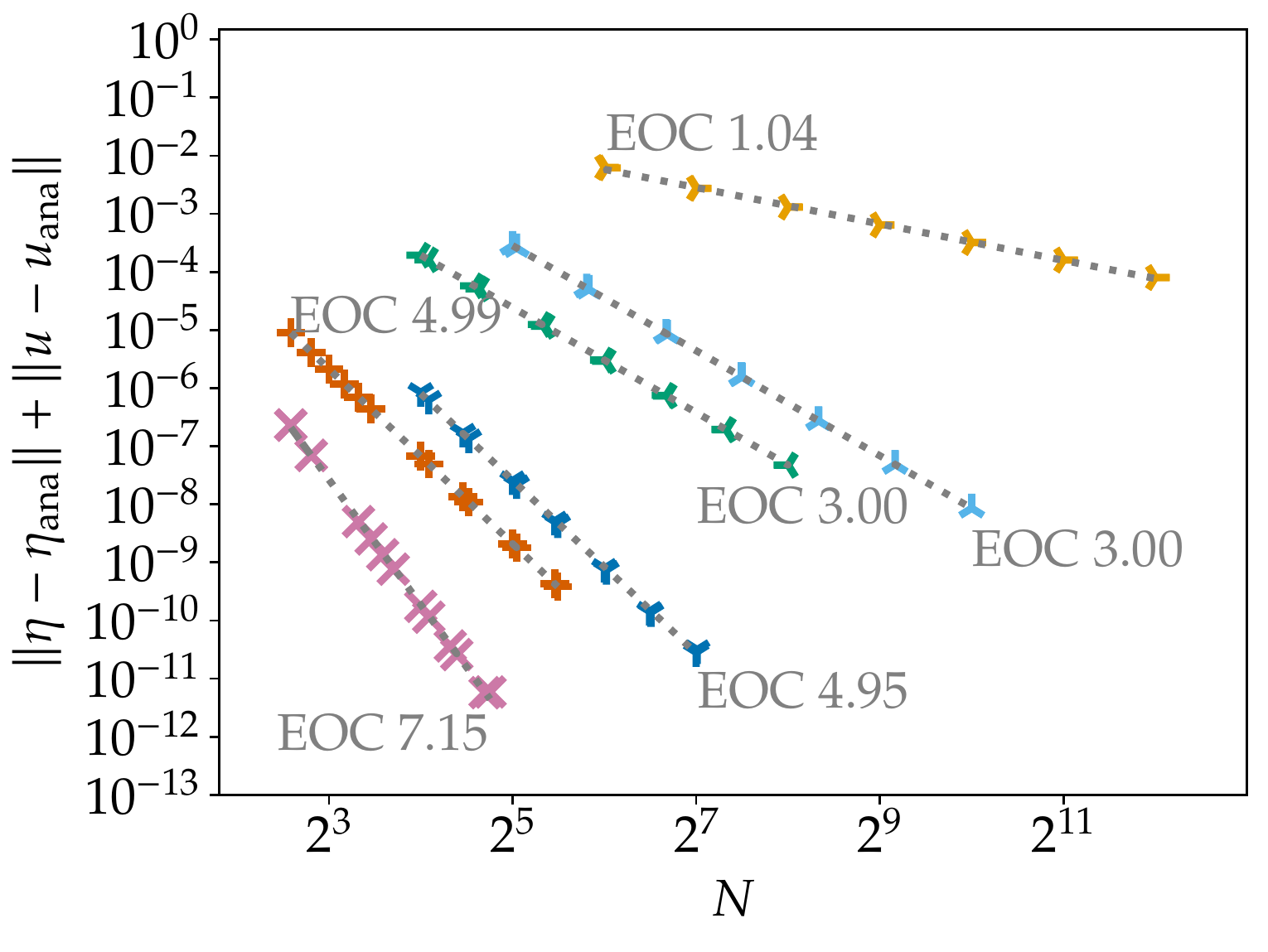}
    \caption{Discontinuous Galerkin methods, $\D2 = \D1^2$.}
  \end{subfigure}%
  \hspace*{\fill}
  \begin{subfigure}[t]{0.49\textwidth}
  \centering
    \includegraphics[width=\textwidth]{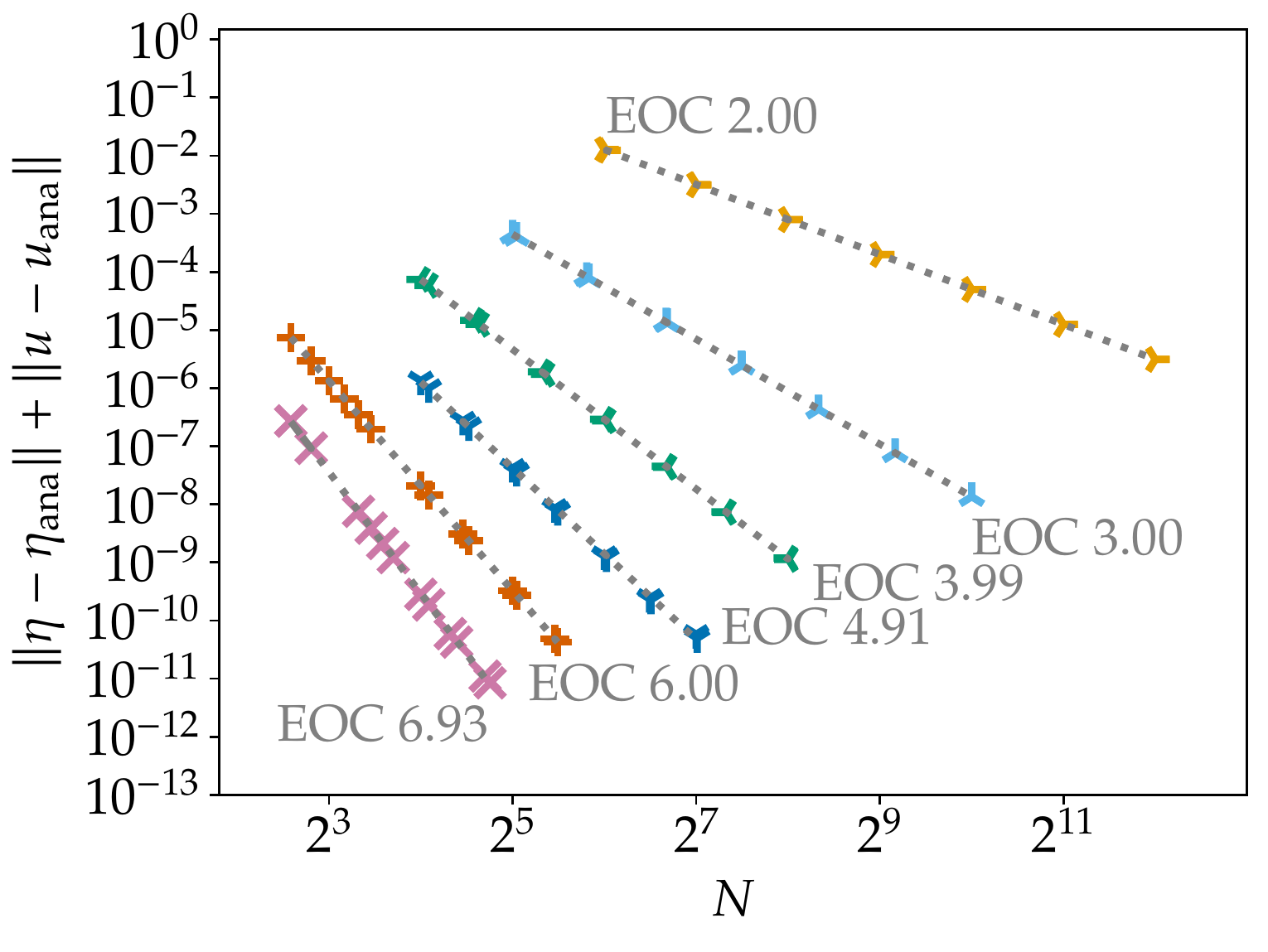}
    \caption{Discontinuous Galerkin methods, $\D2 = \D1{+} \D1{-}$.}
  \end{subfigure}%
  \caption{Convergence results of the spatial semidiscretizations
           \eqref{eq:bbm_bbm-inv-periodic-SBP}
           for the manufactured solution \eqref{eq:bbm_bbm-periodic-manufactured}
           of the BBM-BBM system.
           All of these semidiscretizations conserve the linear invariants
           \eqref{eq:bbm_bbm-invariants} of the BBM-BBM system
           \eqref{eq:bbm_bbm-dir-periodic}.
           The FD methods and the Galerkin methods with wide stencil $\D2$
           conserve the quadratic invariant as well.}
  \label{fig:bbm_bbm-periodic-manufactured-convergence}
\end{figure}

\subsubsection*{Reflecting boundary conditions}

A similar procedure is used for reflecting boundary conditions.
The method of manufactured solutions is applied to
\begin{equation}
\label{eq:bbm_bbm-reflecting-manufactured}
  \eta(t,x) = \e^{2 t} \cos(\pi x),
  \qquad
  u(t,x) = \e^{t} x \sin(\pi x),
\end{equation}
in the domain $[0, 1]$ for $t \in [0, 1]$. Results of a convergence study
for the semidiscretization \eqref{eq:bbm_bbm-dir-reflecting-SBP}
are shown in Figure~\ref{fig:bbm_bbm-reflecting-convergence}.

\begin{figure}[htbp]
\centering
  \begin{subfigure}{0.49\textwidth}
  \centering
    \includegraphics[width=\textwidth]{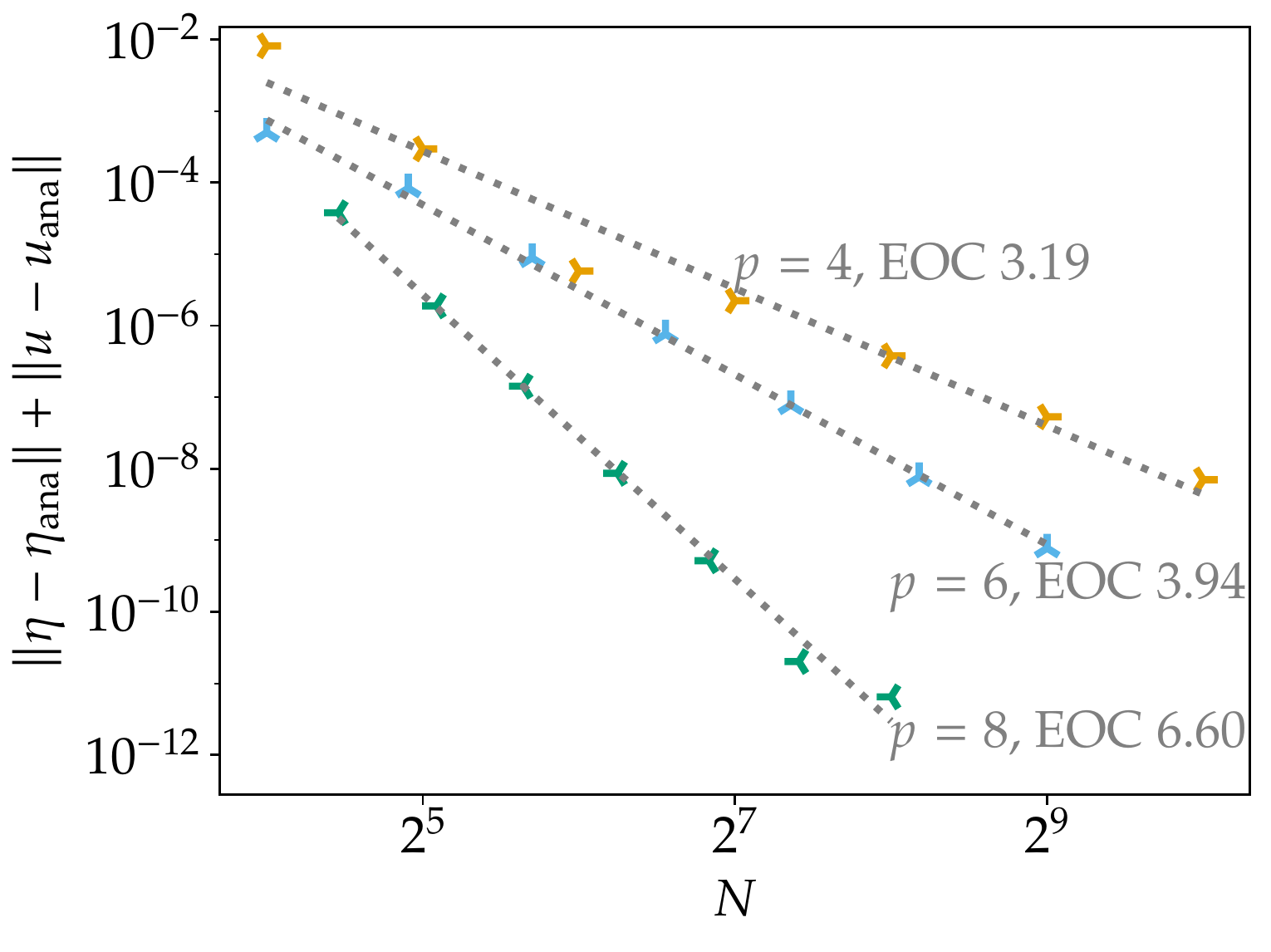}
    \caption{Accurate FD operators of \cite{mattsson2018boundary}, even $N$.}
  \end{subfigure}%
  \hspace*{\fill}
  \begin{subfigure}{0.49\textwidth}
  \centering
    \includegraphics[width=\textwidth]{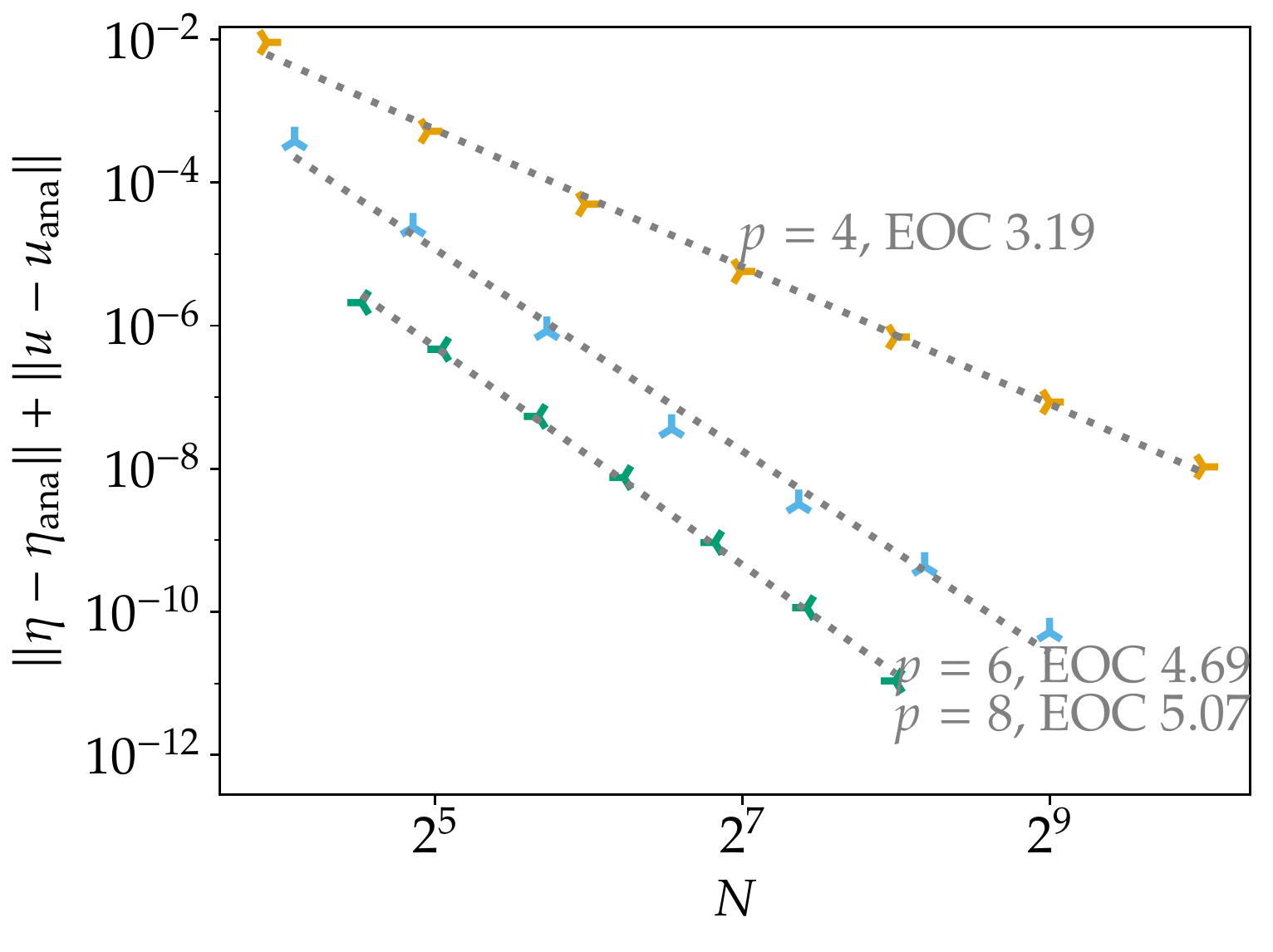}
    \caption{Accurate FD operators of \cite{mattsson2018boundary}, odd $N$.}
  \end{subfigure}%
  \\
  \begin{subfigure}{0.8\textwidth}
  \centering
    \includegraphics[width=\textwidth]{figures/Galerkin_legend_p1_p6}
  \end{subfigure}%
  \\
  \begin{subfigure}{0.49\textwidth}
  \centering
    \includegraphics[width=\textwidth]{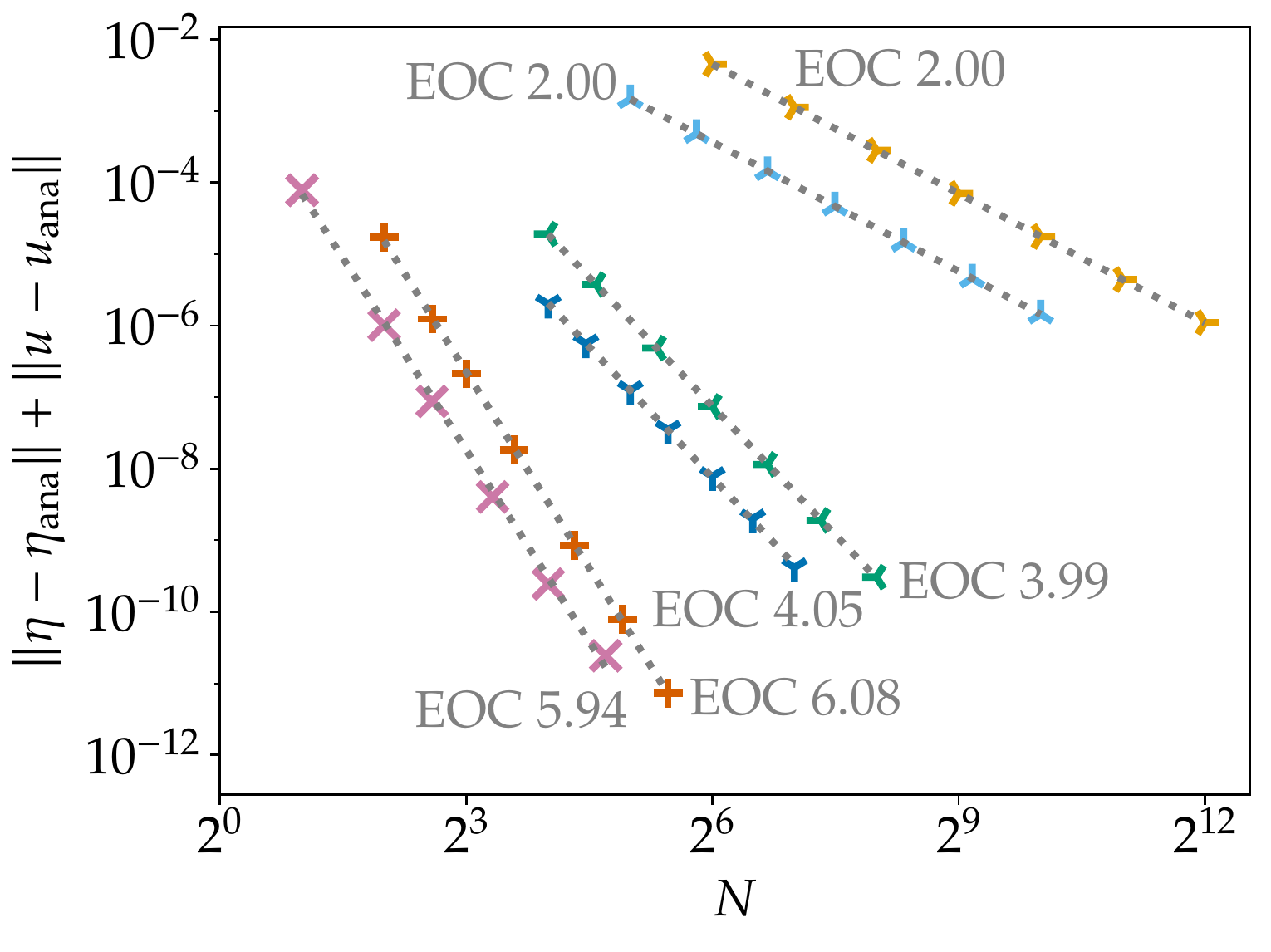}
    \caption{Continuous Galerkin methods, even $N$.}
  \end{subfigure}%
  \hspace*{\fill}
  \begin{subfigure}{0.49\textwidth}
  \centering
    \includegraphics[width=\textwidth]{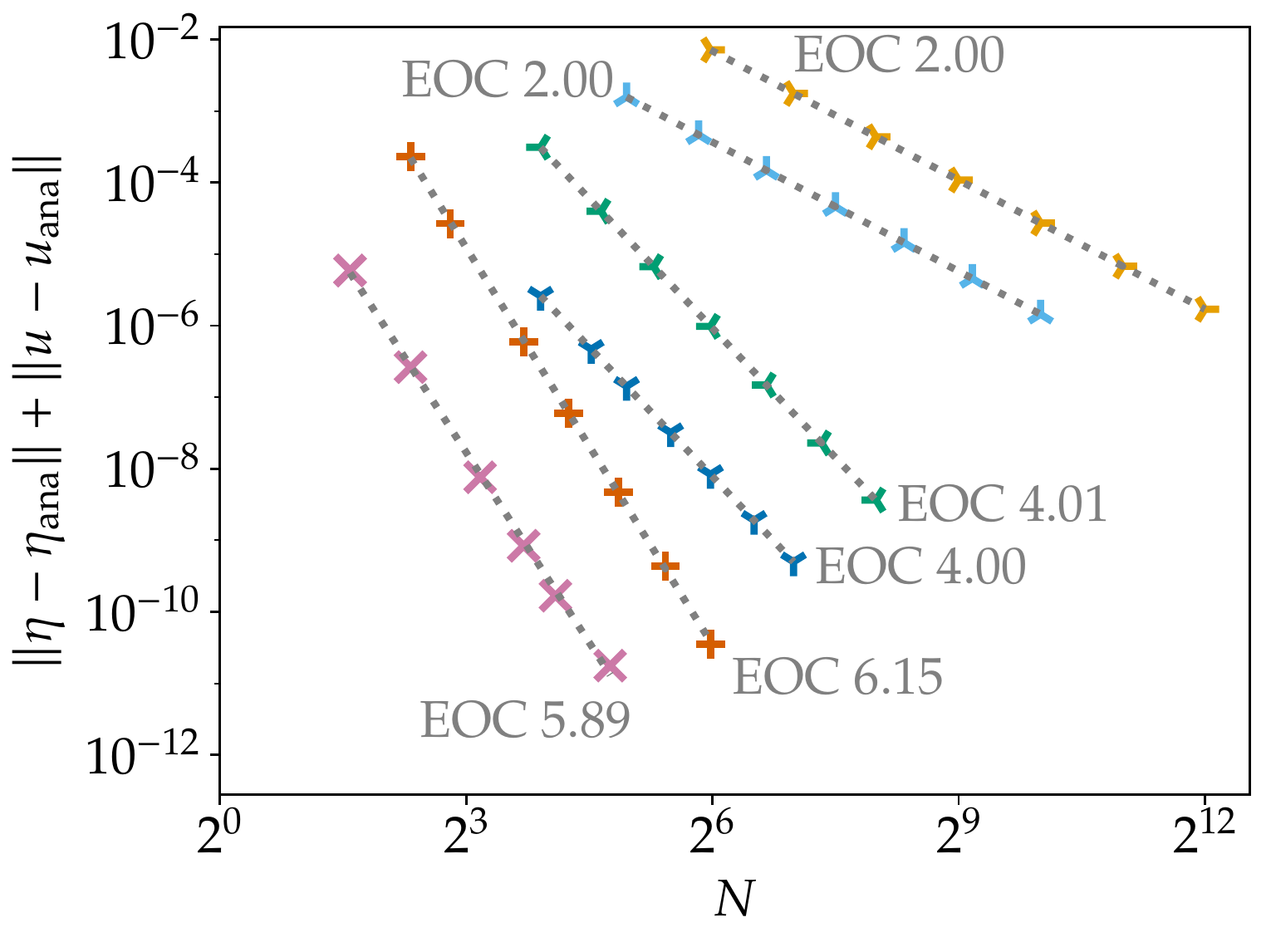}
    \caption{Continuous Galerkin methods, odd $N$.}
  \end{subfigure}%
  \\
  \begin{subfigure}{0.49\textwidth}
  \centering
    \includegraphics[width=\textwidth]{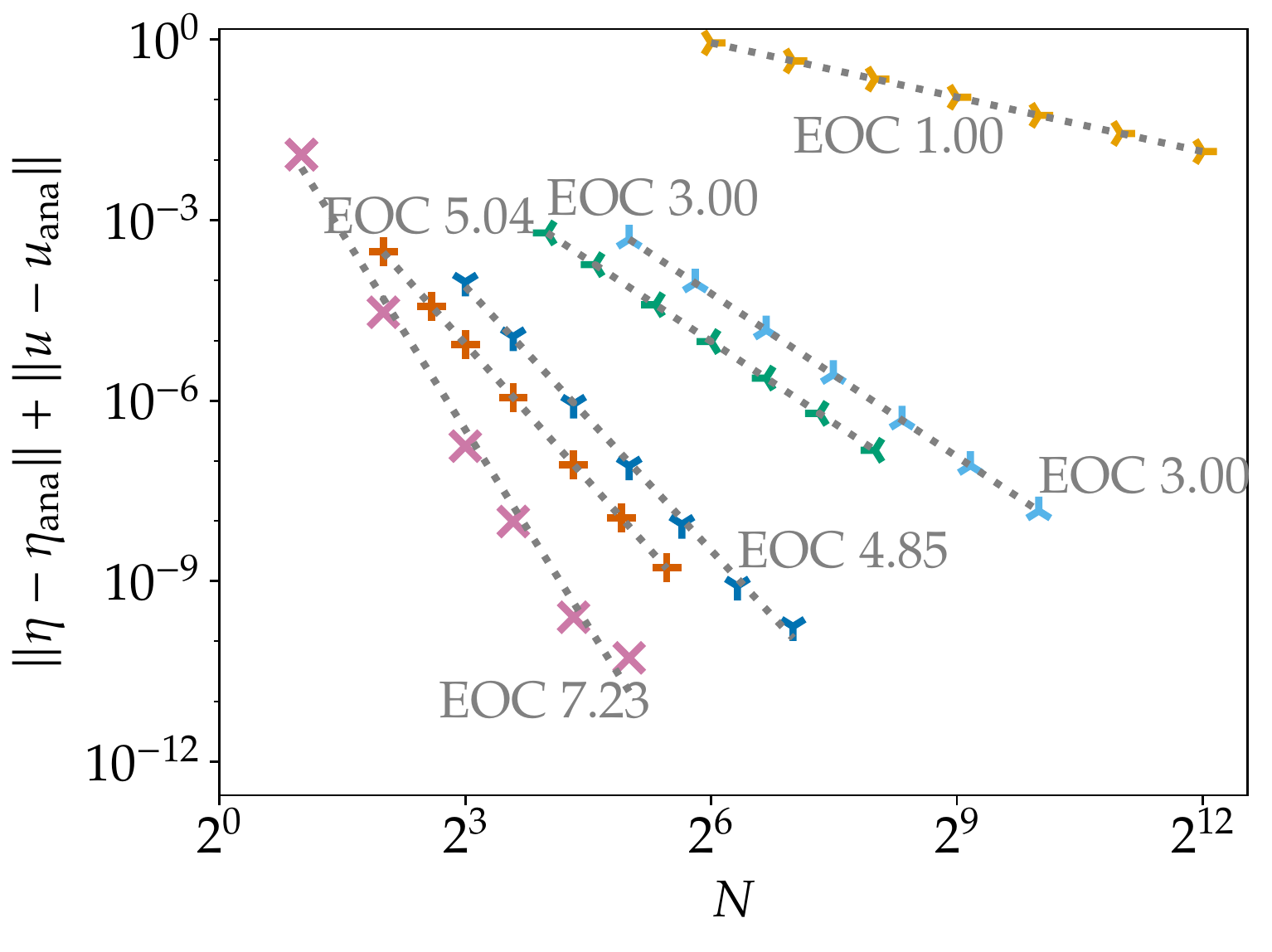}
    \caption{Discontinuous Galerkin methods, even $N$.}
  \end{subfigure}%
  \hspace*{\fill}
  \begin{subfigure}{0.49\textwidth}
  \centering
    \includegraphics[width=\textwidth]{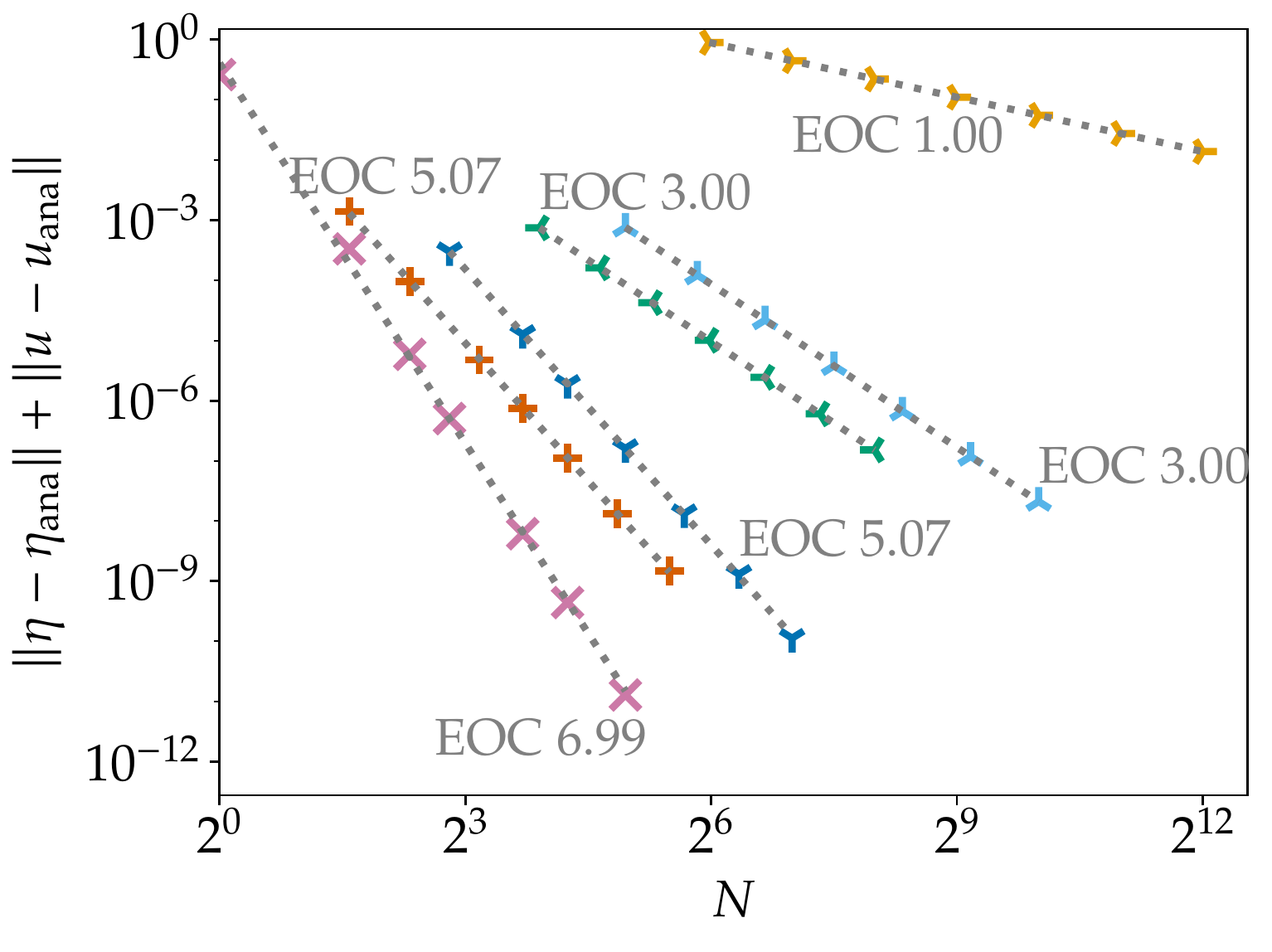}
    \caption{Discontinuous Galerkin methods, odd $N$.}
  \end{subfigure}%
  \caption{Convergence results of the spatial semidiscretizations
           \eqref{eq:bbm_bbm-dir-reflecting-SBP} for the manufactured
           solution \eqref{eq:bbm_bbm-reflecting-manufactured}
           of the BBM-BBM equation \eqref{eq:bbm_bbm-dir-reflecting}
           with reflecting boundary conditions.}
  \label{fig:bbm_bbm-reflecting-convergence}
\end{figure}

The FD methods can be expected to converge with an
$\text{EOC} \approx \nicefrac{p}{2} + 1$ because of the order of accuracy
$\nicefrac{p}{2}$ near the boundary.
However, there is a clear influence of the parity of the number of nodes $N$
for some operators:
For $p = 6$, odd $N$ yield more than an order of magnitude smaller errors and
a slightly bigger EOC. For $p = 8$, the behavior for even and odd $N$ is
the other way round. For $p = 4$, there is no significant influence.
In particular, it is noteworthy that using one node less or more can result in
an increase or decrease of the order of magnitude of the error.

For continuous Galerkin methods with wide-stencil operator $\D2 = \D1^2$,
$\text{EOC} \approx p+1$ for odd $p$ and $\text{EOC} \approx p$ for even $p$
as for periodic BCs.
There is a significant influence of the parity of the number of elements $N$
for the odd polynomial degrees $p \in \set{3, 5}$. An even number of elements
can reduce the error by an order of magnitude.

Finally, discontinuous Galerkin methods reproduce the $\text{EOC} \approx p+1$
for even $p$ and $\text{EOC} \approx p$ for odd $p$ as for periodic BCs and
wide-stencil second-derivative operators.
The influence of the parity of the number of elements is much less pronounced
than for CG methods.

\subsubsection{Convergence study for long-time simulations of traveling waves}
\label{sec:bbm_bbm-convergence-spacetime}

For long time simulations,
structure-preserving methods such as the conservative semidiscretizations
\eqref{eq:bbm_bbm-inv-periodic-SBP} with wide-stencil operators $\D2 = \D1^2$
coupled with relaxation methods in time can yield both qualitative and
quantitative improvements over standard methods. To demonstrate this, we consider
the traveling wave solutions
described in Remark~\ref{rem:bbm_bbm-traveling-wave}
with a final time $t = \num{7500}$, corresponding to $50$ periods.

In the following, CG methods with $p = 4$ and DG methods with $p = 3$ are used.
The sixth-order accurate Runge-Kutta method of \cite{verner2010numerically}
is used with a time step $\dt \propto \dx$ for a space-time convergence study.
The standard schemes use the narrow-stencil second-derivative operators and
the baseline time integration method. For the conservative schemes, the
wide-stencil second-derivative operators $\D2 = \D1^2$ are combined with
relaxation in time.

As can be seen in Figure~\ref{fig:bbm_bbm-convergence-spacetime}, the EOC of
the standard methods is higher but the absolute error of the conservative
methods is still lower for the applied range of parameters. Increasing the
resolution even further, it can be expected that the error will be smaller for
the standard methods. On the other hand, the conservative methods will probably
still result in a smaller error for even longer simulation times. Moreover,
they are more efficient at reasonably small error tolerances.

\begin{figure}[htbp]
\centering
  \begin{subfigure}{0.49\textwidth}
  \centering
    \includegraphics[width=\textwidth]{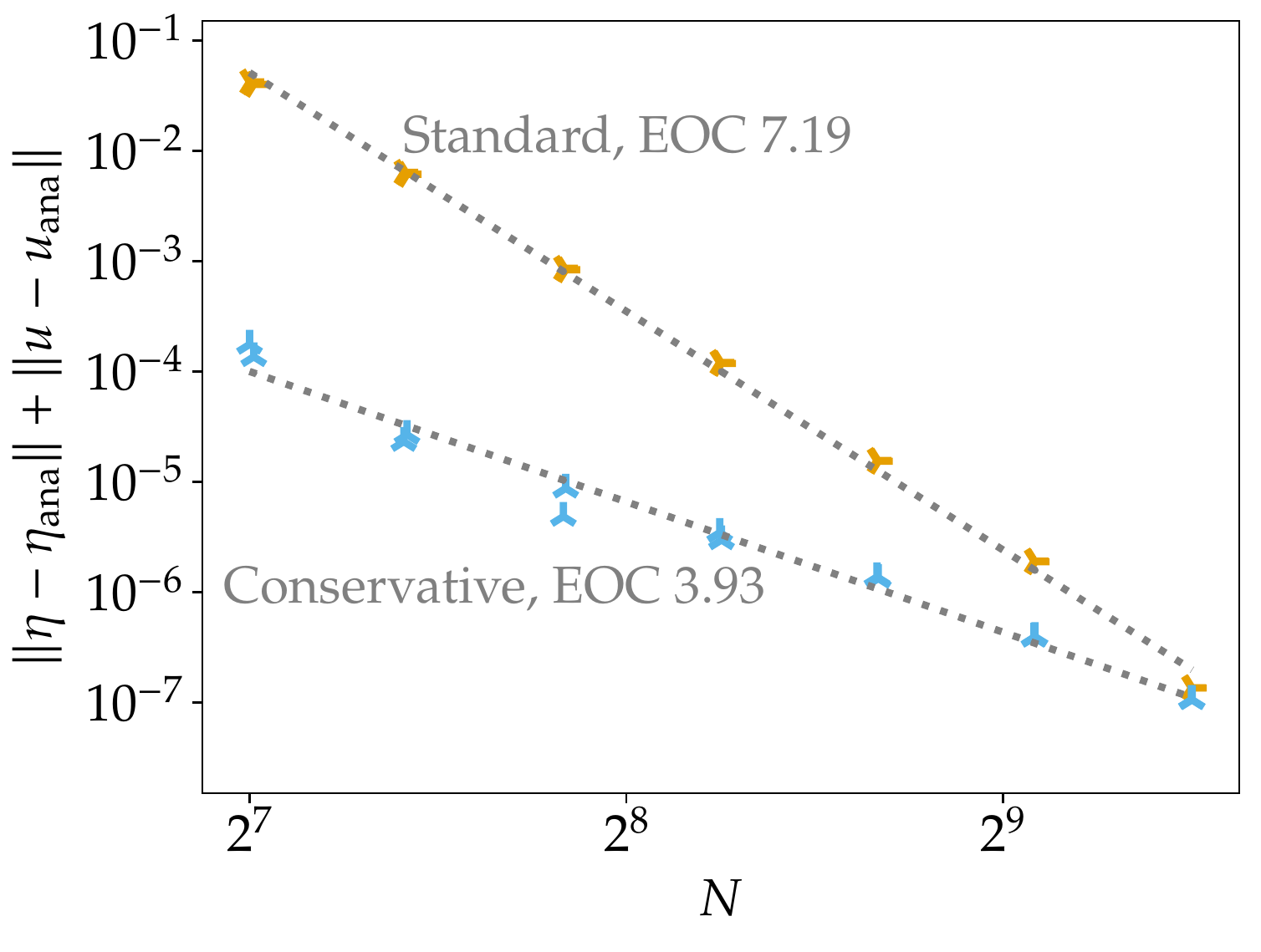}
    \caption{Continuous Galerkin methods, $p = 4$.}
  \end{subfigure}%
  \hspace*{\fill}
  \begin{subfigure}{0.49\textwidth}
  \centering
    \includegraphics[width=\textwidth]{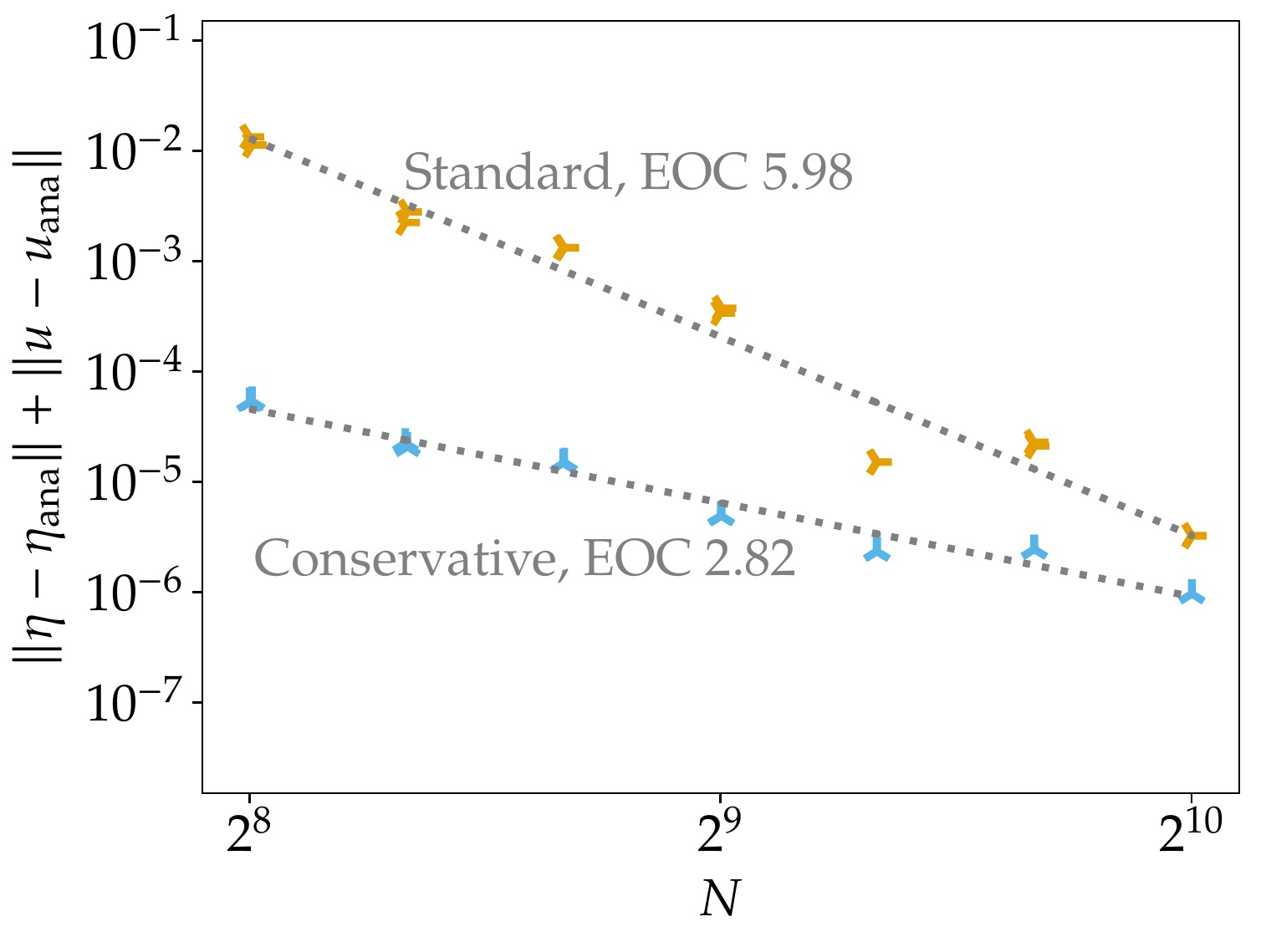}
    \caption{Discontinuous Galerkin methods, $p = 3$.}
  \end{subfigure}%
  \caption{Convergence results of the spatial semidiscretizations
           \eqref{eq:bbm_bbm-inv-periodic-SBP} for a long time simulation
           of a traveling wave solution of the BBM-BBM equation
           \eqref{eq:bbm_bbm-dir-periodic}.}
  \label{fig:bbm_bbm-convergence-spacetime}
\end{figure}

\subsubsection{Conservation of invariants for reflecting boundary conditions}
\label{sec:bbm_bbm-reflecting}

To test the method for reflecting boundary conditions, the traveling wave
initial condition described in Remark~\ref{rem:bbm_bbm-traveling-wave}
is used.
For SBP finite difference methods with interior order of accuracy $p = 6$,
results are visualized in Figure~\ref{fig:bbm_bbm-reflecting}. The classical
operators of \cite{mattsson2004summation} result in undesired oscillations of
small amplitude at the final time $t = 3050$ which vanish under grid refinement.
These are generated by the interaction of the wave with the boundary and
can also be reduced significantly by applying operators with improved accuracy
near the boundary, \eg the ones of \cite{mattsson2014optimal} or
\cite{mattsson2018boundary}.

Again, applying relaxation to conserve the energy improves the accuracy of
the fully-discrete methods. Choosing $\dt = 0.1$ for the accurate operators
of \cite{mattsson2018boundary}, the solutions with and without relaxation are
visually indistinguishable and further refinement causes no visible change.
Increasing the time step by a factor of ten, the solution of the baseline method is
clearly distorted with changes of the amplitude and phase while the relaxation
solution is barely affected.

\begin{figure}[htb]
\centering
  \begin{subfigure}{0.5\textwidth}
  \centering
    \includegraphics[width=\textwidth]{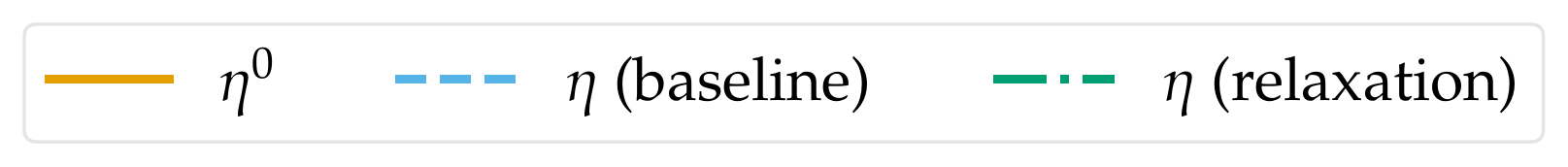}
  \end{subfigure}%
  \\
  \begin{subfigure}{0.33\textwidth}
  \centering
  \captionsetup{width=0.8\textwidth}
    \includegraphics[width=\textwidth]{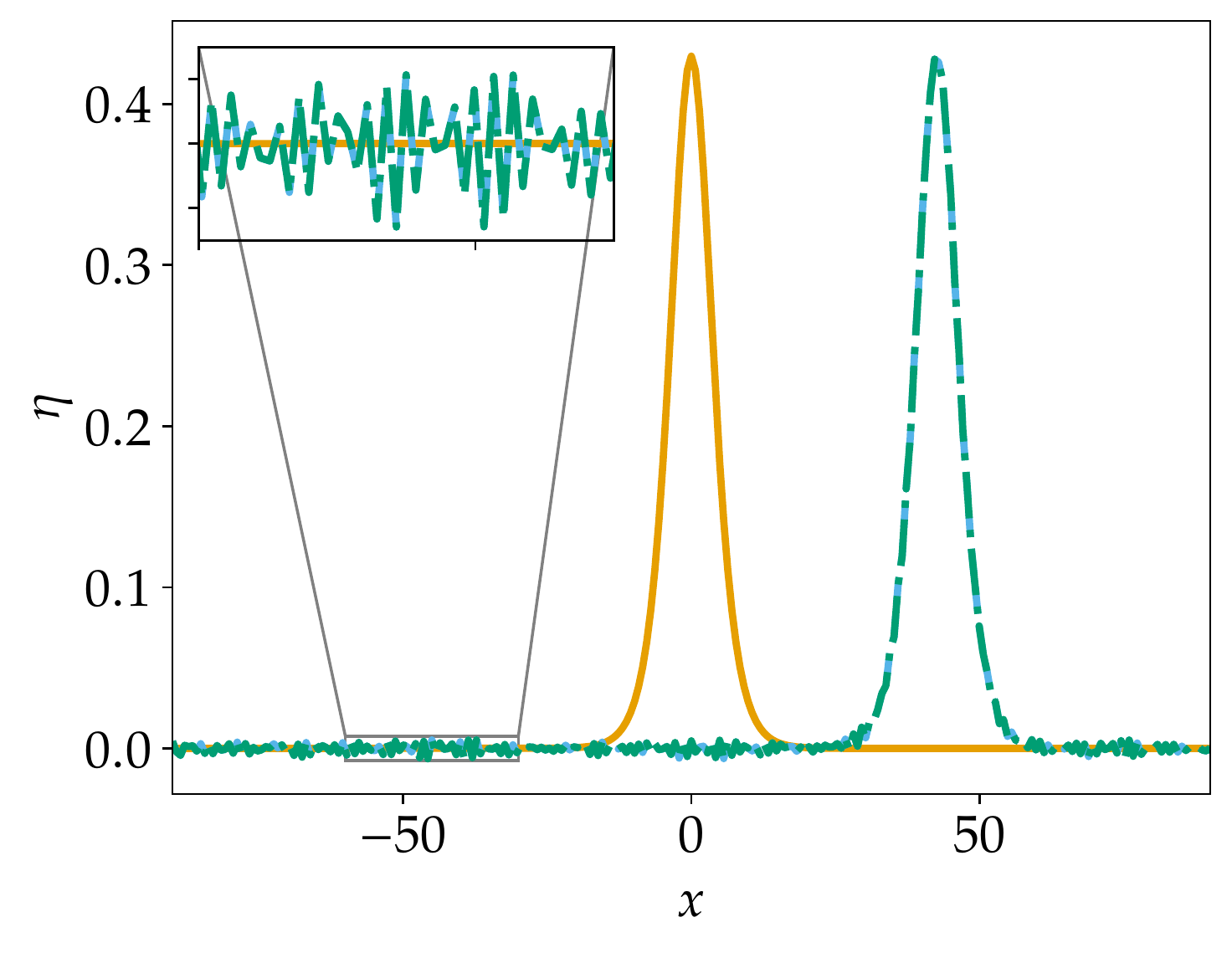}
    \caption{Classical operator of \cite{mattsson2004summation},
             $\dt = 0.1$.}
  \end{subfigure}%
  \hspace*{\fill}
  \begin{subfigure}{0.33\textwidth}
  \centering
  \captionsetup{width=0.8\textwidth}
    \includegraphics[width=\textwidth]{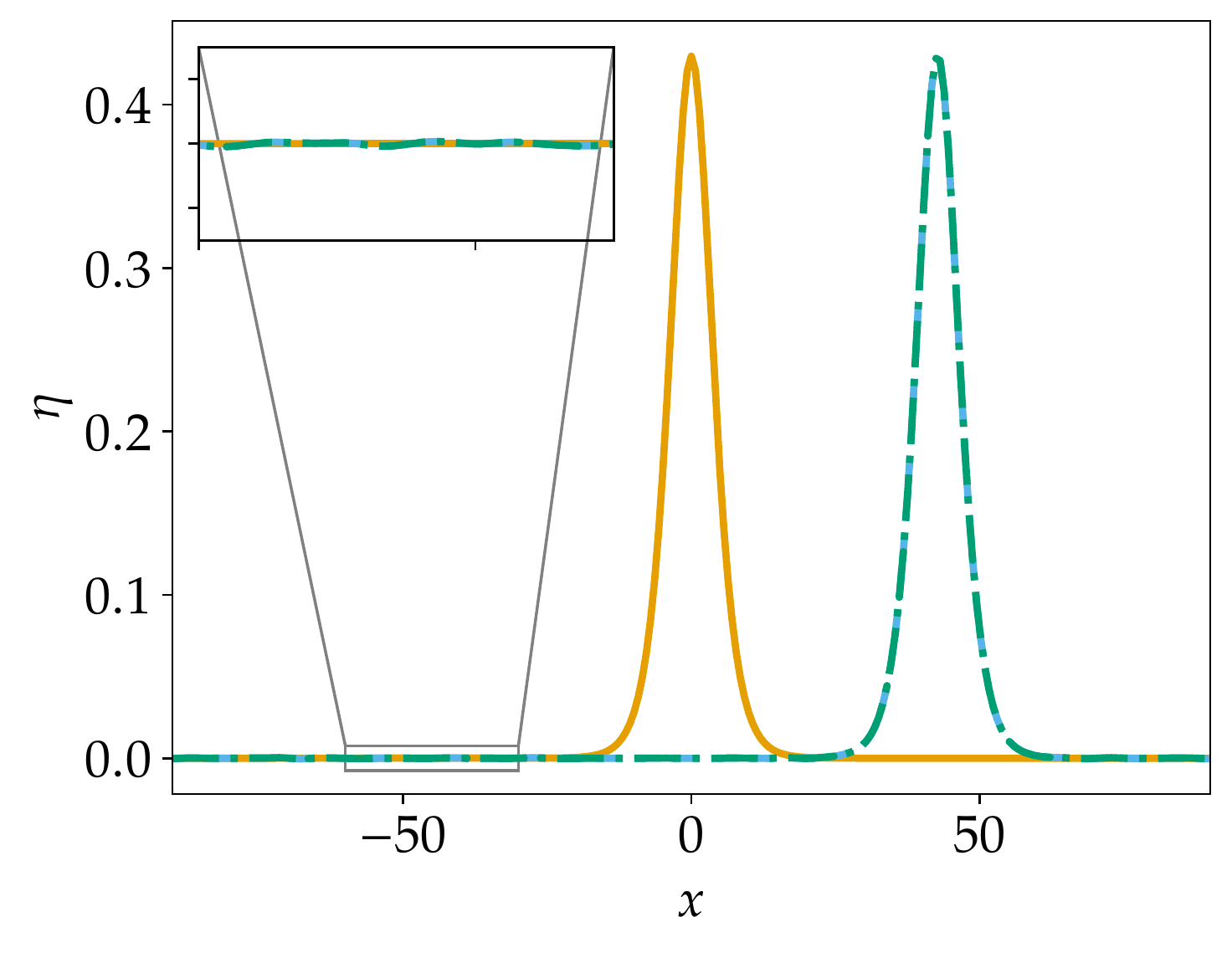}
    \caption{Accurate operator of \cite{mattsson2018boundary},
             $\dt = 0.1$.}
  \end{subfigure}%
  \hspace*{\fill}
  \begin{subfigure}{0.33\textwidth}
  \centering
  \captionsetup{width=0.8\textwidth}
    \includegraphics[width=\textwidth]{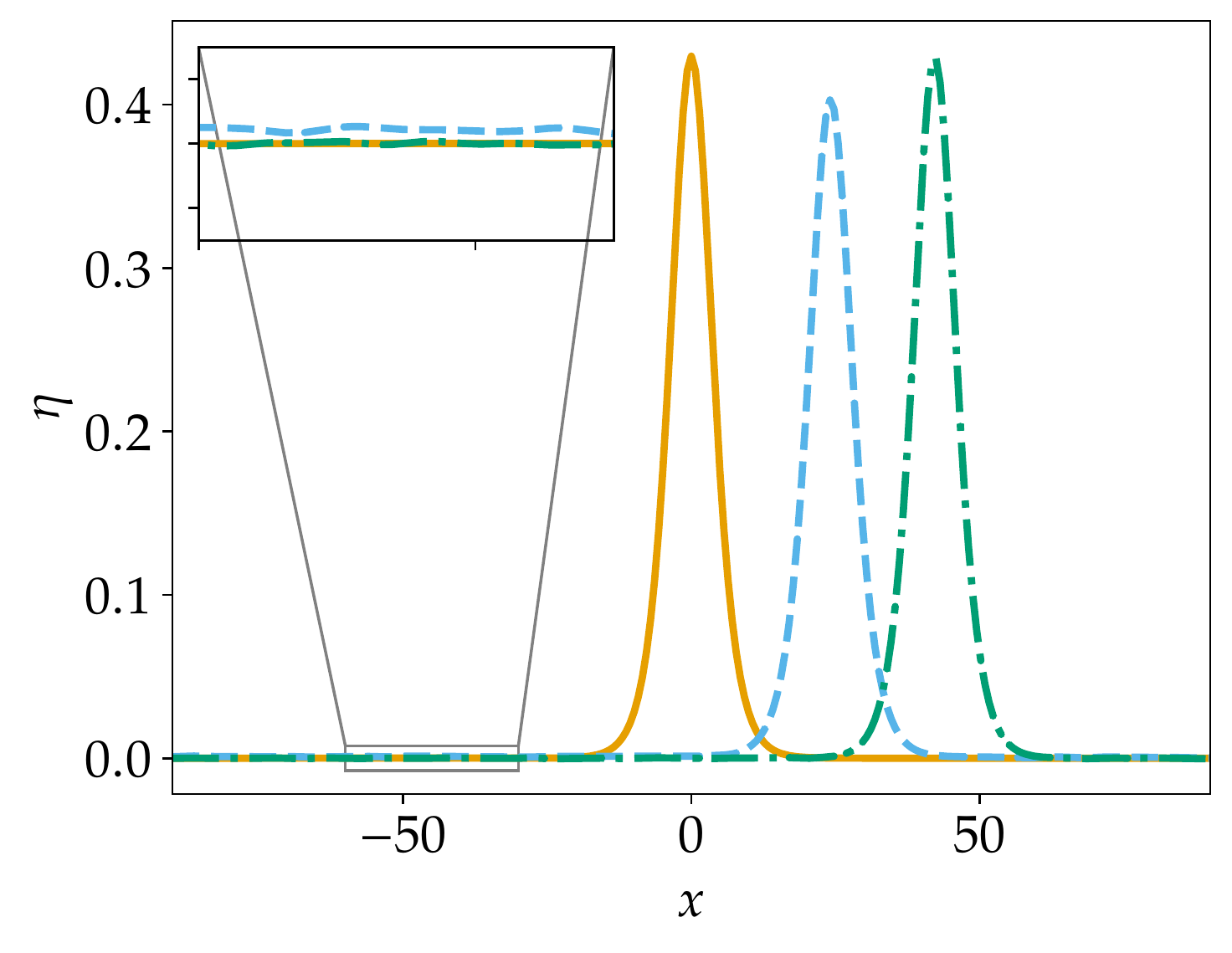}
    \caption{Accurate operator of \cite{mattsson2018boundary},
             $\dt = 1.0$.}
  \end{subfigure}%
  \\
  \begin{subfigure}{0.8\textwidth}
  \centering
    \includegraphics[width=\textwidth]{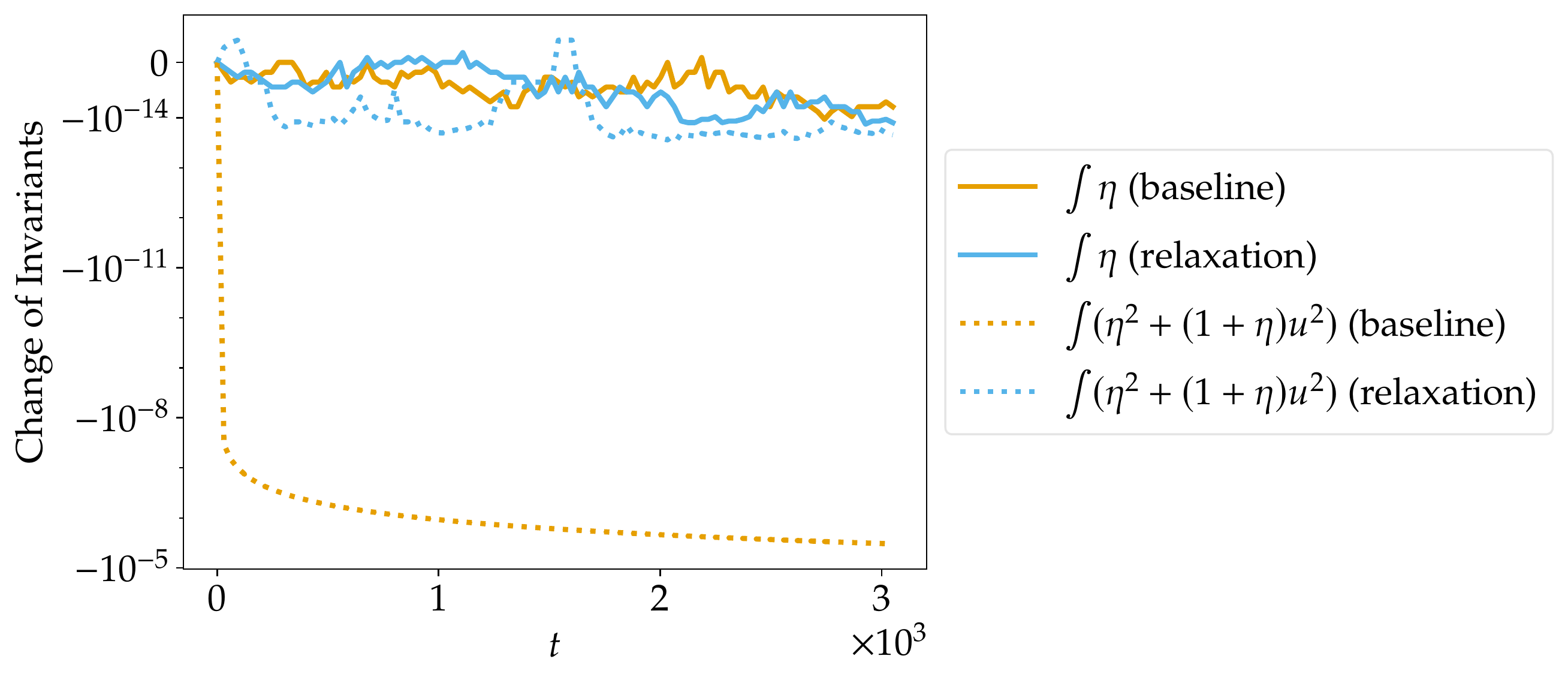}
    \caption{Change of the invariants \eqref{eq:bbm_bbm-invariants} for the
             accurate operator of \cite{mattsson2018boundary} with
             $\dt = 0.1$.}
  \end{subfigure}%
  \caption{Numerical solutions of the BBM-BBM system with reflecting boundary
           conditions \eqref{eq:bbm_bbm-dir-reflecting} obtained by finite
           difference methods and RK4 with and without relaxation to preserve
           the energy \eqref{eq:bbm_bbm-invariants-energy}.}
  \label{fig:bbm_bbm-reflecting}
\end{figure}

\section{Summary and conclusions}
\label{sec:summary}

We have further developed general tools to construct conservative methods
and applied them to a broad range of dispersive wave
equation models. These fully-discrete schemes combine summation by parts
operators in space with relaxation methods in time to conserve all linear
and one nonlinear invariant for each model. Because of the generality of the
SBP framework, the conservation properties of four different classes
of schemes, namely Fourier collocation, finite difference, continuous Galerkin,
and discontinuous Galerkin methods, can be analyzed simultaneously for six
different dispersive wave models studied in this article. Instead of
requiring $4 \cdot 6 = 24$ ad hoc approaches, we have established general results
first, allowing a unified analysis and relatively simple proofs of conservation
for all methods.  The proposed schemes do not require exact integration
(as long as the mass matrix is diagonal) and use time discretizations that
are explicit except for the solution of a scalar equation at each step.

While the application of the relaxation approach is straightforward given
the established results, the construction of conservative spatial
semidiscretizations requires some tuning for each model. Nevertheless,
the only required techniques are the application of split forms and the
special choice of higher-derivative operators. In bounded domains, the
carefully developed imposition of boundary conditions is also crucial.

Having developed a broad framework of conservative numerical methods for
dispersive wave equations, the present work will be extended in the future.
While the analysis of conservation properties for different classes of
methods can be conducted in the unifying SBP framework, the detailed study
of error estimates still seems to require specializations on the schemes
and does not fit into this manuscript.  Estimates of the order
of convergence resulting from numerical experiments are summarized in
Table~\ref{tab:summary}.

As alluded to in the introduction, conservative fully-discrete numerical
methods can have improved properties concerning the error growth in time.
This is related to the results shown in Section~\ref{sec:bbm_bbm-convergence-spacetime},
where standard numerical methods can have a higher order of convergence but
result still in bigger errors than conservative numerical methods. We will
focus on these aspects in the future.

\begin{table}[!htb]
\centering
  \caption{Summary of the experimental order of convergence (EOC)
           and conservation properties (Cons.) for various semidiscretizations
           of the dispersive wave equations considered in this manuscript.
           The methods conserve either the linear invariant(s) only (lin.),
           a chosen nonlinear invariant only (nonl.), or both.
           For FD methods, $p$ is the interior order of accuracy.
           For CG and DG methods, $p$ is the polynomial degree.
           If no EOC is given, the results of numerical experiments were
           not clear enough.}
  \label{tab:summary}
  \begin{minipage}{\textwidth}
  \begin{tabular*}{\linewidth}{@{\extracolsep{\fill}}cccclclcl@{}}
    \toprule
    \multicolumn{3}{c}{Method}
      & \multicolumn{2}{c}{BBM \eqref{eq:bbm-inv-periodic-SBP}}
      & \multicolumn{2}{c}{FW \eqref{eq:fw-inv-periodic-SBP}}
      & \multicolumn{2}{c}{CH \eqref{eq:ch-inv-periodic-SBP}, $\alpha = \nicefrac{1}{2}$} \\
    Class & $\D2$ stencil & $p$ & Cons. & EOC & Cons. & EOC & Cons. & EOC \\
    \midrule
      FD &                &          & both  & $\approx p$
                                     & both  & $\in [p-\nicefrac{1}{2}, p]$
                                     & both  & $\approx p$ \\
      \cmidrule{2-9}
      CG &  wide          &  odd $p$ & both  & $\approx p+1$
                                     & both  & $\in [p+\nicefrac{1}{2}, p+1]$
                                     & both  & $\approx p+1$ \\
         &                & even $p$ &       & $\approx p$
                                     &       & $\approx p+1$
                                     &       & $\approx p$ \\
         & narrow         &  odd $p$ & both  & $\approx p+2$\textsuperscript{a}
                                     & lin.  & $\in [p+\nicefrac{1}{2}, p+1]$
                                     & both  & $\approx p+1$ \\
         &                & even $p$ &       & $\approx p+2$
                                     &       & $\approx p+1$
                                     &       & $\approx p$ \\
      \cmidrule{2-9}
      DG &  wide          &  odd $p$ & both  & $\approx p$
                                     & both  & $\approx p$
                                     & both  & $\approx p$ \\
         &                & even $p$ &       & $\approx p+1$
                                     &       & $\approx p+1$
                                     &       & $\approx p+1$ \\
         & narrow         &  odd $p$ & both  & $\approx p+1$
                                     & lin.  & $\approx p$
                                     & both  & $\approx p$\textsuperscript{b} \\
         &                & even $p$ &       & $\approx p+1$
                                     &       & $\approx p+1$
                                     &       & $\approx p+1$ \\
    \bottomrule
  \end{tabular*}
  \begin{tabular*}{\linewidth}{@{\extracolsep{\fill}}cccclclcl@{}}
    \toprule
    \multicolumn{3}{c}{Method}
      & \multicolumn{2}{c}{DP \eqref{eq:dp-inv-periodic-SBP}}
      & \multicolumn{2}{c}{HH \eqref{eq:hh-inv-periodic-SBP}}
      & \multicolumn{2}{c}{BBM-BBM \eqref{eq:bbm_bbm-inv-periodic-SBP}} \\
    Class & $\D2$ stencil & $p$ & Cons. & EOC & Cons. & EOC & Cons. & EOC \\
    \midrule
      FD &                &          & both  & $\in [p-\nicefrac{1}{2}, p]$
                                     & both  & $\approx p$
                                     & both  & $\approx p$ \\
      \cmidrule{2-9}
      CG &  wide  &  odd $p$ & both  & $\in [p+\nicefrac{1}{2}, p+1]$
                                     & both  & $\approx p+1$
                                     & both  & $\approx p+1$ \\
         &                & even $p$ &       & $\approx p$
                                     &       & $\approx p$
                                     &       & $\approx p$ \\
         & narrow &  odd $p$ & both  & $\in [p+\nicefrac{1}{2}, p+1]$
                                     & nonl. &
                                     & lin.  & $\approx p+2$\textsuperscript{a} \\
         &                & even $p$ &       & $\approx p$
                                     &       &
                                     &       & $\approx p+2$ \\
      \cmidrule{2-9}
      DG &  wide  &  odd $p$ & both  & $\approx p$
                                     & both  &
                                     & both  & $\approx p$ \\
         &                & even $p$ &       & $\approx p+1$
                                     &       &
                                     &       & $\approx p+1$ \\
         & narrow &  odd $p$ & both  & $\approx p$
                                     & nonl. &
                                     & lin.  & $\approx p+1$ \\
         &                & even $p$ &       & $\approx p+1$
                                     &       &
                                     &       & $\approx p+1$ \\
    \bottomrule
  \end{tabular*}
  \end{minipage}
\begin{flushleft}
  \small\textsuperscript{a}$p+1=2$ for $p=1$.

  \small\textsuperscript{b}For $p = 1$ and $\alpha = \nicefrac{1}{2}$, this DG method does not converge for the manufactured solution. However, it converges for other $p$, other $\alpha$ such as $\alpha = 1$, and a traveling wave solution.
\end{flushleft}
\end{table}

\appendix
\section*{Acknowledgments}

Research reported in this publication was supported by the
King Abdullah University of Science and Technology (KAUST).
DM expresses his gratitude to KAUST for the hospitality and financial
support especially during his last visit where he met DK and HR.

\printbibliography

\end{document}